\documentclass{amsart}

\usepackage{amsmath,amssymb,amsfonts,amscd}
\usepackage{latexsym}
\usepackage{epic,eepic,epsfig,psfrag}
\usepackage{amscd}
\usepackage{hyperref}
\usepackage{color}
\usepackage{makecell}

\newcommand\ul{\underline}
\newcommand\ov{\overline}
\newcommand\ti{\tilde}
\newcommand\Ti{\widetilde}
\def\dbar{{\overline\partial}}
\def\rd{{\rm d}}
\def\rD{{\rm D}}
\def\rT{{\rm T}}

\def\k{\kappa}

\def\g{\gamma}

\def\s{\sigma}
\def\t{\tau}
\def\i{\iota}

\def\o{\omega}
\def\l{\lambda}
\def\L{\Lambda}
\def\G{\Gamma}
\def\S{\mathhexbox278}

\def\Om{\Omega}

\def\cA{{\mathcal A}}
\def\cB{{\mathcal B}}
\def\cC{{\mathcal C}}

\def\cE{{\mathcal E}}
\def\cF{{\mathcal F}}

\def\cI{{\mathcal I}}

\def\cM{{\mathcal M}}

\def\bM{{\overline{\mathcal M}}}
\def\bX{{{\mathbf X}}}
\def\bW{{{\mathbf W}}}

\def\cO{{\mathcal O}}
\def\cP{{\mathcal P}}
\def\cR{{\mathcal R}}
\def\cS{{\mathcal S}}
\def\cT{{\mathcal T}}
\def\cU{{\mathcal U}}
\def\cV{{\mathcal V}}
\def\cW{{\mathcal W}}
\def\cX{{\mathcal X}}

\def\E{{\mathbb E}}

\def\R{{\mathbb R}}
\def\Q{{\mathbb Q}}

\def\N{{\mathbb N}}
\def\C{{\mathbb C}}
\def\D{{\mathbb D}}
\def\CP{{\mathbb C\mathbb P}}

\def\Z{{\mathbb Z}}

\def\half{{\textstyle{\frac 12}}}
\def\im{{\rm im}\,}

\def\ev{{\rm ev}}

\DeclareMathOperator{\supp}{supp}

\def\la{\langle\,}
\def\ra{\,\rangle}
\def\dvol{\,\rd{\rm vol}}

\def\ev{{\rm ev}}

\def\pr{{\rm pr}}
\def\id{{\rm id}}
\def\Crit{{\rm Crit}}

\newcommand{\less}{{\smallsetminus}}

\newcommand{\leftsub}[2]{{\vphantom{#2}}_{#1}{#2}}

\newtheorem{dfn}{Definition}[section]
\newtheorem{lem}[dfn]{Lemma}

\newtheorem{thm}[dfn]{Theorem}
\newtheorem{rmk}[dfn]{Remark}
\newtheorem{cor}[dfn]{Corollary}

\newtheorem{ass}[dfn]{Assumption}

   \newcounter{qcounter}

\newenvironment{enumilist}
   { \begin{list} {(\roman{qcounter})\;}{\usecounter{qcounter}
     \setlength{\itemsep}{.5ex} \setlength{\leftmargin}{1ex} } }
   { \end{list} }

\newenvironment{itemlist}
   { \begin{list} {$\bullet$}
         { \setlength{\topsep}{.5ex}  \setlength{\itemsep}{.5ex} \setlength{\leftmargin}{2.5ex} } }
   { \end{list} }

\newcommand\quotient[2]{
        \mathchoice
            {
                \text{\raise1ex\hbox{$#1$}\Big/\lower1ex\hbox{$#2$}}%
            }
            {
                #1\,/\,#2
            }
            {
                #1\,/\,#2
            }
            {
                #1\,/\,#2
            }
    }

\newcommand\quo[2]{
                \text{\raise1ex\hbox{$#1\!$}\big/\lower1ex\hbox{$\!#2$}}
  }

\newcommand\qu[2]{
                \text{\raise.8ex\hbox{$\scriptstyle#1$}/\lower.8ex\hbox{$\scriptstyle#2$}}
  }

\def\qqquad{\qquad\qquad}

\begin{document}

\author{Benjamin Filippenko, Katrin Wehrheim}

\title{A polyfold proof of the Arnold conjecture}





\begin{abstract}
We give a detailed proof of the homological Arnold conjecture for nondegenerate periodic Hamiltonians on general closed symplectic manifolds $M$ via a direct Piunikhin-Salamon-Schwarz morphism. Our constructions are based on a coherent polyfold description for moduli spaces of pseudoholomorphic curves in a family of symplectic manifolds degenerating from $\CP^1\times M$ to $\C^+ \times M$ and $\C^-\times M$, 
as developed by Fish-Hofer-Wysocki-Zehnder as part of the Symplectic Field Theory package.
To make the paper self-contained we include all polyfold assumptions, describe the coherent perturbation iteration in detail, and prove an abstract regularization theorem for moduli spaces with evaluation maps relative to a countable collection of submanifolds. 

The 2011 sketch of this proof was joint work with Peter Albers, Joel Fish.
\end{abstract}

\maketitle

\section{Introduction}

Let $(M,\omega)$ be a closed symplectic manifold and $H:S^1\times M \to \R$ a periodic Hamiltonian function. It induces a time-dependent Hamiltonian vector field $X_H: S^1\times M\to \rT M$ given by $\omega(X_H(t,x), \cdot)=\rd H(t, \cdot)$.  We denote the set of contractible periodic orbits by
\begin{equation}\label{eq:PH}
\cP(H) := \bigl\{ \g:S^1 \to M \,\big|\, \dot\g(t) = X_H(t,\g(t)) {\rm \,\, and \,\,} \g {\rm \,\, is \,\, contractible} \bigr\}
\end{equation}
and note that periodic orbits can be identified with the fixed points of the time 
$2\pi$ flow $\phi^{2\pi}_H :M\to M$ of $X_H$.
(Here we choose the convention $S^1=\R/2\pi\Z$, i.e.\ period $2\pi$, for ease of notation later on.)
We call this Hamiltonian system nondegenerate if
$\phi^{2\pi}_H \times \id_M$ is transverse to the diagonal and hence cuts out the fixed points transversely. In particular, this guarantees a finite set of periodic orbits.
Arnold \cite{arnold} conjectured in the 1960s that the minimal number of critical points of a Morse function on $M$ is also a lower bound for the number of periodic orbits of a nondegenerate Hamiltonian system as above.
In this strict form, the Arnold conjecture has been confirmed for Riemann surfaces \cite{eliashberg-arnold} and tori \cite{CZ}.
A weaker form is accessible by Floer theory, introduced by Floer \cite{f2,f3} in the 1980s. It constructs a chain complex generated by $\cP(H)$ that can be compared with the Morse complex generated by the critical points of a Morse function.
When Floer homology is well-defined, it is usually independent of the Hamiltonian, and on a compact symplectic manifold can in fact be identified with Morse homology, which is also independent of the Morse function and computes the singular homology.
Using this approach, the following nondegenerate homological form of the Arnold conjecture was first proven by Floer \cite{f1,f4} in the absence of pseudoholomorphic spheres.

\begin{thm} \label{thm:arnold}
Let $(M,\omega)$ be a closed symplectic manifold and $H:S^1\times M \to \R$ a nondegenerate periodic Hamiltonian function. Then
$$ \textstyle
\# \cP(H) \;\geq\;  \sum_{i=0}^{\dim M} \dim H_i(M;\Q).
$$
\end{thm}

Floer's proof was later extended to general closed symplectic manifolds \cite{HS,Ono,FO,LT}, and in the presence of pseudoholomorphic spheres of negative Chern number requires abstract regularizations of the moduli spaces of Floer trajectories since perturbations of the geometric structures may not yield regular moduli spaces; see e.g.\ \cite{MWsmooth}.
Further generalizations and alternative proofs have been published in the meantime, using a variety of regularization methods. 
The purpose of this note is to provide a general and maximally accessible proof of Theorem~\ref{thm:arnold} -- using an abstract perturbation scheme provided by the polyfold theory of Hofer-Wysocki-Zehnder \cite{HWZbook}, following an approach by Piunikhin-Salamon-Schwarz \cite{PSS} based on \cite{schwarz-thesis}, and building on polyfold descriptions of Gromov-Witten moduli spaces \cite{hwz-gw} as well as their degenerations in Symplectic Field Theory \cite{egh,fh-sft-full}.

%

\begin{rmk} \label{rmk:disclosure} \rm  
Since the polyfold descriptions of SFT moduli spaces \cite{fh-sft}--\cite{fh-sft-full} are not completely published, we formulate them as Assumptions~\ref{ass:pss}, \ref{ass:iota,h}, \ref{ass:iso}.
While these descriptions of four kinds of moduli spaces and their relations involve a lot of structures (bundles, sections, evaluation maps, and compatible immersions from Cartesian products to boundaries), they will be familiar from classical descriptions of moduli spaces of pseudoholomorphic curves. 
Our assumptions in polyfold theoretic terms formalize the well known fact that the moduli spaces have local descriptions in terms of Fredholm sections and gluing theorems, which polyfold theory interprets as global smooth structure within an appropriately generalized differential geometry.
Indeed, transition maps between the natural infinite dimensional local models fail to be classically differentiable for only two reasons which polyfold theory resolves as explained in e.g.\ \cite[\S2]{usersguide} and \cite[\S2.1]{hwz-gw}: 
Actions of reparameterization groups satisfy the new notion of scale-smoothness for maps between Banach spaces. Neighbourhoods of maps with broken or nodal domains are given local polyfold models as the image of a retraction (modulo a finite group action in the case of isotropy), which becomes scale-smooth after adjusting the smooth structure near nodal curves in Deligne-Mumford spaces. 
With this understood, there is little doubt in the existence of polyfold descriptions for moduli spaces. The much more audacious claim of polyfold theory is the existence of an abstract perturbation scheme for moduli spaces that are described as zero set of a scale-smooth section over a polyfold. However, this claim is fully substantiated in \cite{HWZbook}. 
So the goal of this paper is to demonstrate the use of this abstract perturbation scheme once polyfold descriptions for the basic building blocks of moduli spaces are given.

We moreover chose this structure to give an example of how rigorous and transparent proofs can be written at a time when parts of their foundation are unpublished or in question. 
\end{rmk}

To describe our proof,
let $CF=\oplus_{\g\in\cP(H)} \L \la \g \ra$ be the Floer chain group of the Hamiltonian $H$ with coefficients in the Novikov field $\L$ (see \S\ref{sec:Novikov}).
Let $(CM, \rd)$ be the Morse complex with coefficients in $\L$ associated to a Morse function $f:M\to\R$ and a suitable metric on $M$
(see \S\ref{sec:Morse}).
Then we will prove the following in Lemma~\ref{lem:pss-novikov}, Definition~\ref{def:iota,h}, and Lemmas~\ref{lem:iota chain}, \ref{lem:iota triangular}, \ref{lem:h chain}.

\begin{thm} \label{thm:main}
There exist $\L$-linear maps $PSS:CM \to CF$, $SSP: CF \to CM$, $\iota:CM \to CM$, and $h:CM \to CM$ such that the following holds.
\begin{enumilist}
\item
$\iota$ is a chain map, that is $\iota \circ \rd = \rd \circ \iota$.
\item
$\iota$ is a $\L$-module isomorphism.
\item
$h$ is a chain homotopy between $SSP\circ PSS$ and $\iota$, that is
$\iota  - SSP\circ PSS  = \rd \circ h + h \circ \rd$.
\end{enumilist}
\end{thm}

Here we view the Floer chain group $CF$ as a vector space over $\Lambda$ -- not as a chain complex, and in particular do not consider a Floer differential. Thus we are neither constructing  a Floer homology for $H$, nor identifying it with the Morse homology of $f$. However, the algebraic structures in Theorem~\ref{thm:main} suffice to deduce the homological Arnold conjecture for the Hamiltonian $H$ as follows.

\begin{proof}[Proof of Theorem~\ref{thm:arnold}]
Denote the sum of the Betti numbers $k := \sum_{i=0}^{\dim M} \dim H_i(M;\Q)$. Let $(CM_{\Q}, \rd_{\Q})$ be the Morse complex over $\Q$ as defined in \S\ref{sec:Morse}. Then 
by the isomorphism of singular and Morse homology there exist $c_1,\ldots,c_k \in CM_{\Q}$ that are cycles, $\rd_\Q c_i=0$, and linearly independent in the Morse homology over $\Q$.
Since the Morse differential $\rd:CM\to CM$ is given by $\L$-linear extension of $\rd_\Q$ from $CM_{\Q} \subset CM$ the chains $c_1,\ldots,c_k \in CM$ are also cycles $\rd c_i = \rd_\Q c_i = 0$ and linearly independent in the Morse homology over $\L$.
By Theorem~\ref{thm:main}~(i),(ii), $\iota$ induces an isomorphism $H\iota:HM\to HM$ on homology.
This in particular implies that $[\iota(c_1)],\ldots,[\iota(c_k)]\in HM$ are also linearly independent in homology, that is for any $\lambda_1, \ldots,\lambda_k\in\Lambda$ we have
\begin{equation} \label{eq:independence}
 \textstyle
\sum_{i=1}^k \lambda_i \cdot \iota(c_i) \;\in\; \im\rd \qquad\Longrightarrow\qquad \lambda_1=\ldots=\lambda_k = 0 .
\end{equation}

We now show that $PSS(c_1),\ldots, PSS(c_k) \in CF$ are $\Lambda$-linearly independent, proving $\#\cP(H) \geq k$ since the elements of $\cP(H)$ generate $CF$ by definition. This proves the theorem.

Let $\lambda_1, \ldots,\lambda_k\in\Lambda$ be a tuple such that
$$\textstyle
\sum_{i=0}^k \lambda_i \cdot PSS(c_i) = 0.
$$
Then we deduce from Theorem~\ref{thm:main}~(iii) that
\begin{align*}
\textstyle
\sum_{i=0}^k \lambda_i \cdot \iota(c_i)
&\;=\; \textstyle
\sum_{i=0}^k \lambda_i \cdot \bigl(  SSP\bigl(PSS(c_i)\bigr) + \rd h(c_i) + h(\rd c_i) \bigr) \\
&\;=\; \textstyle 
SSP\Bigl( \sum_{i=0}^k \lambda_i \cdot PSS (c_i ) \Bigr)  \;+\;  \sum_{i=0}^k \lambda_i \cdot \rd h(c_i)  \;=\; \textstyle
 \rd  \Bigl( \sum_{i=0}^k \lambda_i \cdot h(c_i)  \Bigr),
\end{align*}
which implies $\l_1 = \ldots = \l_k = 0$ by \eqref{eq:independence}.
\end{proof}

This algebraically minimalistic approach of deducing the homological Arnold conjecture from the existence of maps $PSS$ and $SSP$ whose composition is chain homotopic to an isomorphism on the Morse complex was developed in 2011 discussions of the second author, Peter Albers, and Joel Fish with Mohammed Abouzaid and Thomas Kragh. 
These were prompted by our observation that proofs of ``Floer homology equals Morse homology'' require equivariant transversality which is generally obstructed  -- even for equivariant sections of finite rank bundles.  
Thus our goal was a proof using the least amount of geometric insights or new abstract tools. 
Beyond this we expect the \cite{PSS}-approach to yield an isomorphism between Floer and Morse homology, and spectral invariants \cite{schwarz-spec} on all closed symplectic manifolds, 
using refinements of polyfold theory described in Remark~\ref{rmk:equivariant}.

\smallskip

To maximize accessibility we begin with reviews of the pertinent facts on the Novikov field, \S\ref{sec:Novikov}, and Morse trajectories, \S\ref{sec:Morse}. 
The proof of Theorem~\ref{thm:main} then proceeds by constructing the $PSS$ and $SSP$ maps in \S\ref{sec:PSS} from curves in $\C^\pm \times M$, constructing the isomorphism $\iota$ and chain homotopy $h$ in \S\ref{sec:iota,h} from curves in $\CP^1 \times M$ and its degeneration into $\C^-\times M$ and $\C^+ \times M$, and proving their algebraic relations in \S\ref{sec:algebra} by constructing coherent perturbations. 
We give a detailed account of these iterative constructions in the proofs of Lemma~\ref{lem:iota chain} and \ref{lem:h chain}. While these results should be contained in \cite{fh-sft-full}, neck-stretching is not addressed in \cite{fh-sft}, and it seemed timely to give the proof in a case whose structure is vastly simplified by the absence of trivial cylinders compared with \cite[\S 3.5]{fh-sft}.
To strike a balance between technical details and maximal accessibility, we have clearly labeled all such technical work. Readers willing to view polyfold theory as a black box can save 20 pages by skipping these parts. 
For readers new to polyfold theory we provide in Appendix~\ref{sec:polyfold} a summary of all notions and facts that are necessary for the present application. Here we moreover establish in Theorem~\ref{thm:transversality} a relative perturbation result that should be of independent interest: It allows one to bring moduli spaces with an evaluation map into general position to a countable collection of submanifolds. 
We combine this result with \cite{Ben-fiber} to construct polyfold descriptions of the \cite{PSS} moduli spaces as fiber products of SFT moduli spaces with the Morse trajectory spaces constructed in \cite{W-Morse}.

\begin{rmk} \rm \label{rmk:equivariant}
(i)
There are essentially two approaches to the general Arnold conjecture as stated in Theorem~\ref{thm:arnold}. The first -- developed by \cite{f4} and used verbatim in \cite{HS,Ono,FO,LT} -- is to establish the independence of Floer homology from the Hamiltonian function, and to identify the Floer complex for a $\cC^2$-small $S^1$-invariant Hamiltonian $H:M\to\R$ with the Morse complex for $H$. This requires $S^1$-equivariant transversality to argue that isolated Floer trajectories must be $S^1$-invariant, hence Morse trajectories.
A conceptually transparent construction of equivariant and transverse perturbations -- under transversality assumptions at the fixed point set which are met in this setting -- can be found in \cite{zhengyi-quotient}, assuming a polyfold description of Floer trajectories. 

\medskip\noindent(ii)
The second approach to Theorem~\ref{thm:arnold} by \cite{PSS} is to construct a direct isomorphism between the Floer homology of the given Hamiltonian and the Morse homology for some unrelated Morse function. Two chain maps $PSS: CM \to CF$, $SSP: CF \to CM$ between the Morse and Floer complexes are constructed from moduli spaces of once punctured perturbed holomorphic spheres with one marking evaluating to the unstable resp.\ stable manifold of a Morse critical point, and with the given Hamiltonian perturbation of the Cauchy-Riemann operator on a cylindrical neighbourhood of the puncture. Then gluing and degeneration arguments are used to argue that both $PSS\circ SSP $ and $SSP\circ PSS$ are chain homotopic to the identity, and hence $SSP$ is the inverse of $PSS$ on homology.
However, sphere bubbling can obstruct these arguments: In the first chain homotopy it creates an ambiguity in the choice of nodal gluing when the intermediate Morse trajectory shrinks to zero length. 
(We expect to be able to avoid this by arguing that ``index 1 solutions generically avoid codimension 2 strata'' -- another classical fact in differential geometry that should generalize to polyfold theory.)
The second chain homotopy is as claimed in Theorem~\ref{thm:main}~(iii) but with $\iota=\id$, which requires arguing that the only isolated holomorphic spheres with two marked points evaluating to an unstable and stable manifold are constant.
This again requires $S^1$-equivariant transversality (which we expect to be able to achieve with the techniques in \cite{zhengyi-quotient}). 

\medskip\noindent(iii)
Theorem~\ref{thm:main} is proven by following the  \cite{PSS}-approach as above but avoiding the use of new polyfold technology such as equivariant or strata-avoiding perturbations. In particular, $\iota$ is the map that results from counting holomorphic spheres that intersect an unstable and stable manifold; its invertibility is deduced from an ``upper triangular'' argument.

\medskip\noindent(iv)
The techniques in this paper -- combining existing perturbation technology with the polyfold descriptions of SFT moduli spaces -- would also allow one to define the Floer differential, prove $\rd^2=0$, establish independence of Floer homology from the Hamiltonian (and other geometric data), and prove that $PSS$ and $SSP$ are chain maps. 
Then the chain homotopy between $SSP\circ PSS$ and the isomorphism $\iota$ implies that $PSS$ is injective and $SSP$ surjective on homology. However, proving that $PSS$ and $SSP$ are isomorphisms 
on 
homology, or directly identifying the Floer complex of a small $S^1$-invariant Hamiltonian with its Morse complex, requires the techniques discussed in (ii).

Moreover, a proof of independence of Floer homology from the choice of abstract perturbation would require a study of the algebraic consequences of self-gluing Floer trajectories in expected dimension $-1$ during a homotopy of perturbations, as developed in the $A_\infty$-setting in \cite{jiayong}.
\end{rmk}

\noindent
{\it 
We thank Peter Albers and Joel Fish for helping develop the outline of this project -- and Edi Zehnder for asking the initial question.  
The project was further supported by various discussions with Mohammed Abouzaid, Helmut Hofer, Thomas Kragh, Kris Wysocki, and Zhengyi Zhou.
Crucial financial support was provided by NSF grants DMS-1442345 and DMS-1708916. 
}

\section{The Novikov field}
\label{sec:Novikov}

We use the following Novikov field $\Lambda$ associated to the symplectic manifold $(M,\omega)$. Let $H_2(M)$ denote integral homology and consider the map $\o : H_2(M) \rightarrow \mathbb{R}$ given by the pairing $\o(A) := \langle \o,A \rangle$ for $A \in H_2(M)$. The image of this pairing is a finitely generated additive subgroup of the real numbers denoted
$$\G \,:=\; \im \o \;=\; \o(H_2(M)) \;\subset\; \mathbb{R}.$$
The \emph{Novikov field} $\Lambda$ is the set of formal sums
$$\textstyle
\lambda = \sum_{r \in \G} \lambda_{r} T^{r},$$
where $T$ is a formal variable, with rational coefficients $\lambda_r \in \mathbb{Q}$ which satisfy the finiteness condition
$$
\forall c\in \R \qquad 
\# \{ r \in \G \,\, | \,\, \lambda_r \neq 0,\, r \leq c \} < \infty .
$$
The multiplication is given by
$$\textstyle
\lambda \cdot \mu \;=\;
\left( \sum_{r \in \G} \lambda_r T^{r} \right) 
 \cdot \left( \sum_{s \in \G} \mu_s T^{s}  \right)
 \;:=\; \sum_{t \in \G} \left(
 \sum_{r+s=t} 
 \lambda_r \mu_s  \right) T^{t} . 
$$

This defines a field $\Lambda$ by \cite[Thm.4.1]{HS} and the discussion preceding the theorem in \cite[\S4]{HS}, the key being that $\G$ is a finitely generated subgroup of $\mathbb{R}$.

We will moreover make use of the following generalization of the invertibility of triangular matrices with nonzero diagonal entries.

\begin{lem} \label{lem:invertibilitymatrixnovikov}
Let $M=(\lambda^{ij})_{1\leq i,j \leq \ell} \in\Lambda^{\ell\times \ell}$ be a square matrix with entries $\lambda^{ij}\in\Lambda$ in the Novikov field. Suppose that $\lambda^{ij} = \sum_{r\in \G, r\geq 0} \lambda^{ij}_r T^r$ with $\lambda^{ij}_0 = 0$ for $i\neq j$ and $\lambda^{ii}_0 \neq 0$. Then $M$ is invertible.
\end{lem}
\begin{proof}
Since $\Lambda$ is a field, invertibility of $M$ is equivalent to $\det(M) \neq 0$. Write $\det(M) = \sum_{r \in \G} \mu_rT^r \in \Lambda$ for some $\mu_r \in \mathbb{Q}$. It suffices to show that $\mu_0 \neq 0$.

We proceed by induction on the size of the matrix $M$. In the $\ell=1$ base case, when $M$ is a $1 \times 1$ matrix $M = [\lambda^{11}]$, we have $\det(M) = \lambda^{11}=\sum_{r\in \G} \mu_r T^r$ with $\mu_r=\lambda^{11}_r$ so $\mu_0 = \lambda^{11}_0 \neq 0$.

Now suppose that $M$ is size $\ell \times \ell$ for some $\ell > 1$ and inductively assume that, for any size $(\ell - 1) \times (\ell-1)$ matrix $N$ satisfying the hypotheses of the lemma, we have $\det(N) = \sum_{r \in \G} \mu^N_r T^r$ with $\mu^N_0 \neq 0$. For $1 \leq j \leq 
\ell$, let $C_{1 j}$ denote the matrix obtained by deleting the first row and $j$-th column of $M$. Then $N := C_{11}$ is an $(\ell - 1) \times (\ell-1)$ matrix that satisfies the hypotheses of the lemma, and the cofactor expansion of the determinant yields
$$\textstyle
\det(M) \;=\; \lambda^{11}  \det(N) \;+\; \sum_{j=2}^\ell (-1)^{1+j} \lambda^{1j}  \det(C_{1j}).
$$
By hypothesis, all entries of $M$ are of the form $\lambda^{ij} = \sum_{r \geq 0} \lambda^{ij}_r T^r$. Since the determinants $\det(N)$ and $\det(C_{1j})$ are polynomials of those entries, they are of the same form -- with zero coefficients for $T^r$ with $r < 0$. Since we moreover have $\lambda^{1j}_0 = 0$ for $j \geq 2$ by hypothesis, it follows that the constant term (i.e. the coefficient on $T^0$) of $\lambda^{1j}\det(C_{1j})$ is $0$. Hence the constant term of $\det(M)=\sum \mu_r T^r$ is $\mu_0 = \lambda_0^{11} \cdot \mu_0^N$, where $\mu_0^N \neq 0$ by induction and $\lambda_0^{11} \neq 0$ by hypothesis. This implies $\det(M)= \mu_0 + \ldots \neq 0$ and thus finishes the proof.
\end{proof}

\section{The Morse complex and half-infinite Morse trajectories}
\label{sec:Morse}

This section reviews the construction of the Morse complex as well as the compactified spaces of half-infinite Morse trajectories which will appear in all our moduli spaces.

\subsection{Euclidean Morse-Smale pairs}

The Morse complex can be constructed for any Morse-Smale pair of function and metric on a closed smooth manifold $M$ (and more general spaces). However, we will also work with half-infinite Morse trajectories, and to obtain natural manifold with boundary and corner structures on these, we will restrict ourselves to the following special setting.

\begin{dfn} \label{def:ems}
A {\bf Euclidean Morse-Smale pair} on a closed manifold $M$ is a pair $(f,g)$ consisting of a smooth function $f\in\cC^\infty(M,\R)$ and a Riemannian metric $g$ on $M$ satisfying a normal form and transversality condition as follows.
\begin{enumerate}
\item
For every critical point $p\in{\rm Crit}(f)$ of index $|p|\in\N_0$ there exists a local chart $\phi$ to a neighbourhood of $0\in\R^n$ such that
\begin{align*}
\phi^* f  (x_1,\ldots,x_n) &\;=\; f(p) - \half( x_1^2 + \ldots + x_{|p|}^2 ) + \half ( x_{|p|+1}^2 + \ldots + x_n^2 ) ,\\
\phi^* g &\;=\; \rd x_1 \otimes \rd x_1 + \ldots + \rd x_n \otimes \rd x_n .
\end{align*}
\item
For every pair of critical points $p,q\in{\rm Crit}(f)$ the intersection of unstable and stable manifolds is transverse, $W^-_p \pitchfork W^+_q $.
\end{enumerate}
\end{dfn}

\begin{rmk}\rm  \label{rmk:Euclidean}
Euclidean Morse-Smale pairs exist on every closed manifold, and for any given Morse function. Indeed, given any Morse function $f$ and metric $g$, 
there are arbitrarily $C^0$-small perturbations $g'$ of $g$ in any neighborhood of the critical points of $f$ such that $(f,g')$ satisfies Definition~\ref{def:ems}(i); see e.g.\ \cite[Prp.1]{bh}. Furthermore, any $L^2$-generic perturbation $g''$ of $g'$ on annuli around the critical points yields a pair $(f,g'')$ that additionally satisfies Definition~\ref{def:ems}(ii) and hence is a Euclidean Morse-Smale pair; see e.g.\ \cite[Prp.2]{bh} or \cite[Prp.2.24]{schwarz}.
\end{rmk}

\subsection{The Morse complex} \label{ssec:CM}

For distinct critical points $p_-\neq p_+\in{\rm Crit}(f)$ the space of unbroken Morse trajectories (which are necessarily nonconstant) is
\begin{align}
\cM(p_-,p_+) &:= \bigl\{ \t:\R\to M \,\big|\, \dot{\t}=-\nabla f(\t) ,
\lim_{s\to\pm\infty}\t(s)=p_\pm  \bigr\} / \R  \label{eq:inftraj} \\  \nonumber
&\cong \bigl(W^-_{p_-}\cap W^+_{p_+}\bigr) / \R
\quad\cong\; W^-_{p_-}\cap W^+_{p_+} \cap f^{-1}(c) .
\end{align}
It is canonically identified with the intersection of unstable and stable manifold modulo the $\R$-action given by the flow of $-\nabla f$, or their intersection with a level set for any regular value $c\in (f(p_+),f(p_-))$. Both formulations equip it with a canonical smooth structure of dimension $|p_-| - |p_+| - 1$, see e.g.\ \cite[\S2.4.1]{schwarz}.
Moreover, any choice of orientation of the unstable manifolds $W^-_{p}$ for all $p\in{\rm Crit}(f)$ 
induces orientations on the trajectory spaces $\cM(p_-,p_+)$ by e.g.\ \cite[\S3.4]{Weber}. 
Then the Morse chain complex of $(f,g)$ is obtained by counting (with signs induced by the orientations) the zero dimensional spaces of unbroken trajectories, 
\begin{equation} \label{eq:MorseDiff}
CM_\Q := \bigoplus_{p\in {\rm Crit}(f)} \Q \la p \ra , \qquad\qquad
d_\Q \; \la p_- \ra := \sum_{ |p_+|= |p_-|-1 }  \# \cM(p_-,p_+) \; \la p_+ \ra  .
\end{equation}
It computes the singular homology of $M$; see e.g.\ \cite[\S4.3]{schwarz}.
More precisely, the Morse complex is graded 
$CM_\Q=\bigoplus_{i=0,\ldots,\dim M} C_i M$ by Morse indices $C_i M=\bigoplus_{|p|=i} \Q \la p \ra$, and with $d_i:= d_\Q|_{C_iM}$ we have $H_i(M;\Q) \cong \ker d_i / \im d_{i+1}$. 

The PSS and SSP morphisms will be constructed on the Morse complex with coefficients in the Novikov field $\Lambda$ from Section~\ref{sec:Novikov}, 
\begin{equation} \label{eq:CM}
 \textstyle
CM \;=\; CM_\Lambda \,:=\; CM_\Q \otimes \Lambda \;=\;  \bigoplus_{p\in {\rm Crit}(f)} \Lambda \la p \ra ,
\end{equation}
with differential $d=d_\Lambda$ the $\Lambda$-linear extension of $d_\Q$ (defined as above on generators). 
This complex is naturally graded with differential of degree 1, 
\begin{equation} \label{eq:CMgraded}
 \textstyle
C_*M \;=\; \bigoplus_{i=0}^{\dim M} C_i M, \qquad C_iM = \bigoplus_{|p|=i} \Lambda \la p \ra , \qquad
d:C_i M\to C_{i-1}M. 
\end{equation}

\subsection{Compactified spaces of Morse trajectories} \label{ssec:Morse}

Our construction of moduli spaces will also make use of the following spaces of half-infinite unbroken Morse trajectories for $p_\pm\in{\rm Crit}(f)$
\begin{align*}
\cM(M,p_+) &\,:=\; \bigl\{ \t: \phantom{-} [0,\infty) \to M \,\big|\, \dot{\t}=-\nabla f(\t), \lim_{s\to\infty\phantom{-}}\t(s)=p_+  \bigr\}  , \\
\cM(p_-,M) &\,:=\; \bigl\{ \t: (-\infty,0] \to M \,\big|\, \dot{\t}=-\nabla f(\t), \lim_{s\to-\infty}\t(s)=p_-  \bigr\}  .
\end{align*}
These will be equipped with smooth structures of dimension
$\dim \cM(M,p_+) = \dim M - |p_+|$ resp.\ $\dim \cM(p_-,M) = |p_-|$ by the evaluation maps
$$
\ev : \cM(M,p_+) \to M ,  \quad \t \mapsto \t(0),  \qqquad
\ev : \cM(p_-,M) \to M , \quad  \t \mapsto \t(0) ,
$$
which identify the trajectory spaces with the unstable and stable manifolds 
$\cM(M,p_+) \cong W^+_{p_+}$ resp.\ $\cM(p_-,M) \cong W^-_{p_-}$. 
Note that these spaces contain constant trajectories at a critical point, $\{\tau\equiv p_+\} \in \cM(M,p_+)$ and $\{\tau\equiv p_-\} \in \cM(p_-,M)$.
To compactify these trajectory spaces in a manner compatible with Morse theory, we cannot simply take the closure of the unstable or stable manifold $W^\pm_{p_\pm}\subset M$, but must add broken trajectories involving the bi-infinite Morse trajectories.
The bi-infinite trajectories from \eqref{eq:inftraj} which appear in such a compactification are always nonconstant, i.e.\ between distinct critical points $p_-\neq p_+$. So, unlike constant half-infinite length trajectories, our constructions will not involve constant bi-infinite trajectories, and we simplify subsequent notation by setting $\cM(p,p):=\emptyset$ for all $p\in{\rm Crit}(f)$. 
With that we first introduce spaces of $k$-fold broken half- or bi-infinite Morse trajectories for $k\in\N_0$ and $p_\pm\in\Crit(f)$,  
\begin{align}
\bM(M,p_+)_k &:= \textstyle
\bigcup_{p_1,\ldots, p_k\in{\rm Crit}(f)}
\cM(M,p_1)\times \cM(p_1,p_2) \ldots \times \cM(p_k,p_+), \nonumber \\
\bM(p_-,M)_k &:= \textstyle
\bigcup_{p_1,\ldots, p_k\in{\rm Crit}(f)}
\cM(p_-,p_1)\times \cM(p_1,p_2) \ldots \times \cM(p_k,M), \label{eq:bMk}\\
\bM(p_-,p_+)_k &:= \textstyle
\bigcup_{p_1,\ldots, p_k\in{\rm Crit}(f)}
\cM(p_-,p_1)\times \cM(p_1,p_2) \ldots \times \cM(p_k,p_+).  \nonumber
\end{align}
Now the compactifications of the spaces of half- or bi-infinite Morse trajectories are given by 
$$ 
\bM(M,p_+) := \bigcup_{k\in\N_0} \bM(M,p_+)_k , \quad
\bM(p_-,M) :=  \bigcup_{k\in\N_0} \bM(p_-,M)_k , \quad
\bM(p_-,p_+) := \bigcup_{k\in\N_0} \bM(p_-,p_+)_k , 
$$
with topology given by the Hausdorff distance between the images of the broken or unbroken trajectories. Compactness of these spaces is proven analogously to the bi-infinite Morse trajectory spaces in e.g.\ \cite[Prp.3]{bh}, using \cite[Lemma~3.5]{W-Morse}. Moreover, \cite[Lemma~3.3]{W-Morse} shows that the evaluation maps extend continuously to
\begin{align} \label{eval}
\ev \,:\; \bM(M,p_+) \to M , \quad &\bigl(\t_0,[\t_1], \phantom{\ell} \ldots \phantom{\ell} ,[\t_k]) \mapsto \t_0(0) ,  \\
\ev \,:\; \bM(p_-,M) \to M , \quad & \bigl([\t_0],\ldots,[\t_{k-1}],\t_k)  \mapsto \t_k(0) . \nonumber
\end{align}
Smooth structures on these spaces are obtained by the following variation of a folk theorem, which is proven in \cite{W-Morse}, using techniques similar to those of \cite{bh} for the bi-infinite trajectory spaces.

\begin{thm} \label{thm:Morse}
Let $(f,g)$ be a Euclidean Morse-Smale pair and $p_\pm\in\Crit(f)$.
Then $\bM(M,p_+)$, $\bM(p_-,M)$, and $\bM(p_-,p_+)$ are compact, separable metric spaces and carry the structure of a smooth manifold with corners of dimension 
$\dim\bM(M,p_+)=\dim M - |p_+|$, $\dim\bM(p_-,M)=|p_-|$, and $\dim\bM(p_-,p_+)=|p_-|-|p_+|-1$.
Their $k$-th boundary stratum is $\partial_k\bM(\ldots)=\bM(\ldots)_k$. 
Moreover, the evaluation maps \eqref{eval} are smooth.
\end{thm}

For reference, we recall the definition of a manifold with (boundary and) corners and its strata.

\begin{dfn} \label{def:corners}
A {\bf smooth manifold with corners} of dimension $n\in\N_0$ is a second countable Hausdorff space $M$ together with a maximal atlas of charts $\phi_\iota:M\supset U_\iota \to V_\iota\subset [0,\infty)^n$ (i.e.\ homeomorphisms between open sets such that $\cup_\iota U_\iota = M$) whose transition maps are smooth.

For $k=0,\ldots,n$ the $k$-th boundary stratum $\partial_k M$ is the set of all $x\in M$ such that for some (and hence every) chart the point $\phi_\iota(x)\in [0,\infty)^n$ has $k$ components equal to $0$.
\end{dfn}

\begin{rmk} \rm \label{rmk:MorseOrient}
\begin{enumilist}
\item
To orient the Morse trajectory spaces in Theorem~\ref{thm:Morse} we fix a choice of orientation on each unstable manifold $W^-_{p}\cong \cM(p,M)$ for $p\in{\rm Crit}(f)$, and orient $W^+_p\cong\cM(M,p)$ such that $\rT_p M = \rT_p W^- \oplus \rT_p W^+$ induces the orientation on $M$ given by the symplectic form. 
This also induces orientations on $\cM(p_-,p_+)= W^-_{p_-}\cap  W^+_{p_+} / \R$ that are coherent (by e.g.\ \cite[\S3.4]{Weber}) in the sense that the top strata of the oriented boundaries of the compactified Morse trajectory spaces are products $\partial_1\bM(\cdot, \cdot)=\bigcup_{q\in{\rm Crit}(f)} o(\cdot,q,\cdot) \cM(\cdot,q)\times \cM(q,\cdot)$ with  universal signs $o(\cdot,q,\cdot)=\pm 1$. We compute the relevant cases: For $\cM(M,q)\times \cM(q,p_+)\hookrightarrow\partial_1\bM(M,p_+)$ with $\dim\cM(q,p_+)=0$ the sign is $o(M,q,p_+)=(-1)^{|p_+|+1}$. 
Indeed, a point in $\cM(q,p_+)$ is positively oriented if $\rT W^-_q \cong \la -\nabla f \ra \times {\rm N} W^+_{p_+}$. Here we identify ${\rm N}_{p_+} W^+_{p_+} \cong \rT_{p_+} W^-_{p_+}$, and the outer normal direction is represented by $ \nabla f$, so that the sign arises from 
\begin{align*}
 \rT W^-_{p_+} \times \rT W^+_{p_+} 
&\;\cong\;  \rT W^-_q \times \rT W^+_q 
\;\cong\; 
\la -\nabla f \ra \times {\rm T} W^-_{p_+} \times \rT W^+_q \\
&\;\cong\; 
 \rT W^-_{p_+} \times  \la (-1)^{1+ |p_+|} \nabla f \ra \times \rT W^+_q \times \rT \cM(q,p_+) .
\end{align*}
Similarly, for $\cM(p_-,q)\times \cM(q,M)\hookrightarrow\partial_1\bM(p_-,M)$ with $\dim \cM(p_-,q) = 0$ the sign is $o(p_-,q,M)=+1$ since $-\nabla f$ is an outer normal and $\rT W^-_{p_-} \cong \la -\nabla f \ra \times \rT W^-_{q}$ when $\rT \cM(p_-,q)=+ \{0\}$.

\item
For computational purposes in \S\ref{ssec:iso} we determine the fiber products of the compactified Morse trajectory spaces of critical points $p_-,p_+\in\Crit(f)$ with the same Morse index $|p_-|=|p_+|$, 
\begin{align*}
\bM(p_-,M) \leftsub{\ev}{\times}_\ev \bM(M,p_+)
&\;=\;
\left\{ (\ul\tau^- , \ul\tau^+ ) \in  \bM(p_-,M) \times \bM(M,p_+)
\,\left|\, \ev(\ul\tau^-)= \ev(\ul\tau^+) \right. \right\} \\
&\;=\; \begin{cases}
\quad \emptyset &; p_-\neq p_+ , \\
(\tau^-\equiv p_-, \tau^+\equiv p_+) &; p_-= p_+ .
\end{cases}
\end{align*}
To verify this recall that the compactifications $\bM(p_-,M)$ and $\bM(M,p_+)$ are constructed in \eqref{eq:bMk} via broken flow lines involving bi-infinite Morse trajectories in $\cM(p_i,p_{i+1})$, which are (defined to be) nonempty only for $|p_i|>|p_{i+1}|$. So we have
$\cM(p_-,p_1)\times \ldots \times\cM(p_k,M)\subset \bM(p_-,M)$ only for $|p_k|<|p_-|$
and $\cM(M,p_1)\times \ldots \times\cM(p_k,p_+)\subset \bM(M,p_+)$ only for $|p_1|>|p_+|$, 
and thus the image of the evaluation maps are contained in unions of unstable/stable manifolds
$$ \textstyle
\ev(\bM(p_-,M)) \;\subset\; W^-_{p_-} \cup \bigcup_{|q_-|<|p_-|} W^-_{q_-},
\qquad
\ev(\bM(M,p_+)) \;\subset\; W^+_{p_+} \cup \bigcup_{|q_+|>|p_+|} W^+_{q_+} . 
$$
Since the intersections $W^-_{q_-}\cap W^+_{q_+}$ are transverse by the Morse-Smale condition, they can be nonempty only for $|q_-|+ \dim M - |q_+| \geq \dim M$. So this intersection is empty whenever $|q_+|>|q_-|$. 
Thus for $|q_-|<|p_-|=|p_+|<|q_+|$ in the above images we have empty intersections $W^-_{q_-}\cap W^+_{q_+}=\emptyset$ as well as $W^-_{q_-}\cap W^+_{p_+}=\emptyset$ and $W^-_{p_-}\cap W^+_{q_+}=\emptyset$. 
This proves 
$\ev(\bM(p_-,M))\cap \ev(\bM(M,p_+)) =  W^-_{p_-} \cap W^+_{p_+}$, 
and for $p_-\neq p_+$ this intersection is empty by transversality in \eqref{eq:inftraj}. 
Lastly, for $p_\pm=p$ we have 
$W^-_{p} \cap W^+_{p} = \{p\}$ since gradient flows do not allow for nontrivial self-connecting trajectories. 
This proves $\bM(p,M) \leftsub{\ev}{\times}_\ev \bM(M,p) = \{ (p,p) \}$.
\end{enumilist}
\end{rmk}

\section{The PSS and SSP maps}
\label{sec:PSS}

In this section we construct the PSS and SSP morphisms in Theorem~\ref{thm:main} between Morse and Floer complexes.
As in the introduction, we fix a closed symplectic manifold $(M,\omega)$ and a smooth function $H:S^1\times M \to \R$. This induces a time-dependent Hamiltonian vector field $X_H:S^1\to \G(\rT M)$, which we assume to be nondegenerate. Thus it has a finite set of contractible periodic orbits, denoted by $\cP(H)$ as in \eqref{eq:PH}. 
We moreover pick a Morse function $f:M\to\R$ and denote its -- again finite -- set of critical points by ${\rm Crit}(f)$. Then we will work with the Floer and Morse complexes over the Novikov field from Section~\ref{sec:Novikov}, 
$$
CF=\oplus_{\g\in\cP(H)} \L \la \g \ra, \qquad\qquad CM=\oplus_{p\in {\rm Crit}(f)} \L \la p \ra , 
$$
and construct the $\L$-linear maps $PSS:CM \to CF$, $SSP: CF \to CM$ from moduli spaces which we introduce in \S\ref{ssec:moduli}. We provide these moduli spaces with a compactification and polyfold description in \S\ref{ssec:poly}, and in \S\ref{ssec:construct} rigorously construct the PSS/SSP map by using polyfold perturbations to obtain well defined (but still choice dependent) counts of compactified-and-perturbed moduli spaces. 

\subsection{The Piunikhin-Salamon-Schwarz moduli spaces} \label{ssec:moduli}

To construct the moduli spaces, we need to make further choices as follows.
\begin{itemlist}
\item 
Let $J$ be an $\omega$-compatible almost complex structure on $M$. 

Then the Cauchy-Riemann operator on maps $u: \Sigma \to M$ parametrized by a Riemann surface $\Sigma$ with complex structure $j$
is $\dbar_J u := \frac 12 \bigl( \rd u + J(u) \circ \rd u \circ j\bigr)\in\Om^{0,1}(\Sigma,u^*\rT M)$.  
\item
Let $g$ be a metric on $M$ such that $(f,g)$ is a Euclidean Morse-Smale pair as in Definition~\ref{def:ems}. It exists by Remark~\ref{rmk:Euclidean}. 
\item
Let $\beta: [0,\infty) \to [0,1]$ be a smooth cutoff function with $\beta|_{[0,1]}\equiv 0$, $\beta'\geq 0$, and $\beta|_{[e,\infty)}\equiv 1$. 

Then we define the anti-holomorphic vector-field-valued 1-form
$Y_H\in \Om^{0,1}(\C, \G(\rT M))$ in polar coordinates 
$$
Y_H(re^{i\theta},x) := \tfrac 12 \beta(r) \bigl( J X_H(\theta,x) \, r^{-1} \rd r + X_H(\theta,x) \, \rd\theta  \bigr) .
$$
In the notation of \cite[\S8.1]{MS}, we have $Y_H=-(X_{H_\beta})^{0,1}$ given by the anti-holomorphic part of the 1-form with values in Hamiltonian vector fields $X_{H_\beta}$ which arises from the 1-form with values in smooth functions $H_\beta\in\Om^1(\C,\cC^\infty(M))$ given by 
$H_\beta (r e^{i\theta}) = \beta(r) H (\theta,\cdot) \rd\theta$.

The vector-field-valued 1-form $Y_H$ encodes the Floer equation on both the positive cylindrical end $\{z\in \C\,|\, |z|\geq e\}\cong [1,\infty)\times S^1$ and the negative end $\{|z|\geq e\}\cong (-\infty,-1]\times S^1$ (where $\beta\equiv 1$) as follows: 
The reparametrization $v(s,t):= u(e^{\pm(s+it)})$ of a map $u: \C \to M$ satisfies the Floer equation $(\partial_s + J \partial_t) v(s,t) = J X_H(t,v(s,t))$ iff $\dbar_J u (z) = Y_H(z,u(z))$. 

\item
For each $\g \in \cP(H)$,  fix a smooth disk
$u_{\g} : D^2 \rightarrow M$ with $u_{\g}|_{\partial D^2}(e^{it}) = \g(t)$. 

We denote the oriented complex plane by $\mathbb{C}^{+}:=(\mathbb{C}, i)=\C$, and denote its reversed complex structure and orientation by $\mathbb{C}^{-}:=(\mathbb{C}, - i)$.
Then for $u : \C^{\pm} \rightarrow M$ with $\lim_{R\rightarrow \infty} u(Re^{\pm i t}) = \g(t)$, denote by $u \# u_{\g} : \CP^1 \rightarrow M$ the continuous map given by gluing $u$ to $u_{\gamma}^{\pm}$ (where the $\pm$ denotes the orientation of $D^2$). By abuse of language, we will call $A:=[u \# u_{\g}] = (u \# u_{\g})_*[\CP^1] \in H_2(M)$ the homology class represented by $u$.
Moreover, we denote by $\Ti u_\g: D^2 \to D^2\times M$ the graph of $u_\g$. Then the graph $\Ti u : \C \rightarrow \C\times M, z\mapsto (z,u(z))$ glues with $\Ti u_\g^{\pm}$ to a continuous map representing $[\Ti u \# \Ti u_{\g}] = \Ti A:= [\CP^1] + A \in H_2(\CP^1\times M)$, or more precisely $\Ti A = [\CP^1]\times[{\rm pt}] + [{\rm pt}]\times A$.
Now the condition $[v \# \Ti u_{\g}] = \Ti A$ makes sense for other maps $v : \C \rightarrow \C\times M$ with the same asymptotic behaviour, and we say $v$ represents $\Ti A$. 
In fact, we will suppress the notation $\Ti A$ and label spaces with $A$ -- as this specifies the topological type of $v$.
\end{itemlist}

\medskip\noindent
Given such choices, the (choice-dependent) morphisms $PSS:CM \to CF$ and $SSP: CF \to CM$
will be constructed from the following moduli spaces for critical points $p\in{\rm Crit}(f)$, periodic orbits $\g\in\cP(H)$, and $A \in H_2(M)$
\begin{align*}
\cM(p,\g; A) &:= \bigl\{  u : \C^+ \to M \;\big|\; u(0)\in W^-_p , \; \dbar_J u = Y_H(u) \phantom{\ell}, \;\lim_{\scriptscriptstyle R\to\infty\phantom{-} } u(Re^{it})=\g(t) , \; [u \# u_{\g}] =A  \bigr\} , \\
\cM(\g,p; A) &:= \bigl\{ u : \C^- \to M \;\big|\; u(0)\in W^+_p , \; \dbar_J u = Y_H(u), \;\lim_{\scriptscriptstyle R\to\infty} u(Re^{-it})=\g(t) , \; [u \# u_{\g} ]=A \bigr\}.
\end{align*}
Each of these moduli spaces can be described as the zero set of a Fredholm section
$\dbar_J - Y_H : \cB_\pm \to \cE_\pm$. Here the Banach manifolds $\cB_\pm$ are given by a weighted Sobolev closure of the set of smooth maps $u: \C^{\pm} \to M$ representing the homology class $A$ with point constraint $u(0) \in W^\mp_p$ and satisfying a decay condition $\lim_{R \rightarrow \infty}u(Re^{\pm it})=\g(t)$, but not necessarily satisfying the perturbed Cauchy-Riemann equation $\dbar_J u = Y_H(u)$.
Then $\dbar_J - Y_H$ is a Fredholm section of index
\begin{align}\label{PSS index}
I(p,\g;A) &= \phantom{-} CZ(\g) + 2c_1(A) - \tfrac{\dim M}{2} + |p| , \\
I(\g,p;A) &= -CZ(\g) + 2c_1(A) + \tfrac{\dim M}{2}  - |p| ,  \nonumber
\end{align}
where $CZ(\g)$ is the Conley-Zehnder index with respect to a trivialization of $u_{\g}^*\rT M$ as in e.g.\ \cite{schwarz-thesis}, $c_1(A)$ is the first Chern class of $(TM,J)$ paired with $A$, and $|p|$ is the Morse index of $p \in {\rm Crit}(f)$.

If the moduli spaces were compact oriented manifolds, then we could define $PSS$ (and analogously $SSP$) by a signed count of the index $0$ solutions,
$$\textstyle 
PSS \la p \ra := \# \cM(p,\g; A) \cdot T^{\o(A)} \la \g \ra , 
$$
where the sum is over $\g\in \cP(H)$ and $A\in H_2(M)$ with $I(p,\g;A)=0$. 
In many cases -- if sphere bubbles of negative Chern number can be excluded -- this compactness and regularity can be achieved by a geometric perturbation of the equation, e.g.\ in the choice of almost complex structure. 
In general, obtaining well defined ``counts'' of the moduli spaces requires an abstract regularization scheme. We will use polyfold theory to replace ``$\# \cM(p,\g; A)$'' by a count of $0$-dimensional perturbed moduli spaces. In the presence of sphere bubbles with nontrivial isotropy, the perturbations will be multi-valued, yielding rational counts.

\begin{rmk} \rm \label{rmk:pss-energy}
Compactness, or rather Gromov-compactifications, of the moduli spaces $\cM(p,\g; A)$ and $\cM(\g,p; A)$ will result from energy estimates \cite[Remark~8.1.7]{MS} for solutions of $\dbar_J u = Y_H(u)$,
\begin{equation} \label{eq:pss-energy}
\textstyle
E(u) \,:=\; \frac 12 \int_\C |\rd u + X_{H_\beta}(u) |  \;\leq\; 
\int_\C u^*\omega + \|R_{H_\beta}\| \;\leq\; \omega([u\#u_\g]) + K . 
\end{equation}
Here the curvature
$R_{H_\beta} \dvol_\C = \rd H_\beta + \tfrac 12 { H_\beta\wedge H_\beta } =  \beta' \,  H \, \rd r \wedge \rd \theta$ has finite Hofer norm
$$\textstyle
\| R_{H_\beta}\| = \int_C ( \max R_{H_\beta} - \max R_{H_\beta} ) = \int_0^\infty \int_{S^1} |\beta'(r)| ( \max_{x\in M} H(\theta,x) -  \min_{x\in M} H(\theta,x) ) \,  \rd\theta \, \rd r
$$
since $\beta'$ has compact support in $[1,e]$. 
Since moreover $\cP(H)$ is a finite set, we obtain the above estimate with a finite constant
$K := \| R_{H_\beta} \| +\max_{\g\in\cP(H)} \int_{D^2} u_\gamma^* \omega$. 
Thus the energy of the perturbed pseudoholomorphic maps in each of our moduli spaces will be bounded since we fix $[u\#u_\g]=A$. 

Now SFT-compactness \cite{sft-compactness} asserts that for any $C>0$ the set of solutions of bounded energy $\{  u : \C \to M \,| \, \dbar_J u = Y_H(u) , \lim_{\scriptscriptstyle R\to\infty} u(Re^{\pm it})=\g(t) , E(u)\leq C  \}$ is compact up to breaking and bubbling.
This compactness will be stated rigorously in polyfold terms in Assumption~\ref{ass:iota,h}~(ii). 
\end{rmk}

\subsection{Polyfold description of moduli spaces} \label{ssec:poly}

We will obtain a polyfold description for the moduli spaces in \S\ref{ssec:moduli} by a fiber product construction motivated by the natural identifications
\begin{equation}\label{eq:Mfiber}
\cM(p,\g; A) \cong \cM(p,M) \leftsub{\ev}\times_{\ev} \cM^- (\g; A), 
\qquad
\cM(\g,p; A) \cong \cM^+ (\g; A) \leftsub{\ev}\times_{\ev} \cM(M,p) .
\end{equation}
This couples the half-infinite Morse trajectory spaces from \S\ref{ssec:Morse} with a space of perturbed pseudoholomorphic maps
\begin{align}
\cM^\pm(\g; A) & := \bigl\{ u : \C^{\pm} \to M  \;\big|\;  \dbar_J u = Y_H(u) , \; \lim_{\scriptscriptstyle R\to\infty} u(Re^{\pm it})=\g(t) , \; [u\# u_\g]=A \bigr\},  \label{HOLmod}
\end{align} 
via the evaluation maps \eqref{eval} and
\begin{align}
\label{HOLeval} 
\ev &: \cM^\pm(\g; A) \to M, \quad u \mapsto u(0). 
\end{align}
More precisely, the general approach to obtaining counts or more general invariants from moduli spaces such as \eqref{eq:Mfiber} is to replace them by compact manifolds -- or more general `regularizations' which still carry `virtual fundamental classes').
Polyfold theory offers a universal regularization approach after requiring a compactification $\cM(\ldots)\subset\bM(\ldots)$ of the moduli space and a description of the compact moduli space $\bM(\ldots)=\sigma^{-1}(0)$ as zero set of a sc-Fredholm section $\sigma:\cB(\ldots)\to\cE(\ldots)$ of a strong polyfold bundle. For an introduction to the language \cite{HWZbook} used here see Appendix \S\ref{sec:polyfold}.

The Morse trajectory spaces are compactified and given a smooth structure in Theorem~\ref{thm:Morse}. The Gromov compactification and perturbation theory for \eqref{HOLmod} will be achieved by identifying theses spaces with moduli spaces that appear in Symplectic Field Theory (SFT) as introduced in \cite{egh}, compactified in \cite{sft-compactness, CM-SFT-compactness}, and given a polyfold description in \cite{fh-sft-full}. 
Here we identify $u:\C\to M$ with the map to its graph $\Ti u : \C \to \C\times M, z \mapsto (z,u(z))$ as in \cite[\S8.1]{MS} to obtain a homeomorphism (in appropriate topologies) $\cM^\pm(\g; A) \cong \Ti\cM^\pm_{\rm SFT}(\Ti \g; A)/{\rm Aut}(\C^\pm)$ 
to an SFT moduli space for the symplectic cobordism\footnote{For definitions of these notions see \cite[\S2]{CM-SFT-compactness}. 
For $\C\times M$ the positive symplectization end is $\R^+\times S^1\times M \to \C\times M, (r, \theta , x) \mapsto ( e^{r+i\theta}, x)$. After reversing orientation on $\C$ there is an analogous negative end $\R^-\times S^1\times M \hookrightarrow \C^-\times M$.} $\C^\pm\times M$ between $\emptyset$ and $S^1\times M$. Here $S^1 \times M$ is equipped with the stable Hamiltonian structure
$(\pm\rd t,\omega + dH_t \wedge dt)$ whose Reeb field $\pm \partial_t + X_{H_t}$ has simply covered Reeb orbits\footnote{Here we have implicitly chosen asymptotic markers that fix a parametrization of each Reeb orbit.} given by the graphs $\tilde{\g} : t \mapsto (\pm t,\g(t))$ of the 
periodic orbits $\g \in \cP(H).$
Moreover, ${\rm Aut}(\C^\pm)$ is the action of biholomorphisms $\phi:\C\to\C$ by reparametrization $v\mapsto v\circ\phi$ on the SFT space for an almost complex structure $\Ti J_H^\pm$ on $\C^\pm\times M$ induced by $J$, $X_H$, and $j=\pm i$ on $\C^\pm$, 
$$
\widetilde\cM^\pm_{\rm SFT}(\Ti\g; A) := \bigl\{ v : \C^\pm \to \C^\pm\times M  \;\big|\;  \dbar_{\Ti J^\pm_H} v = 0 , \; v(Re^{\pm it}) \sim \Ti\g_R(t) , \;  [v\#\Ti u_\g]=[\CP^1] + A \bigr\}   .
$$
More precisely, the asymptotic requirement is $d_{\C\times M}\bigl( v(Re^{\pm i(t + t_0)}) , \Ti\g_R(t) \bigr) \to 0$ for some $t_0 \in S^1$ as $R\to \infty$ for the graphs $\Ti\g_R(t) = ( Re^{\pm it} , \g(t) )$ of the orbit $\g$ parametrized by $S^1 \cong \{|z|=R\}\subset\C^\pm$.

To express the evaluation \eqref{HOLeval} in SFT terms note that a holomorphic map in the given homology class intersects the holomorphic submanifold $\{0\}\times M$ in a unique point\footnote{For solutions in $\widetilde\cM^\pm_{\rm SFT}(\Ti\g; A)$ this follows from 
$\pr_{\C^\pm}\circ v : \C^\pm \to \C^\pm$ being an entire function with a pole of order $1$ at infinity (prescribed by the asymptotics). For ${\Ti J^\pm_H}$-holomorphic curves in the compactification, it follows from positivity of intersections, see e.g.\ \cite[Prop.7.1]{CM}.
}, so we can fix the point $0\in\C^\pm$ in the domain where this intersection occurs and rewrite the moduli space $\cM^\pm(\g; A) \cong  \bigl\{ v \in \widetilde\cM^\pm_{\rm SFT}(\Ti\g; A) \;\big|\; v(0)\in \{0\}\times M \bigr\}/{\rm Aut}(\C^\pm,0)$ with a slicing condition and quotient by the biholomorphisms which fix $0\in\C^\pm$.  
Thus we rewrite \eqref{eq:Mfiber} into the fiber products over $\C^\pm\times M$
\begin{align}\label{eq:MfiberSFT}
\cM(p,\g; A) & \;\cong\; \cM(p,M) \; \leftsub{\{0\}\times \ev}\times_{\ev^+} \; \cM^+_{\rm SFT} (\g; A), \\
\cM(\g,p; A) &\;\cong\; \cM^-_{\rm SFT} (\g; A) \; \leftsub{\ev^-}\times_{\{0\}\times\ev} \; \cM(M,p)  \nonumber
\end{align}
using evaluation maps on the SFT moduli space with one marked point
\begin{align}
\ev^\pm \;:\;  \cM^\pm_{\rm SFT}(\g; A) \,:=\; \quo{\widetilde\cM^\pm_{\rm SFT}(\Ti\g; A)}{{\rm Aut}(\C^\pm,0)} 
\;\to\; \C^\pm \times M, \qquad [v] \;\mapsto\;  v(0) .  \label{SFTeval}
\end{align} 
Now we will obtain a polyfold description of the PSS/SSP moduli spaces \eqref{eq:MfiberSFT} by the slicing construction of \cite{Ben-fiber} applied to polyfold descriptions of the SFT-moduli spaces $\cM^\pm_{\rm SFT}(\Ti\g; A)$ (compactified as space of pseudoholomorphic buildings with one marked point). 
This result is outlined in \cite{fh-sft}, but to enable a self-contained proof of our results, we formulate it as assumption, where we use
$$
\overline{\C^\pm} \,:=\; \C^\pm \cup S_1 \; \cong\; \{z\in\C^\pm  \,|\, |z|\leq 1\} 
$$
as target factor for a simplified evaluation map, as explained in the following remark.

\begin{rmk}\rm \label{rmk:evaluation drama}
Note that the compactified moduli space $\bM^\pm_{\rm SFT}(\g; A)$ -- in view of the noncompact target $\C^\pm \times M$ -- contains broken curves $\ul v : \Sigma = \C^\pm \sqcup \R\times S^1 \sqcup \ldots \sqcup \R\times S^1 \to \Sigma\times M$. We do not need a precise description of this compactification (beyond the fact that it exists and is cut out by a sc-Fredholm section), but it affects the formulation of the evaluation maps $[\ul v , z_0] \mapsto \ul v(z_0)$ for a marked point $z_0\in\Sigma$ that $\ul v$ might map to a cylinder factor $\R\times S^1 \times M \subset \Sigma \times M$. 
We will simplify the resulting sc$^\infty$ evaluation with varying target -- being developed in \cite{fh-sft-full} -- to a continuous evaluation map
$\overline\ev^\pm: \bM^\pm_{\rm SFT}(\g; A) \to \overline{\C^\pm}$ into the compactified target $\overline{\C^\pm}$. 

For that purpose we topologize $\overline{\C^\pm}\cong\{ |z|\leq 1 \}$ as a disk via a diffeomorphism $\C^\pm \to  \{|z|< 1\}$, $r e^{i\theta} \mapsto f(r) e^{i\theta}$ induced by a diffeomorphism $f:[0,\infty) \to [0,1)$ that is the identity near $0$, and its extension to a homeomorphism $\overline{\C^\pm}\to\{|z|\leq 1\}$ via $S^1=\qu{\R}{2\pi\Z} \to \{ |z|=1\}, \theta \mapsto e^{\pm i\theta}$. 
Then for any marked point $z_0\in\R\times S^1$ on a cylinder we project the evaluation $\ul v(z_0)\in\R\times S^1\times M$ to $S^1\times M= \partial\,\ov{\C^\pm}\times M$ by forgetting the $\R$-factor. The resulting simplified evaluation map will be unchanged and thus still sc$^\infty$ when restricted to the open subset $(\overline\ev^\pm)^{-1}(\C^\pm\times M)$ of the ambient polyfold -- as stated in (iii) below. This open subset inherits a scale-smooth structure, and still contains some broken curves -- just not those on which the marked point leaves the main component. 
This suffices for our purposes since the fiber product construction uses the evaluation map only in an open set of curves $[\ul v , z_0]$ with $\ul v(z_0)\approx 0 \in \C^\pm$.
\end{rmk}

In Assumption~\ref{ass:pss}, Remark~\ref{rmk:ass}, and Lemma~\ref{lem:PSS poly} we introduce some of the polyfolds under construction in \cite{fh-sft-full} and their expected properties. To describe these objects we introduce a significant amount of notation. A summary of the types of curves in each polyfold and subsets thereof is displayed in Table~\ref{tab:asspolyfolds} for the reader's convenience.

\begin{ass} \label{ass:pss}
There is a collection of oriented sc-Fredholm sections of strong polyfold bundles $\s_{\rm SFT}: \cB^\pm_{\rm SFT}(\g; A)\to \cE^\pm_{\rm SFT}(\g;A)$ and continuous maps ${\ov\ev^\pm: \cB^\pm_{\rm SFT}(\g;A)\to \overline{\C^\pm} \times M}$, indexed by $\g\in\cP(H)$ and $A\in H_2(M)$, with the following properties.
\begin{enumilist}
\item
The sections have Fredholm index $\text{\rm ind}(\s_{\rm SFT}) = CZ(\g) + 2c_1(A) + \tfrac{\dim  M}{2} + 2$ on $\cB^+_{\rm SFT}(\g; A)$, resp.\ $\text{\rm ind}(\s_{\rm SFT}) = - CZ(\g) + 2c_1(A) + \tfrac{\dim M}{2} + 2$ on $\cB^+_{\rm SFT}(\g; A)$. 
\item
Each zero set
$\bM^\pm_{\rm SFT}(\g; A):=\s_{\rm SFT}^{-1}(0)$ is compact, and given any $C\in\R$ there are only finitely many $A\in H_2(M)$ with $\omega(A)\leq C$ and nonempty zero set $\s_{\rm SFT}^{-1}(0)\cap \cB^\pm_{\rm SFT}(\g; A)\neq\emptyset$. 
\item 
The sections $\s_{\rm SFT}$ have tame sc-Fredholm representatives in the sense of \cite[Def.5.4]{Ben-fiber}, and the evaluation maps $\ov\ev^\pm$ restrict on the open subsets
$\cB^{\pm,\C}_{\rm SFT}(\g;A) := (\ov\ev^\pm)^{-1}(\C^\pm\times M) \subset \cB^\pm_{\rm SFT}(\g;A)$ to sc$^\infty$ maps $\ev^\pm : \cB^{\pm,\C}_{\rm SFT}(\g;A) \to \C^\pm\times M$, 
which are $\s_{\rm SFT}$-compatibly submersive in the sense of Definition~\ref{def:submersion}. Finally, this open subset contains the interior,
$\partial_0\cB^\pm_{\rm SFT}(\g;A)\subset\cB^{\pm,\C}_{\rm SFT}(\g;A)$.
\end{enumilist}
\end{ass}

\begin{rmk} \label{rmk:ass} \rm  
\begin{enumilist}
\item
The polyfolds, bundles, and sections in Assumption~\ref{ass:pss} are constructed for a closely analogous situation (considering curves in $\R\times Q$, with e.g.\ $Q=S^1\times M$) in \cite[\S 3]{fh-sft}, so -- while not needed for our proof -- we state the following properties for intuition:

\smallskip\noindent
{\it
Equivalence classes under reparametrization of ${\rm Aut}(\C^\pm,0)$ of smooth maps $v: \C^\pm \to \C^\pm\times M$ that satisfy $v(Re^{\pm it})=\bigl(Re^{\pm it}, \g(t)\bigr)$ for sufficiently large $R>1$ and represent the class $[v\#\Ti u_\g]= [\CP^1] + A$ form a dense subset $\cB^\pm_{\rm dense}(\g;A)\subset \cB^\pm_{\rm SFT}(\g;A)$ contained in the interior. 
On this subset, the section is $\s_{\rm SFT}([v]) = [(v, \dbar_{\Ti J_H^\pm} v)]$
and $\ov\ev^\pm([v])$ is evaluation as in \eqref{SFTeval}.
The intersection of $\s_{\rm SFT}^{-1}(0)$ with this dense subset is contained in the moduli space $\cM^\pm_{\rm SFT}(\g; A)$ from \eqref{SFTeval}.  The full moduli space $\cM^\pm_{\rm SFT}(\g; A)$ is obtained by enlarging $\cB^\pm_{\rm dense}(\g;A)$ to include equivalence classes with $\sup_{t\in S^1} d_{\C\times M}\bigl( v(Re^{\pm it}) , (Re^{\pm it}, \g(t)) \bigr)\to 0$ as $R\to\infty$.  However, only classes with specific exponential decay of this quantity and related derivatives are contained in $\cB^\pm_{\rm SFT}(\g;A)$.
}
\smallskip

\item The sc-smooth structure, sc-Fredholm property, and compactness is stated in \cite[Thm.3.4]{fh-sft}. The proof of polyfold and bundle structure outlined in \cite[\S 7--11]{fh-sft} extends the construction of Gromov-Witten polyfolds in \cite{hwz-gw} by local models for punctures and neck-stretching from \cite[\S 3]{fh-2}, using the implanting method in \cite[\S 3,\S 5]{fh-1}. These constructions automatically satisfy the tameness assumed in (iii). 
The nonlinear Fredholm property needs to be proven globally -- in close analogy to \cite{hwz-gw}. The Fredholm index stated in (i) is computed in a local chart, where the linearized section coincides with a restriction of the classical linearized Cauchy-Riemann operator to a local slice to the reparametrization action. The compactness properties follow from SFT-compactness of the moduli spaces \cite{sft-compactness} since the topology on the polyfolds given in \cite[\S 3.4]{fh-sft} generalizes the notion of SFT-convergence. 
Orientations are constructed in \cite[\S 15]{fh-sft}. 
Sc-smoothness of the evaluation maps is proven analogously to \cite[Thm.1.8]{hwz-gw}, and their submersion property in (iii), which is used to construct fiber products in Lemma~\ref{lem:PSS poly}, is proven as in \cite[Ex.5.1]{Ben-fiber}.

\item
We also expect the existence of a direct polyfold description of the moduli space \eqref{HOLmod} in terms of a collection of sc-Fredholm sections $\s: \cB^\pm(\g; A)\to \cE^\pm(\g;A)$ with the same indices, and submersive sc$^\infty$ maps ${\ev^\pm: \cB^\pm(\g;A)\to M}$ with the following simplified properties. 

\smallskip\noindent
{\it
The smooth maps $u: \C \to M$ which equal $u(Re^{\pm it})= \g(t)$ for sufficiently large $R>1$ and represent the class $A$ form a dense subset of $\cB^\pm(\g; A)$ that is contained in the interior. On this subset, the section is $\s(u) =  \dbar_{J} u - Y_H(u)$, and the evaluation is $\ev^\pm(u)=u(0)$.
The intersection of $\s^{-1}(0)$ with this dense subset is contained in the moduli space $\cM^\pm(\g; A)$ from \eqref{HOLmod}.  
The full moduli space $\cM^\pm(\g; A)$ is obtained by enlarging the dense subset to include maps with $\sup_{t\in S^1} d_{M}\bigl( u(Re^{\pm it}) , \g(t) \bigr)\to 0$ as $R\to\infty$.  However, only maps with specific exponential decay of this quantity and related derivatives are contained in $\cB^\pm(\g;A)$.
}

\smallskip

While such a construction should follow from the same construction principles as in \cite{fh-sft}, there is presently no writeup beyond \cite{W-Fred}, which proves the Fredholm property in a model case. Alternatively, one could abstractly obtain this construction from restricting the setup in Assumption~\ref{ass:pss} to subsets consisting of maps of the form $v(z)=(z, u(z))$. 
Thus there would be no harm in using this property as intuitive guide for following our work with the abstract setup. 
\end{enumilist}
\end{rmk}

Given one or another polyfold description of the naturally identified moduli spaces \eqref{HOLmod} or \eqref{SFTeval} and corresponding evaluation maps, we will now extend the identifications \eqref{eq:Mfiber} or  \eqref{eq:MfiberSFT} to a fiber product construction of polyfolds which will contain these PSS/SSP moduli spaces. For $p\in{\rm Crit}(f)$, $\g\in\cP(H)$, and $A\in H_2(M)$ we define the topological spaces
\begin{align}
\tilde{\cB}^+(p,\g; A) &:= \bigl\{ (\ul{\t},\ul{v})\in\bM(p,M) \times \cB^+_{\rm SFT}(\g;A)  \,\big|\, ( 0 , \ev(\ul{\t}) ) = \ov\ev^+(\ul{v}) \bigr\} \nonumber\\
&\;= \bigl\{ (\ul{\t},\ul{v})\in\bM(p,M) \times \cB^{+,\C}_{\rm SFT}(\g;A)  \,\big|\, ( 0 , \ev(\ul{\t}) ) = \ev^+(\ul{v}) \bigr\}, \label{eq:Bfiber} \\
\tilde{\cB}^-(\g,p; A) &:= \bigl\{ (\ul{v},\ul{\t})\in  \cB^-_{\rm SFT}(\g;A) \times \bM(M,p)  \,\big|\, (0 , \ev(\ul{\t})) = \ov\ev^-(\ul{v}) \bigr\}  \nonumber \\
&\;= \bigl\{ (\ul{v},\ul{\t})\in  \cB^{-,\C}_{\rm SFT}(\g;A) \times \bM(M,p)  \,\big|\, (0 , \ev(\ul{\t})) = \ev^-(\ul{v}) \bigr\}.  \nonumber
\end{align}
We will use \cite{Ben-fiber} to equip these spaces with natural polyfold structures and show that the pullbacks of the sections $\sigma_{\rm SFT}$ by the projections to $\cB^\pm_{\rm SFT}(\g;A)$ yield sc-Fredholm sections whose zero sets are compactifications of the PSS/SSP moduli spaces. 
This will require a shift in levels which is of technical nature as each $m$-level $\cB_m\subset\cB$ contains the dense ``smooth level'' $\cB_\infty\subset\cB_m$, which itself contains the moduli space $\bM=\s^{-1}(0)\subset\cB_\infty$; see Remark~\ref{rmk:levels}.

\begin{lem} \label{lem:PSS poly}
For any $p\in{\rm Crit}(f)$, $\g\in\cP(H)$, and $A\in H_2(M)$ there exist open subsets $\cB^+(p,\g; A)  \subset \tilde{\cB}^+(p,\g; A)_1$ and $\cB^-(\g,p; A) \subset \tilde{\cB}^-(\g,p; A)_1$ which contain the smooth levels $\tilde{\cB}^\pm(\ldots;A)_\infty$ of the fiber products \eqref{eq:Bfiber} and inherit natural polyfold structures. 
The smooth level of their interior is\footnote{
Here we can only make statements about the smooth level because we do not know what points of other levels are included in the fiber products. This is sufficient for applications as the zero set of any sc-Fredholm section (and its admissible perturbations) is contained in the smooth level.
}
\begin{align*}
\partial_0\cB^+(p,\g; A)_\infty &\;=\; \cM(p,M) \; \leftsub{\{0\}\times\ev}{\times_{\ev^+}} \; \partial_0\cB_{\rm SFT}^{+,\C}(\g; A)_\infty  , \qqquad\\
\partial_0\cB^-(\g,p; A)_\infty &\;=\; \partial_0\cB_{\rm SFT}^{-,\C}(\g; A)_\infty \; \leftsub{\ev^-}{\times_{\{0\}\times\ev}}\;  \cM(M,p)  .
\end{align*}
Moreover, pullback of the sc-Fredholm sections of strong polyfold bundles $\sigma^\pm_{\rm SFT}: \cB^\pm_{\rm SFT}(\g;A) \to \cE^\pm_{\rm SFT}(\g; A)$ under the projection $\cB^\pm(\ldots; A) \to \cB^\pm_{\rm SFT}(\ldots; A)$ induces sc-Fredholm sections of strong polyfold bundles $\s^+_{(\g,p; A)}: \cB^+(\g,p;A) \to \cE^+(\g,p;A)$ resp.\  $\s^-_{(p,\g; A)}: \cB^-(p,\g;A) \to \cE^-(p,\g;A)$ of index $I(p,\g; A)$ resp.\ $I(\g,p; A)$ given in \eqref{PSS index}. 
Their zero sets contain\footnote{
As in Remark~\ref{rmk:ass}, this identification is stated for intuition and will ultimately not be used in our proofs. }
the moduli spaces from \S\ref{ssec:moduli},
\begin{align*}
{\s^+_{(p,\g; A)}}\!\!^{-1}(0) &\;=\; \bM(p,M) \;\leftsub{\{0\}\times\ev}{\times_{\ov\ev^+}} \; {\s_{\rm SFT}^+}^{-1}(0)  \;\supset\; \cM(p,\g; A), 
\qquad\\
{\s^-_{(\g,p; A)}}\!\!^{-1}(0) &\;=\;  {\s_{\rm SFT}^-}^{-1}(0)\; \leftsub{\ov\ev^-}{\times_{\{0\}\times\ev}} \; \bM(M,p) \;\supset\; \cM(\g,p; A). 
\end{align*}
Finally, each zero set ${\s^\pm_{(\ldots;A)}}\!\!^{-1}(0)$ is compact, and given any $p\in\Crit(f)$, $\g\in\cP(H)$, and $C\in\R$, there are only finitely many $A\in H_2(M)$ with  $\omega(A)\leq C$ and nonempty zero set ${\s^\pm_{(\ldots;A)}}\!\!^{-1}(0)\neq\emptyset$. 
\end{lem}

\begin{proof} 
We will follow \cite[Cor.7.3]{Ben-fiber} to construct the PSS polyfold, bundle, and sc-Fredholm section $\s^+_{p,\g;A}$ in detail, and note that the construction of the SSP section $\s^-_{\g,p;A}$ is analogous.

Consider an ep-groupoid representative $\cX = (X,{\bf X})$ of the polyfold $\cB^+_{\rm SFT}(\g;A)$ with source and target maps denoted $s,t : {\bf X} \rightarrow X$ together with a strong bundle $P : W \rightarrow X$ over the $M$-polyfold $X$ and a structure map $\mu : {\bf X} \leftsub{s}\times_{P} W \rightarrow X$ such that the pair $(P,\mu)$ is a strong bundle over $\cX$ representing the polyfold bundle $\cE^+_{\rm SFT}(\g;A) \rightarrow \cB^+_{\rm SFT}(\g;A)$. In addition, consider a sc-Fredholm section functor $S_{\rm SFT} : X \rightarrow W$ of $(P,\mu)$ that represents $\s^+_{\rm SFT}$. The ep-groupoid $\cX$ and the bundle $(P,\mu)$ are tame, since they represent a tame polyfold and a tame bundle, respectively. Moreover, $S_{\rm SFT}$ is a tame sc-Fredholm section in the sense of \cite[Def.5.4]{Ben-fiber} by Assumption~\ref{ass:pss}(iii). 

We view the Morse moduli space $\bM(p,M)$ as the object space of an ep-groupoid with morphism space another copy of $\bM(p,M)$ and with unit map a diffeomorphism; that is, the only morphisms are the identity morphisms. The unique rank-$0$ bundle over $\bM(p,M)$ is a strong bundle in the ep-groupoid sense, and the zero section of this bundle is a tame sc-Fredholm section functor.
Next, note that $\tilde{\cB}^+(p,\g; A) \subset \bigl\{ (\ul{\t},\ul{v})\in\bM(p,M) \times |X|  \,|\, \ov\ev^+(\ul{v})\in \{0\}\times M \bigr\} \subset \bM(p,M) \times |X^\ev|$
is represented within the open subset $X^\ev:=(\ov\ev^+)^{-1}(\C \times M)\subset X$ and the corresponding full ep-subgroupoid $\cX^\ev$ of $\cX$, which represent the open subset $\cB^{+,\C}_{\rm SFT}(\g,A)\subset |X|$, 
and by Assumption~\ref{ass:pss}(iii) the restricted evaluation $\ev^+ : X^\ev \rightarrow \C\times M$ is sc$^\infty$ and $S_{\rm SFT}$-compatibly submersive (see Definition~\ref{def:submersion}). 
Denote by $\ev_0: \bM(p,M) \rightarrow \C\times M, \ul\tau\mapsto (0,\ev(\ul\tau))$ the product of the trivial map to $0\in\C$ and the Morse evaluation map. 
We claim that the product map $\ev_0 \times \ev^+ : \bM(p,M) \times X^\ev \rightarrow (\C\times M) \times (\C\times M)$ is $S_{\rm SFT}$-compatibly transverse to the diagonal $\Delta \subset (\C \times M) \times (\C\times M)$. 
Indeed, given $(\ul{\t},\ul{v}) \in (\ev_0 \times \ev^+)^{-1}(\Delta)$ let
$L \subset \rT_{\ul{v}}^RX^{\ev}$ be a sc-complement of the kernel of the linearization of $\ev^+$ at some $\ul{v} \in X^{\ev}_{\infty}$ that satisfies the conditions for $S_{SFT}$-compatible submersivity
in Definition~\ref{def:submersion} w.r.t.\ a coordinate change $\psi^{ev}$ on a chart of $X^{\ev}$. 
Then the subspace $\{0\} \times L \subset \rT^R_{\ul{\t}}\overline{\cM}(p,M) \times \rT^R_{\ul{v}}X^{\ev}$ satisfies the conditions for $S_{\rm SFT}$-compatible transversality of $\ev_0 \times \ev^+$ with $\Delta$ at $(\ul{\t},\ul{v})$ w.r.t.\ the product change of coordinates $\id\times\psi^\ev$ in a product chart on the Cartesian product $\overline{\cM}(p,M) \times X^{\ev}$.
(See \cite[Lem.7.1,~7.2]{Ben-fiber} for a discussion of the sc-Fredholm property on Cartesian products.)

Next, note that $\bM(p,x) \leftsub{\ev_0}{\times_{\ev^+}} X^\ev_{\infty}$ represents the smooth level of the fiber product topological space $\tilde{\cB}^+(p,\g; A)$. 
So \cite[Cor.7.3]{Ben-fiber} yields an open neighbourhood $X' \subset \bM(p,M) \leftsub{\ev_0}{\times_{\ev^+}} X^\ev_1$ containing the smooth level $\bM(p,x) \leftsub{\ev_0}{\times_{\ev^+}} X^\ev_{\infty}$ such that the full subcategory $\cX' := (X',{\bf X}')$ of $\bM(p,M) \times \cX^\ev_1$ is a tame ep-groupoid and the pullbacks of $(P,\mu)$ and $S_{\rm SFT}$ to $\cX'$ are a tame bundle and tame sc-Fredholm section.
Here we used the fact that the smooth level $\bM(p,x)_\infty=\bM(p,x)$ of any finite dimensional manifold is the manifold itself; see Remark~\ref{rmk:levels}.

The tame ep-groupoid $\cX'$ yields the claimed polyfold $\cB^+(p,\g;A):=|\cX'|$, and similarly the pullbacks of $(P,\mu)$ and $S_{\rm SFT}$ through the projection $X' \rightarrow X_1$ define the claimed bundle and sc-Fredholm section $\s^+_{(p,\g;A)} : \cB^+(p,\g;A) \rightarrow \cE^+(p,\g;A)$. The identification of the interior $\partial_0\cB^+(p,\g;A)_\infty$ follows from the degeneracy index formula $d_{\cX'}(x_1,x_2) = d_{\bM(p,M)}(x_1) + d_{\cX}(x_2)$ in \cite[Cor.7.3]{Ben-fiber} and the interior of the Morse trajectory spaces $\partial_0 \bM(p,M) = \cM(p,M)$ from Theorem~\ref{thm:Morse}. 

The index formula in \cite[Cor.7.3]{Ben-fiber} yields $\text{\rm ind}(\s^+_{(p,\g;A)}) = \text{\rm ind}(\s_{\rm SFT}) + |p| - \dim(\C\times M) = I(p,\g;A)$ since $\dim \bM(p,M) = |p|$ and $\text{\rm ind}(\s_{\rm SFT}) = CZ(\g) + 2c_1(A) + \frac 12 \dim M + 2$. 

Finally, the zero set ${\s^+_{(p,\g;A)}}^{-1}(0)$ is 
the fiber product of the zero sets as claimed, as these are contained in the smooth level, and 
the restriction to $\ov\ev^{-1}(\{0\}\times M)$ already restricts considerations to the domain $X^\ev$ from which the fiber product polyfold is constructed. 
Moreover, ${\s^+_{(p,\g;A)}}^{-1}(0)$ is compact as in \cite[Cor.7.3]{Ben-fiber}, since both $\bM(p,M)$ and ${\s^+_{\rm SFT}}^{-1}(0)$ are compact
and both $\ev_0$ and $\ov\ev^+$ are continuous. The final statement then follows from Assumption~\ref{ass:pss}(ii).
\end{proof}

\begin{table}
\centering
 \begin{tabular}{||c | c | c ||}  
 \hline
 Notation & Description & Definition\\ [0.5ex] 
 \hline\hline
 $\cB^\pm_{\rm dense}(\g;A)$  & \thead{elements are equivalence classes under reparameterization by ${\rm Aut}(\C^\pm,0)$ \\ of smooth maps $v: \C^\pm \to \C^\pm\times M$ that satisfy $v(Re^{\pm it})=\bigl(Re^{\pm it}, \g(t)\bigr)$ \\ for sufficiently large $R>1$ and represent the class $[v\#\Ti u_\g]= [\CP^1] + A$ }
 & Remark~\ref{rmk:ass} \\
 \hline
$\cB^\pm_{\rm SFT}(\g; A) $ &  \thead{
a polyfold with dense subset $\cB^\pm_{\rm dense}(\g;A)$, which contains the \\ SFT-compactification $\bM^\pm_{\rm SFT}(\g; A)$ 
of the moduli space in \eqref{SFTeval}} 
& \thead{Assumption~\ref{ass:pss}, \\ Remark~\ref{rmk:ass}} \\ 
 \hline
$\cB^{\pm,\C}_{\rm SFT}(\g;A)$  &  \thead{the open subset of $\cB^\pm_{\rm SFT}(\g; A)$ containing the curves whose \\ evaluation   at 
a marked point lands in $\mathbb{C}^{\pm} \times M$
 rather than \\ in a broken off cylinder $\R\times S^1\times M$; see Remark~\ref{rmk:evaluation drama} }
 & \thead{Assumption~\ref{ass:pss}(iii)
} \\
 \hline
$\tilde{\cB}^+(p,\g; A)$  & \thead{elements are pairs of a half-infinite broken Morse trajectory \\ starting from the critical point $p$ and a curve in $\cB^{+,\C}_{\rm SFT}(\g;A)$, \\ 
whose evaluation agrees with the end point of the Morse trajectory}
&  \eqref{eq:Bfiber} \\
 \hline
$ \tilde{\cB}^-(\g,p; A)$ & \thead{elements are pairs of a half-infinite broken Morse trajectory \\ ending at the critical point $p$ and a curve in $\cB^{-,\C}_{\rm SFT}(\g;A)$, whose \\ 
evaluation agrees with the starting point of the Morse trajectory} 
 & \eqref{eq:Bfiber}\\
 \hline
 $\cB^+(p,\g; A)$  & \thead{open subset of $\tilde{\cB}^+(p,\g; A)_1$ containing $\cM(p,\g; A)$  \\ over which the section $\s^+_{(p,\g;A)}$ is sc-Fredholm \\
 (possibly smaller than $\tilde{\cB}^+(p,\g; A)_1$ due to shrinking in \cite[Cor.7.3]{Ben-fiber})}& Lemma~\ref{lem:PSS poly}\\
 \hline
$ \cB^-(\g,p; A)$ & \thead{open subset of $\tilde{\cB}^-(\g,p; A)_1$ containing $\cM(\g,p; A)$ \\ over which the section $\s^-_{(p,\g;A)}$ is sc-Fredholm}  & Lemma~\ref{lem:PSS poly} \\ [1ex] 
 \hline
\end{tabular}
\caption{
Summary of the polyfolds and their subsets introduced in this section
}
\label{tab:asspolyfolds}
\end{table}

\subsection{Construction of the morphisms} \label{ssec:construct}

To construct the $\L$-linear maps PSS and SSP in Theorem~\ref{thm:main} with relatively compact notation we index all moduli spaces from \S\ref{ssec:moduli} by the two sets 
\begin{align*}
\cI^+ &\,:=\; \bigl\{ \alpha=(p,\g;A) \,\big|\, 
p\in{\rm Crit}(f) , \g\in\cP(H) , A\in H_2(M) \bigr\}, \\
\cI^- &\,:=\;  \bigl\{ \alpha=(\g,p;A) \,\big|\, p\in{\rm Crit}(f) , \g\in\cP(H) , A\in H_2(M) \bigr\} .
\end{align*}
To simplify notation we then denote $\cI:=\cI^-\cup\cI^+$ and drop the superscripts from the polyfolds $\cB(\alpha)=\cB^\pm(\alpha)$. 
Since Lemma~\ref{lem:PSS poly} provides each moduli space $\cM(\alpha)$ for $\alpha\in \cI$ with a compactification and polyfold description $\cM(\alpha)\subset\sigma_{\alpha}^{-1}(0)$, we can apply \cite[Theorems~18.2,18.3,18.8]{HWZbook} to obtain admissible regularizations of the moduli spaces, and counts of the 0-dimensional perturbed solution spaces \cite[\S15.4]{HWZbook}, in the following sense. Here we denote by $\Q^+:=\Q\cap[0,\infty)$ the groupoid with only identity morphisms.

\begin{cor}  \label{cor:regularize}
\begin{enumilist}
\item
For every $\alpha\in \cI$, choice of neighbourhood of the zero sets $\sigma_{\alpha}^{-1}(0)\subset \cV_{\alpha}\subset \cB(\alpha)$, and choice of sc-Fredholm section functor $S_{\alpha}: \cX_{\alpha} \to \cW_{\alpha}$ representing $\sigma_{\alpha}|_{\cV_{\alpha}}$, there exists a pair $(N_\alpha,\cU_\alpha)$ controlling compactness in the sense of Definition~\ref{def:control} with $|S_{\alpha}^{-1}(0)|\subset |\cU_{\alpha}| \subset \cV_{\alpha}$. 

For $\alpha\in\cI$ with $\sigma_\alpha^{-1}(0)=\emptyset$ we can choose $\cU_{\alpha}=\emptyset$.
\item
For every collection $(N_{\alpha}, \cU_{\alpha})_{\alpha\in \cI}$ of 
pairs controlling compactness, there exists a collection $\underline\kappa = \bigl( \kappa_{\alpha}:\cW_{\alpha}\to\Q^+\bigr)_{\alpha \in\cI}$ of $(N_{\alpha}, \cU_{\alpha})$-admissible sc$^+$-multisections in the sense of \cite[Definitions~13.4,15.5]{HWZbook}
that are in general position relative to $(S_\alpha)_{\alpha\in\cI}$ in the sense that each pair $(S_\alpha,\kappa_\alpha)$ is in general position as per \cite[Def.15.6]{HWZbook}. 

Here admissibility in particular implies $\kappa_\alpha\circ S_\alpha |_{\cX_\alpha\less\cU_\alpha} \equiv 0$ and thus $\kappa_\alpha\circ S_\alpha\equiv 0$ when $\sigma_\alpha^{-1}(0)=\emptyset$.
\item
Every collection $\underline\kappa$ of admissible sc$^+$-multisections in general position  from (ii) induces a collection of compact, tame, branched ep$^+$-groupoids
$\bigl(\kappa_{\alpha}\circ S_{\alpha}: \cX_{\alpha} \to \Q^+\bigr)_{\alpha\in\cI}$. In particular, each perturbed zero set 
$$
Z^{\underline\kappa}(\alpha)\,:=\; \bigl|\{x\in X_{\alpha} \,|\, \kappa_{\alpha}(S_{\alpha}(x))>0\} \bigr|  \;\subset\; |\cU_\alpha| _\infty \;\subset\; |\cX_{\alpha}|_\infty  \;\cong\; \cB(\alpha)_\infty 
$$ 
is compact, contained in the smooth level, 
and carries the structure of a weighted branched orbifold of dimension $I(\alpha)$ as in \eqref{PSS index}. 
Moreover, the inclusion in $|\cU_\alpha|$ and general position of $\ul\kappa$ implies that for $I(\alpha)<0$ or $\sigma_\alpha^{-1}(0)=\emptyset$ the perturbed zero set $Z^{\underline\kappa}(\alpha)=\emptyset$ is empty. 
\item
For $\alpha\in\cI$ with Fredholm index $I(\alpha)=0$ and $\kappa_\alpha: \cW_{\alpha}\to\Q^+$ as in (ii) the perturbed zero set is contained in the interior $Z^{\underline\kappa}(\alpha)\subset\partial_0\cB(\alpha)_\infty$ and yields a well defined count 
$$\textstyle
\# Z^{\underline\kappa}(\alpha) \,:=\; \sum_{|x|\in Z^{\underline\kappa}(\alpha)} 
\; o_{\s_\alpha}(x) \; \kappa_{\alpha}(S_{\alpha}(x))
\;\;\in\;\Q .
$$ 
Here $o_{\s_\alpha}(x)\in \{\pm 1\}$ is determined by the orientation of $\s_\alpha$ as in \cite[Thm.6.3]{HWZbook}.
If $|\cU_\alpha|\cap\partial\cB(\alpha)=\emptyset$ then this count is independent of the choice of admissible sc$^+$-multisection $\k_\alpha$. 
\item
For every $\alpha\in\cI$ with Fredholm index $I(\alpha)=1$ and $\kappa_\alpha: \cW_{\alpha}\to\Q^+$ as in (ii) the boundary of the perturbed zero set is given by its intersection with the first boundary stratum of the polyfold, 
$$
\partial Z^{\underline\kappa}(\alpha) \;=\;  Z^{\underline\kappa}(\alpha)  \cap \partial_1\cB(\alpha)_\infty . 
$$
With orientations $o_{\s_\alpha|_\partial\cB(\alpha)}(x)\in \{\pm 1\}$ induced by the boundary restriction $\sigma_\alpha|_{\cB(\alpha)}$ this implies
$$\textstyle
\# \partial Z^{\underline\kappa}(\alpha) \,=\; \sum_{|x|\in \partial Z^{\underline\kappa}(\alpha)} 
\; o_{\s_\alpha|_{\partial\cB(\alpha)}}(x)\;\kappa_{\alpha}(S_{\alpha}(x))
\;\;=\; 0. 
$$ 
\end{enumilist}
\end{cor}

\begin{rmk} \label{rmk:PSS orient} \rm
\begin{enumilist}
\item
The statements in (iv) and (v) of Corollary~\ref{cor:regularize} require orientations of the sections $\s_\alpha$ for $\alpha\in\cI$. By the fiber product construction in Lemma~\ref{lem:PSS poly} they do indeed inherit orientations from the orientations of the Morse trajectory spaces in Remark~\ref{rmk:MorseOrient}, the orientations of $\s_{\rm SFT}^\pm$ given in Assumption~\ref{ass:pss}, and an orientation convention for fiber products. 

In practice, we will construct the perturbations $\ul\k$ in Corollary~\ref{cor:regularize} by pullback of perturbations $\ul \lambda=(\lambda^\pm_{\g,A})_{\g\in\cP, A\in H_2(M)}$ of the oriented SFT-sections $\s_{\rm SFT}^\pm$. Thus it suffices to specify the orientations of the regularized zero sets, which is implicit in their identification with transverse fiber products of oriented spaces over the oriented manifold $M$, 
\begin{align*}
Z^{\ul\k}(p,\g; A) \;=\; \bM(p,M) \; \leftsub{\ev_0}{\times_{\ev^+}} \; Z^{\ul\l}(\g;A)  , \qquad
Z^{\ul\k}(\g,p; A) \;=\; Z^{\ul\l}(\g; A) \; \leftsub{\ev^-}{\times_{\ev_0}}\;  \bM(M,p)  .
\end{align*}
Orientations of the boundary restrictions in (v) are then induced by the orientations of $Z^{\ul\k}(\alpha)$, via oriented isomorphisms of the tangent spaces $\R\nu(z) \times \rT_z \partial Z^{\ul\k}(\alpha) \cong  \rT_z  Z^{\ul\k}(\alpha)$, where $\nu(z)\in  \rT_z  Z^{\ul\k}(\alpha)$ is an exterior normal vector at $z\in \partial Z^{\ul\k}(\alpha)$. 

\item
Note that the counts in part (iv) of this Corollary may well depend on the choice of the multi-valued perturbations $\kappa_\alpha$ -- unless the ambient polyfold has no boundary, $\partial\cB(\alpha)=\emptyset$. 
Indeed, although the moduli space $\cM(\alpha)$ is expected to have dimension $0$, it may not be cut out transversely from the ambient polyfold $\cB(\alpha)$, and moreover it may not be compact. Assumption~\ref{ass:pss} provides an inclusion in a compact set $\cM(\alpha)\subset\sigma_{\alpha}^{-1}(0)$, and the perturbation theory for sc-Fredholm sections of strong bundles then associates to $\sigma_{\alpha}^{-1}(0)$ a perturbed zero set $Z^{\underline\kappa}(\alpha)\subset \cB(\alpha)$ with weight function $\kappa_\alpha\circ S_\alpha:  Z^{\underline\kappa}(\alpha) \to \Q\cap (0,\infty)$. This process generally adds points on the boundary $\sigma_{\alpha}^{-1}(0)\less \cM(\alpha)\subset \cB(\alpha) \less \partial_0\cB(\alpha)$, which may or may not persist under variations of the perturbation $\kappa_\alpha$. 
\end{enumilist}
\end{rmk}

The following construction of morphisms will depend on the choices of perturbations and orientation convention (see the previous remark) as well as geometric data fixed in \S\ref{ssec:moduli}, and possibly the choice of polyfold construction in Assumption~\ref{ass:pss} and ep-groupoid representation in Remark~\ref{rmk:polyfolds}. 
The algebraic properties in Theorem~\ref{thm:main} will be achieved in \S\ref{sec:algebra} -- for any given choice of geometric data -- by particular choices of ep-groupoids and perturbations $\underline\kappa^\pm$, and an overall sign adjustment.

\begin{dfn} \label{def:PSS}
Given collections $\underline\kappa^\pm=( \kappa^\pm_{\alpha} )_{\alpha \in\cI^\pm}$ of admissible sc$^+$-multisections in general position as in Corollary~\ref{cor:regularize}, we define the maps $PSS_{\underline\kappa^+}:CM \to CF$ and $SSP_{\underline\kappa^-}: CF \to CM$ to be the $\Lambda$-linear extension of
$$
PSS_{\underline\kappa^+} \la p \ra := \hspace{-5mm} \sum_{\begin{smallmatrix}
\scriptstyle  \g, A  \\ \scriptscriptstyle I(p,\g;A)=0  \end{smallmatrix} }  \hspace{-4mm}
\# Z^{\underline\kappa^+}(p,\g;A)  \cdot T^{\o(A)} \la \g \ra , \qquad
SSP_{\underline\kappa^-} \la \g \ra := \hspace{-5mm}\sum_{\begin{smallmatrix}
\scriptstyle  p, A \\ \scriptscriptstyle I(\g,p;A)=0  \end{smallmatrix} }  \hspace{-4mm}
\# Z^{\underline\kappa^-}(\g,p;A) \cdot T^{\o(A)} \la p \ra.
$$
\end{dfn}

\begin{lem} \label{lem:pss-novikov}
The maps $PSS_{\underline\kappa^+}:CM \to CF$ and $SSP_{\underline\kappa^-}: CF \to CM$ in Definition~\ref{def:PSS} are well defined, i.e.\ the coefficients take values in the Novikov field $\Lambda$ defined in \S\ref{sec:Novikov}.
\end{lem}
\begin{proof}
To prove that $PSS_{\underline\kappa^+}$ is well defined we need to check finiteness of the following set for any $p\in\Crit(f)$, $\g\in\cP(H)$, and $c\in\R$,
$$\textstyle
 \Bigl\{ r\in \omega(H_2(M)) \cap (-\infty,c]  \;\Big|\;  \sum_{\begin{smallmatrix}
\scriptstyle  A\in H_2(M) \\ \scriptstyle \omega(A)=r
 \end{smallmatrix}} 
 \# Z^{\underline\kappa^+}(p,\g;A) \neq 0  \Bigr\}  .
$$
Here $\o : H_2(M) \rightarrow \mathbb{R}$ is given by pairing with the symplectic form on $M$, and recall from Lemma~\ref{lem:PSS poly} that there are only finitely many homology classes $A\in H_2(M)$ with $\omega(A) \leq c$ and $\sigma_\alpha^{-1}(0)\neq\emptyset$.
On the other hand, the perturbations $\ul\kappa^+$ were chosen in Corollary~\ref{cor:regularize} (iii),(iv) so that $\# Z^{\underline\kappa^+}(\ldots; A) = 0$ whenever $\sigma_\alpha^{-1}(0)=\emptyset$. Thus there are in fact only finitely many $A\in H_2(M)$ with $\omega(A) \leq c$ and $\# Z^{\underline\kappa^+}(\ldots; A) \neq 0$, which proves the required finiteness.
The proof for $SSP_{\underline\kappa^-}$ is analogous. 
\end{proof}

\section{The chain homotopy maps}
\label{sec:iota,h}

In this section we construct $\L$-linear maps $\iota:CM \to CM$ and $h:CM \to CM$ on the Morse complex over the Novikov field $\Lambda$ given in \eqref{eq:CM}, which appear in Theorem~\ref{thm:main}. 
For that purpose we again fix a choice of geometric data as in \S\ref{ssec:moduli} to construct moduli spaces in \S\ref{ssec:iota-mod} and \S\ref{ssec:h-mod}. We equip these with polyfold descriptions in \S\ref{ssec:polyh}, and define the maps $\iota, h$ for admissible regular choices of perturbations in Definitions~\ref{def:iota,h}. 
To obtain the algebraic properties claimed in Theorem~\ref{thm:main}~(i)--(iii) we will then construct particular ``coherent'' choices of perturbations in \S\ref{sec:algebra}.

\subsection{Moduli spaces for the isomorphism $\boldsymbol{\iota}$}\label{ssec:iota-mod}

We will construct $\iota:CM \to CM$ from the following moduli spaces for critical points $p_-,p_+\in{\rm Crit}(f)$,  $A \in H_2(M)$, using the almost complex structure $J$ and the unstable/stable manifolds (see \S\ref{ssec:Morse}) of the Morse-Smale pair $(f,g)$ chosen in \S\ref{ssec:moduli}, 
\begin{equation} \label{iotamod}
\cM^\iota(p_-,p_+; A) := \bigl\{ u: \CP^1 \to M \,\big|\, u([1:0])\in W^-_{p_-} , \; u([0:1])\in W^+_{p_+} , \; \dbar_J u = 0 , \; [u]=A \bigr\} . 
\end{equation}
Note that a cylinder acts on this moduli space by reparametrization with biholomorphisms of $\CP^1$ that fix the two points $[1:0],[0:1]$. However, we do not quotient out this symmetry so describe these moduli spaces as the zero set of a Fredholm section over a Sobolev closure of the set of smooth maps $u: \CP^1 \to M$ in the homology class $[u]=A$ satisfying the point constraints $u([1:0])\in W^-_{p_-}$ and $u([0:1])\in W^+_{p_+}$.
This determines the Fredholm index as 
\begin{equation}\label{iota index}
I^\iota(p_-,p_+;A) \; =\;  2c_1(A) + |p_-| - |p_+|.
\end{equation}
As in \S\ref{ssec:poly} we will obtain a compactification and polyfold description of this moduli space by identifying it with a fiber product of Morse trajectory spaces and a space of pseudoholomorphic curves, in this case the space of parametrized $J$-holomorphic spheres with evaluation maps for $z_0\in\CP^1$, 
$$
\ev_{z_0} \,:\; \cM(A) := \bigl\{  u : \CP^1 \rightarrow  M \,\big|\, \dbar_J u =0  , [u]=A \bigr\} 
\;\to\; M , \qquad u \mapsto u(z_0). 
$$ 
With this we can describe the moduli space \eqref{iotamod} as fiber product with the half-infinite Morse trajectory spaces from \S\ref{ssec:Morse}, using $z_0^+:=[1:0]$ and $z_0^-:=[0:1]$
\begin{equation}\label{iota-fiber}
\cM^\iota(p_-,p_+; A) \; \cong\;  \cM(p_-,M) \; \leftsub{\ev}{\times}_{\ev_{z^+_0}} \; \cM(A) \; \leftsub{\ev_{z^-_0}}{\times}_{\ev} \; \cM(M,p_+).
\end{equation}
Note here that we are not working with a Gromov-Witten moduli space, as we do not quotient by ${\rm Aut}(\CP^1)$. 
This is due to the chain homotopy in Theorem~\ref{thm:main}~(iii), which will result from identifying a compactification of $\cM(A)$ with a boundary of the neck-stretching moduli space $\cM_{\rm SFT}(A)$ in \eqref{SFTstretchmod} that appears in Symplectic Field Theory \cite{egh}.
For that purpose we identify a solution $u:\CP^1\to M$ with the map to its graph $\Ti u : \CP^1 \to \CP^1 \times M, z \mapsto (z,u(z))$ as in \cite[\S8.1]{MS}. This yields is a bijection (and homeomorphism in appropriate topologies) 
$$
\cM(A) \;\cong\; \frac{\widetilde\cM_{\rm GW}([\CP^1]+ A) := \bigl\{ v : \CP^1 \to \CP^1\times M  \;\big|\;  \dbar_{\Ti J} v = 0 , \;  [v]=[\CP^1]+ A \bigr\}}{{\rm Aut}(\CP^1)}
$$
between the Cauchy-Riemann solution space for $M$ and the Gromov-Witten moduli space  for $\CP^1\times M$ in class $[\CP^1] + A$ for the split almost complex structure $\Ti J:=i \times J$ on $\CP^1\times M$. To transfer the evaluation maps at $z^+_0=[1:0]$ and $z_0^-=[0:1]$ we keep track of these as (unique) marked points mapping to $\{z_0^\pm\}\times M$ and thus replace \eqref{iota-fiber} by a fiber product over $\CP^1\times M$, 
\begin{equation}\label{iota-fiber-SFT}
\cM^\iota(p_-,p_+; A) \; \cong\;  \cM(p_-,M) \; \leftsub{\{z^+_0\}\times\ev}{\times}_{\ev^+} \; \cM_{\rm GW}(A) \; \leftsub{\ev^-}{\times}_{\{z^-_0\}\times\ev} \; \cM(M,p_+).
\end{equation}
This uses the evaluation maps from a Gromov-Witten moduli space with two marked points,
\begin{equation}
\ev^\pm \;:\; \cM_{\rm GW}(A) := \quotient{\widetilde\cM_{\rm GW}([\CP^1]+ A)}{{\rm Aut}(\CP^1,z_0^-,z_0^+)}  
\;\to\; \CP^1\times M , \qquad
[v] \;\mapsto\;  v(z_0^\pm)  ,  \label{eq:GWeval}
\end{equation}
where ${\rm Aut}(\CP^1,z_0^-,z_0^+)$ denotes the set of biholomorphisms $\phi:\CP^1\to\CP^1$ which fix $\phi(z_0^\pm)=z_0^\pm$. 
The polyfold setup in \cite[Theorems~1.7,1.10,1.11]{hwz-gw} for Gromov-Witten moduli spaces now provides a strong polyfold bundle $\cE_{\rm GW}(A)\to \cB_{\rm GW}(A)$, and oriented sc-Fredholm section $\s_{\rm GW}: \cB_{\rm GW}(A)\to \cE_{\rm GW}(A)$ that cuts out a compactification $\bM_{\rm GW}(A)=\s_{\rm GW}^{-1}(0)$ of $\cM_{\rm GW}(A)$.
Here a dense subset of the base polyfold $\cB_{\rm GW}(A)$ consists of ${\rm Aut}(\CP^1,z_0^-,z_0^+)$-orbits of smooth maps $v: \CP^1 \to \CP^1\times M$ in the homology class $[v]=[\CP^1]+A$, which implicitly carries the two marked points $z_0^\pm\in\CP^1$. Nodal curves in $\bM_{\rm GW}(A)$ then explicitly come with the data of two marked points on their domain. 
On the dense subset the section is given by $\s_{\rm GW}([v]) =[(v,\dbar_{\Ti J} v)]$. 
The setup in \cite[Theorem~1.8]{hwz-gw} moreover provides sc$^\infty$ evaluation maps $\ev^\pm : \cB_{\rm GW}(A) \to \CP^1\times M$ at the marked points, which on the dense subset are given by $\ev^\pm([v])=v(z^\pm_0)$. 

Thus we have given each factor in the fiber product \eqref{iota-fiber-SFT} a compactification\footnote{The term 'compactification' applied to spaces of pseudoholomorphic curves is always to be understood as Gromov-compactification, as $\cM_{\rm GW}(A)\subset \bM_{\rm GW}(A)$ may not be dense.} that is either a manifold with corners given by the compactified Morse trajectory spaces in Theorem~\ref{thm:Morse}, or the compact zero set $\bM_{\rm GW}(A)=\s_{\rm GW}^{-1}(0)$ of a sc-Fredholm section. 
In \S\ref{ssec:polyh} we will combine the polyfold description of the Gromov-compactification of \eqref{eq:GWeval} with an abstract construction of fiber products in polyfold theory \cite{Ben-fiber} to obtain compactifications and polyfold descriptions of the moduli spaces. Then the construction of $\iota:CM\to CM$ proceeds as in \S\ref{ssec:construct}. The algebraic properties of $\iota$ in Theorem~\ref{thm:main}~(i) and (ii) will follow from the boundary stratifications of the Morse trajectory spaces $\bM(p_-,M)$ and $\bM(M,p_+)$ since the ambient polyfold $\cB_{\rm GW}(A)$ has no boundary. However, this requires specific ``coherent'' choices of perturbations in \S\ref{sec:algebra}.

\begin{rmk} \rm \label{rmk:iota-energy}
Gromov-compactifications of the moduli spaces $\cM^\iota(p_-,p_+; A)$ will result from the energy identity \cite[Lemma~2.2.1]{MS} for solutions of $\dbar_J u = 0$,
\begin{equation} \label{eq:iota-energy}
\textstyle
E(u) \,:=\; \frac 12 \int_\C |\rd u |^2  \;=\; \int_{\CP^1} u^*\omega \;=\; \omega([u])  . 
\end{equation}
This fixes the energy of solutions on each solution space $\cM(A)$, and Gromov compactness asserts that $\{u : \CP^1 \to M \,| \, \dbar_J u =0 , E(u)\leq C  \}$ is compact up to bubbling for any $C>0$.

Another consequence of \eqref{eq:iota-energy} is that for $\omega(A)\leq 0$ we have no solutions $\cM(A)=\emptyset$ except for $A=0\in H_2(M)$ when the solution space is the space of constant maps 
$$
\cM(0) \;=\; \{ u\equiv x \,|\, x\in M\} \;\simeq\; M, 
$$
which is compact and cut out transversely. 

Translated to graphs in $\CP^1\times M$ with two marked points, this means $\bM_{\rm GW}(0)\simeq \CP^1\times\CP^1\times M$ by adding two marked points in the domain. That is, $(z^-, z^+,x)\in \CP^1\times\CP^1\times M$ corresponds to the (equivalence class of) graphs $\Ti u_x: z\mapsto (z,x)$ with two marked points $z^-,z^+\in\CP^1$. For $z^-\neq z^+$ this tuple can be reparametrized to the fixed marked points $z^-_0,z^+_0\in\CP^1$ and then represents an ${\rm Aut}(\CP^1,z^-_0,z^+_0)$-orbit. 
For $z^-=z^+$ the tuple $(z^-, z^+,x)$ corresponds to a stable map in $\bM_{\rm GW}(0)$, given by the graph $\Ti u_x$ with a node at $z^-=z^+$ attached to a constant sphere with two distinct marked points. 
This will be stated in polyfold terms in Assumption~\ref{ass:iota,h}~(ii). 
\end{rmk}

\subsection{Moduli spaces for the chain homotopy $\boldsymbol{h}$} \label{ssec:h-mod}

To construct the moduli spaces from which we will obtain $h:CM \to CM$, we again use the almost complex structure $J$ and Morse-Smale pair $(f,g)$ chosen in \S\ref{ssec:moduli}. In addition, we fixed an anti-holomorphic vector-field-valued 1-form $Y_H\in \Om^{0,1}(\C, \G(\rT M))$ that arises from the fixed Hamiltonian function $H:S^1\times M\to\R$ and a choice of smooth cutoff function $\beta: [0,\infty) \to [0,1]$ with $\beta|_{[0,1]}\equiv 0$, $\beta'\geq 0$, and $\beta|_{[e,\infty)}\equiv 1$.
Gluing this 1-form to another copy of $Y_H$ over $\C^-$ with neck length $R>0$ in exponential coordinates yields the anti-holomorphic vector-field-valued 1-form $Y_H^R\in \Om^{0,1}(\CP^1, \G(\rT M))$ that vanishes near $[1:0],[0:1]$ and on $\CP^1\less\{[1:0],[0:1]\} = \{ [1:r e^{i\theta}] \,|\, (r,\theta)\in (0,\infty) \times S^1 \}$ is given by 
$$
Y^R_H([1:re^{i\theta}],x) := \tfrac 12 \beta_R(r) \bigl( J X_H(\theta,x) \, r^{-1} \rd r + X_H(\theta,x) \, \rd\theta  \bigr) .
$$
Here $\beta_R(r):= \beta(re^{\frac R2}) \beta(r^{-1} e^{\frac R2})$ is a smooth cutoff function $\beta_R: (0,\infty) \to [0,1]$ that is identical to $1$ on $[e^{1-\frac{R}{2}} , e^{\frac{R}{2}-1}]$ and identical to $0$ on $(0,e^{-\frac R2}) \cup (e^{\frac R2},\infty)$.  
Now perturbing the Cauchy-Riemann operator on $\CP^1$ by $Y_H^R$ yields the following moduli spaces for critical points $p_-,p_+\in{\rm Crit}(f)$,  $A \in H_2(M)$, and $R \in [0,\infty)$,
$$
\cM_R(p_-,p_+; A) := \bigl\{ u: \CP^1 \to M \;\big|\; u([1:0])\in W^-_{p_-} , \; u([0:1])\in W^+_{p_+} , \;\dbar_{J} u = Y_H^R(u) , \; [u]=A \bigr\}, 
$$
and we will construct $h$ from their union
\begin{equation} \label{hmod} \textstyle
\cM(p_-,p_+;A) := \bigsqcup_{R \in [0,\infty)} \cM_R(p_-,p_+;A).
\end{equation}

\begin{rmk} \rm 
Each vector-field-valued 1-form $Y_H^R=-(X_{H^R_\beta})^{0,1}$ is in the notation of \cite[\S8.1]{MS} induced from the 1-form with values in smooth functions $H^R_\beta\in\Om^1(\CP^1,\cC^\infty(M))$ given by $H^R_\beta (r e^{i\theta}) = \beta_R(r) H (\theta,\cdot) \rd\theta$.
It is constructed so that it has the following properties:
\begin{enumilist}
\item
For $R=0$ we have $Y^0_H\equiv 0$ so that the moduli space $\cM_0(p_-,p_+; A)=\cM^\iota(p_-,p_+; A)$ is the same moduli space \eqref{iotamod} from which $\iota$ will be constructed.

\item
The restriction of any solution $u\in\cM_R(p_-,p_+;A)$ to the middle portion $\{[1:z]\in\CP^1 \,|\, e^{1-\frac R2} < |z|< e^{\frac R2-1}\}\cong (1-\frac R2,\frac R2-1)\times S^1$ satisfies the Floer equation $\partial_s v + J \partial_t v = J X_H(t,v)$  after reparametrization $v(s,t):= u([1:e^{s+it}])$.

\item
The shifts $u_-(z):=u([1: e^{- \frac R2} z])$ and $u_+(z):=u([e^{\frac R2} z : 1])=u([1: e^{\frac R2} z^{-1} ])$ of any solution $u\in\cM_R(p_-,p_+;A)$, restricted to $\{z\in \C \,|\, |z| < e^{R-1} \}$,
satisfy $\dbar_{J} u_\pm = Y_H(u_\pm)$ as in the PSS/SSP moduli spaces in \S\ref{ssec:moduli}. 
\end{enumilist}
\end{rmk}

The moduli space $\cM(p_-,p_+;A)$ is the zero set of a Fredholm section over a Banach manifold $[0,\infty) \times \cB$, where $\cB$ is the same Sobolev closure as in \S\ref{ssec:iota-mod} of the set of smooth maps $u: \C\mathbb{P}^1 \to M$ in the homology class $[u]=A$ satisfying the point constraints $u([1:0])\in W^-_{p-}$ and  $u([0:1]) \in W^+_{p_+}$.
Restricted to $\{0\}\times\cB$ this is the Fredholm section that cuts out $\cM^\iota(p_-,p_+;A)$ in \eqref{iotamod} with $\Ti J_H^0=\Ti J$. 
This determines the Fredholm index as
\begin{equation}\label{h index}
I(p_-,p_+;A) \, :=\; I^\iota(p_-,p_+;A) + 1 \; =\; 2c_1(A) + |p_-| - |p_+| + 1.
\end{equation}
Towards a compactification and polyfold description of these moduli spaces we again -- as in \S\ref{ssec:poly}, \S\ref{ssec:iota-mod}, \cite[\S8.1]{MS} -- identify a solution $u:\CP^1\to M$ with the map to its graph. 
Moreover, we again fix marked points $z_0^+=[1:0]$, $z_0^-=[0:1]$ to implement evaluation maps to express the conditions $u(z_0^\mp)\in W^\pm_{p_\pm}$. 
This yields a homeomorphism (in appropriate topologies) between the moduli space \eqref{hmod} and the fiber product over $\CP^1\times M$ with the half-infinite Morse trajectory spaces from \S\ref{ssec:Morse},
\begin{equation}\label{h-fiber}
\cM(p_-,p_+; A) \; \cong\;  \cM(p_-,M) \; \leftsub{\{z_0^+\}\times\ev}{\times}_{\ev^+} \; \cM_{\rm SFT}(A) \; \leftsub{\ev^-}{\times}_{\{z_0^-\}\times\ev} \; \cM(M,p_+).
\end{equation}
Compared with \eqref{iota-fiber-SFT} this replaces the Gromov-Witten moduli space in \eqref{eq:GWeval} with a family of moduli spaces for almost complex structures $\Ti J_H^R$ on $\CP^1\times M$ arising from $Y_H^R$ for $R\in[0,\infty)$, 
\begin{equation}
 \cM_{\rm SFT}(A) \,:=\;  \bigsqcup_{R \in [0,\infty)} \quo{\bigl\{ v: \CP^1 \to \CP^1\times M  \;\big|\;  \dbar_{\Ti J^R_H} v = 0 , [v]=[\CP^1]+ A \bigr\}}{{\rm Aut}(\CP^1,z_0^-,z_0^+)}  .  \label{SFTstretchmod}
\end{equation}
Here, again, we implicitly include the two marked points $z_0^\pm\in\CP^1$. 
Then, for $R\to\infty$, the degeneration of the PDE $ \dbar_{\Ti J^R_H} v = 0$ is the ``neck stretching''\footnote{Strictly speaking, $R \in [0,2]$ parametrizes a family of Gromov-Witten moduli spaces for varying almost complex structure. At $R = 2$, the manifold $S^1 \times M$ with its stable Hamiltonian structure (see \S\ref{ssec:poly}) embeds as a stable hypersurface in $\CP^1 \times M$. Then $R \in [2,\infty)$ parametrizes the SFT neck-stretching.} considered more generally in Symplectic Field Theory \cite{egh}. The evaluation maps from \eqref{eq:GWeval} directly generalize to
\begin{equation}
\ev^\pm\;:\; \cM_{\rm SFT}(A)  
\;\to\; \CP^1\times M , \qquad
[v] \;\mapsto\;  v(z_0^\pm)  . \label{SFTstretchEval}
\end{equation}
Now, as in \S\ref{ssec:iota-mod}, each factor in the fiber product \eqref{h-fiber} has natural compactifications -- either the compactified Morse trajectory spaces from Theorem~\ref{thm:Morse}, or the compact zero set $\bM_{\rm SFT}(A)=\s_{\rm SFT}^{-1}(0)$ of a sc-Fredholm section that we will introduce in \S\ref{ssec:polyh}. Combined with the construction of fiber products in polyfold theory \cite{Ben-fiber} this will yield compactifications and polyfold descriptions of the moduli spaces \eqref{hmod}, and the construction of $h:CM\to CM$ then again proceeds as in \S\ref{ssec:construct}. 
Establishing the algebraic properties in Theorem~\ref{thm:main} relating $h$ with $\iota$ and $SSP\circ PSS$ will moreover require an in-depth discussion of the boundary stratification of the polyfold domains $\cB_{\rm SFT}(A)$ of these sections, and ``coherent'' choices of perturbations in \S\ref{sec:algebra}.

\begin{rmk} \rm \label{rmk:h-energy}
Gromov-compactifications of the moduli spaces $\cM(p_-,p_+; A)$ will result from energy estimates \cite[Remark~8.1.7]{MS} for solutions of $\dbar_J u = Y_H^R(u)$,
\begin{equation} \label{eq:h-energy}
\textstyle
E_R(u) \,:=\; \frac 12 \int_{\CP^1} |\rd u + X_{H^R_\beta}(u) |  \;\leq\; 
\int_{\CP^1} u^*\omega + \|R_{H^R_\beta}\| \;=\; \omega([u]) + 2 \|H(\theta,\cdot)\| . 
\end{equation}
Here
$R_{H^R_\beta} \dvol_{\CP^1} = \rd H^R_\beta + \tfrac 12 { H^R_\beta\wedge H^R_\beta } =  \beta_R' \,  H \, \rd r \wedge \rd \theta$ has uniformly bounded Hofer norm
$$\textstyle
\| R_{H^R_\beta}\|  \;=\; \int_{\CP^1} ( \max R_{H^R_\beta} - \max R_{H^R_\beta} )  \;=\; \int_0^\infty \int_{S^1} |\beta_R'(r)| \|H(\theta,\cdot)\| \, \rd\theta \, \rd r \;=\; 2 \|H(\theta,\cdot)\|, 
$$
where $\|H(\theta,\cdot)\|:= \max_{x\in M} H(\theta,x) -  \min_{x\in M} H(\theta,x)$ and 
$\beta_R\in\cC^\infty((0,\infty),[0,1])$ is constant except  
\begin{align*}
\beta_R|_{[e^{-\frac R2}, e^{1-\frac{R}{2}}]} \,:\; r \mapsto \beta(r e^{\frac R2}) 
\quad\text{with}\quad & \tfrac{\rd}{\rd r} \beta_R \geq 0 \quad\text{and}\quad\textstyle
\int_{e^{-\frac R2}} ^{e^{1-\frac{R}{2}}} \bigl| \tfrac{\rd}{\rd r} \beta_R \bigr|\, \rd r = \beta(e) - \beta(1) = 1 , \\
\beta_R|_{[e^{\frac{R}{2}-1},e^{\frac R2}]} \,:\; r \mapsto  \beta(r^{-1} e^{\frac R2})
\quad\text{with}\quad & \tfrac{\rd}{\rd r} \beta_R \leq 0 \quad\text{and}\quad\textstyle
\int_{e^{\frac R2 - 1}} ^{e^{\frac{R}{2}}} \bigl| \tfrac{\rd}{\rd r} \beta_R \bigr|\, \rd r = - \bigl( \beta(1) - \beta(e) \bigr) = 1 .
\end{align*}
This proves \eqref{eq:h-energy}, and thus establishes energy bounds on the perturbed pseudoholomorphic maps in each of our moduli spaces, where we fix $[u]=A$. 
Now SFT-compactness \cite{sft-compactness} asserts that for any $C>0$ the set of solutions of bounded energy $\bigsqcup_{R\in[0,\infty)}\{  u : \CP^1 \to M \,| \, \dbar_J u = Y^R_H(u) , E_R(u)\leq C  \}$ is compact up to breaking and bubbling.
This compactness will be stated rigorously in polyfold terms in Assumption~\ref{ass:iota,h}~(ii). 
\end{rmk}

\subsection{Construction of the morphisms} \label{ssec:polyh}

In this section we construct the $\L$-linear maps $\iota:CM \to CM$ and $h:CM \to CM$ analogously to \S\ref{ssec:construct} by first obtaining compactifications and polyfold descriptions for the moduli spaces in \S\ref{ssec:iota-mod} and \S\ref{ssec:h-mod} as in \S\ref{ssec:poly}.
This construction is motivated by the fiber product descriptions of the moduli spaces in \eqref{iota-fiber-SFT}, \eqref{h-fiber}, which couple Morse trajectory spaces from \S\ref{ssec:Morse} with moduli spaces of pseudoholomorphic curves in $\CP^1\times M$ via evaluation maps \eqref{eq:GWeval}, \eqref{SFTstretchEval}. Polyfold descriptions of these moduli spaces and their properties are stated in the following Assumption~\ref{ass:iota,h} for reference, with proofs in \cite{hwz-gw} resp.\ outlined in \cite{fh-sft}.
A summary of the types of curves in each polyfold and subsets thereof is displayed in Table~\ref{tab:asspolyfolds2}.
Here we formulate the evaluation map in the context of neck stretching, as explained in the following remark, using a splitting of the sphere as topological space with smooth structures on the complement of the equator
$$
\overline{\CP^1_\infty} \,:=\; \overline{\C^+} \cup_{S^1} \overline{\C^-} \;\cong\; 
\C^+ \sqcup S^1 \sqcup \C^- , 
$$
using the topologies and smooth structures on $\ov{\C^\pm} = \C^\pm \sqcup S^1 \cong \{z\in\C^\pm \,|\, |z|\leq 1\}$ from Remark~\ref{rmk:evaluation drama}. 

\begin{rmk}\rm \label{rmk:drama2}
(i)
Recall from \S\ref{ssec:iota-mod} that we denote by $\cB_{\rm GW}(A)$ a Gromov-Witten polyfold of curves in class $[\CP^1] + A\in H_2(\CP^1\times M)$ with $2$ marked points. These are determined by $A\in H_2(M)$ as we model graphs of maps $\CP^1\to M$, but should not be confused with a polyfold of curves in $M$. In particular, $\cB_{\rm GW}(A)$ never contains constant maps and hence is well defined for $A=0$. 
The properties of the Gromov-Witten moduli spaces for $\o(A)\leq 0$ are spelled out abstractly in Assumption~\ref{ass:iota,h}(ii) below; for the geometric meaning see Remark~\ref{rmk:iota-energy}.

\medskip\noindent (ii)
The SFT polyfolds $\cB_{\rm SFT}(A)$ will similarly describe curves in class $[\CP^1] + A$ in a neck stretching family of targets $(\CP^1_R \times M)_{R\in[0,\infty]}$ as in \cite[\S3.4]{sft-compactness}, given by 
$$
\CP^1_R := \quotient{ D_+ \sqcup E_R \sqcup D_- }{\sim_R} \qquad\text{with}\qquad
E_R = \begin{cases} 
[-R,R]\times S^1 &; R<\infty, \\
[0,\infty)\times S^1 \sqcup (-\infty,0]\times S^1  &; R=\infty. 
\end{cases}
$$
Here we identify the boundaries of the closed unit disks $D_\pm = \{z\in\C  \,|\, |z|\leq 1\}$ with the boundary components of the necks $E_R$ via 
$$
\partial D_\pm \in e^{i\theta}  \;\sim_R\;
\left. \begin{cases} 
(\pm R, e^{\pm i\theta}) &; R<\infty \\
(0_\pm, e^{\pm i\theta}) &; R=\infty  
\end{cases} \right\}
\in \partial E_R , 
$$
where we denote $0_+:= 0 \in [0,\infty)$ and $0_-:= 0 \in (-\infty,0]$ so that $\partial E_\infty = \{0_+\}\times S^1 \sqcup \{0_-\}\times S^1$. 
To describe convergence and evaluation maps we also embed each $\CP^1_R \subset \overline{\CP^1_\infty}=\C^+\sqcup S^1 \sqcup \C^-$ by 
\begin{align*}
\quotient{D_+ \sqcup [-R,0)\times S^1}{\sim_R} \;\cong\; \quotient{D_+ \sqcup [0,\infty)\times S^1}{\sim_\infty}  \; & =: \,  \C^+  ,  \\
\quotient{D_- \sqcup (0,R]\times S^1}{\sim_R} \;\cong\; \quotient{D_- \sqcup (-\infty,0] \times S^1}{\sim_\infty}  &=:\, \C^- ,
\qquad
E_R \supset \{0\}\times S^1 \;\cong\; S^1 \subset  \overline{\CP^1_\infty} . 
\end{align*}
For $R=0$ this is to be understood as $\CP^1_0=\qu{D_+\sqcup D_-}{\partial D_+\sim\partial D_-}$ with $D_\pm \less  \partial D_\pm \cong \C^\pm$, and for all $R<\infty$ we view the resulting homeomorphism $\CP^1_R\cong\ov{\CP^1_\infty}\cong\CP^1$ as identifying the standard marked points $\CP^1\ni z_0^+=[1:0] \cong 0\in \C^+$ and $\CP^1\ni z_0^-=[0:1] \cong 0\in \C^-$. 
When these embeddings are done via linear shifts $[-R,-1) \cong [0,R-1)$ and $(1,R] \cong (1-R,0]$ extended by a smooth family of diffeomorphisms $[-1,0)\cong [R-1,\infty)$ and $(0,1] \cong (-\infty, 1-R]$, 
then the pullback of the almost complex structures $\Ti J_H^R$ on $\CP^1_R\times M$ 
converges for $R\to\infty$ in $\cC^\infty_{\rm loc}\bigl((\ov{\CP^1_\infty}\less S^1)\times M\bigr)$ to the almost complex structures $\Ti J_H^+, \Ti J_H^-$ on $\C^+ \times M \;\sqcup\; \C^-\times M = \CP^1_\infty \times M \subset \ov{\CP^1_\infty} \times M$, 
which are used in the construction of the PSS and SSP moduli spaces in \S\ref{ssec:poly}.
Moreover, this allows us to extend the evaluation maps from \eqref{SFTstretchEval} to continuous maps $\ov\ev^\pm : \bM_{\rm SFT}(A) \to \ov{\CP^1_\infty}\times M$ on the compactified SFT moduli space. At $R=\infty$ this involves pseudoholomorphic buildings in $\C^+\times M \;\sqcup\; \R\times S^1\times M \ldots  \;\sqcup\; \R\times S^1\times M  \;\sqcup\; \C^-\times M$, and for any marked point with evaluation into a cylinder $\R\times S^1\times M$ we project the result to $S^1\times M\subset \ov{\CP^1_\infty}\times M$ by forgetting the $\R$-component. 

Finally, this formulation with $\ov{\CP^1_\infty}=\ov{\C^+}\cup_{S^1}\ov{\C^-}$ will allow us to compare the evaluation at $R=\infty$ with the product of the evaluations 
$\ov\ev^\pm : \bM^\pm_{\rm SFT}(\g;A) \to \ov{\C^\pm}\times M$
constructed in Remark~\ref{rmk:evaluation drama}. 
While this will be stated rigorously only in Assumption~\ref{ass:iso}~(iii)(c), note here that we should expect three top boundary strata of an ambient polyfold at $R=\infty$, corresponding to the distribution of marked points on the curves in $\C^+\times M \;\sqcup\; \C^-\times M$. 
For the fiber product construction, only the boundary components with one marked point in each factor are relevant -- in fact only those with marked points near $z_0^+\cong 0 \in \C^+$ and $z_0^-\cong 0 \in \C^-$. Thus we will work with the open subset $(\ov\ev^+)^{-1}(\C^+\times M)\cap (\ov\ev^-)^{-1}(\C^-\times M)$ where the two evaluations for any $R\in[0,\infty]$ are constrained to take values in the open sets given by $\C^\pm \subset \ov{\CP^1_\infty}$. 
 \end{rmk}

\begin{ass} \label{ass:iota,h}
There is a collection of oriented sc-Fredholm sections of strong polyfold bundles
$\s_{\rm GW}: \cB_{\rm GW}(A)\to \cE_{\rm GW}(A)$ and
 $\s_{\rm SFT}: \cB_{\rm SFT}(A)\to \cE_{\rm SFT}(A)$ indexed by $A\in H_2(M)$, sc$^\infty$ maps $\ov\ev^\pm: \cB_{\rm GW}(A)\to \CP^1\times M$, and continuous maps $\ov\ev^\pm: \cB_{\rm SFT}(A)\to \ov{\CP^1_\infty}\times M$ with the properties:
\begin{enumilist}
\item  
The sections have Fredholm indices $\text{\rm ind}(\s_{\rm GW}) = 2c_1(A) + \dim M + 4$ on $\cB_{\rm GW}(A)$ resp.\ $\text{\rm ind}(\s_{\rm SFT}) = 2c_1(A) + \dim M + 5$ on $\cB_{\rm SFT}(A)$.
\item
Each zero set $\bM_{\rm GW}(A):=\s_{\rm GW}^{-1}(0)$ and $\bM_{\rm SFT}(A):=\s_{\rm SFT}^{-1}(0)$ is compact, and given any $C\in\R$ there are only finitely many $A\in H_2(M)$ with nonempty zero set $\bM_{\rm GW}(A)\neq\emptyset$ resp.\ $\bM_{\rm SFT}(A)\neq\emptyset$. 
Moreover, for $\omega(A)\leq 0$ we have $\bM_{\rm GW}(A)=\emptyset$ except for $A=0\in H_2(M)$ when $\s_{\rm GW}|_{\cB_{\rm GW}(0)}\pitchfork 0$ is in general position with zero set $\bM_{\rm GW}(0)\simeq\CP^1\times\CP^1\times M$ identified by 
$$
 \cB_{\rm GW}(0) \;\supset\; \s_{\rm GW}^{-1}(0) \;=\; \bM_{\rm GW}(0) \;\overset{\ov\ev^+\times\ov\ev^-}{\longrightarrow}\; \bigl\{ (z^+,x, z^-, x) \,\big|\, z^-,z^+\in\CP^1, x\in M \bigr\} . 
$$
\item
The polyfolds $\cB_{\rm GW}(A)$ have no boundary, $\partial\cB_{\rm GW}(A)=\emptyset$.  
For $\cB_{\rm SFT}(A)$ there is a natural inclusion $[0,\infty) \times\cB_{\rm GW}(A) \subset \cB_{\rm SFT}(A)$ that covers the interior $\partial_0\cB_{\rm SFT}(A) = (0,\infty) \times\cB_{\rm GW}(A)$ and identifies the boundary $\partial \cB_{\rm SFT}(A)$ to consist of the disjoint sets 
$\{0\} \times\cB_{\rm GW}(A)$ and $\lim_{R\to\infty} \{R\} \times\cB_{\rm GW}(A)$ of $\cB_{\rm SFT}(A)$. 
Moreover, this inclusion identifies the section $\sigma_{\rm GW}$ and evaluation maps $\ov\ev^\pm$ with the restricted section $\s_{\rm SFT}|_{\{0\} \times\cB_{\rm GW}(A)}$ and evaluations $\ov\ev^\pm|_{\{0\} \times\cB_{\rm GW}(A)}$.
(A description of the relevant $R=\infty$ parts of the boundary $\partial \cB_{\rm SFT}(A)$ is given in Assumption~\ref{ass:iso}.)

\item The sections $\s_{\rm GW}$ and $\s_{\rm SFT}$ have tame sc-Fredholm representatives in the sense of \cite[Def.5.4]{Ben-fiber}. 
The product of evaluation maps $\ov\ev^+\times\ov\ev^-: \cB_{\rm GW}(A) \to \CP^1\times M\times \CP^1\times M$ is $\s_{\rm GW}$-compatibly submersive in the sense of Definition~\ref{def:submersion}. On the open subset
$$
\cB^{+,-}_{\rm SFT}(A) \,:=\; (\ov\ev^+)^{-1}(\C^+\times M)\cap  (\ov\ev^-)^{-1}(\C^-\times M) \;\subset\;
\cB_{\rm SFT}(A)
$$
the evaluation maps $\ov\ev^\pm: \cB_{\rm SFT}(A) \to \ov{\CP^1_\infty}\times M$
restrict to a $\s_{\rm SFT}$-compatibly submersive map
\begin{equation}\label{eq:SFTevalB}
\ev^+\times \ev^- \,:\; 
\cB^{+,-}_{\rm SFT}(A)  \; \to \;  \C^+\times M\times \C^-\times M  .
\end{equation}
On this domain intersected with $\{0\}\times\cB_{\rm GW}(A)\subset\partial_1\cB_{\rm SFT}(A)$, this map coincides with the Gromov-Witten evaluations $\ov\ev^+\times \ov\ev^-$ viewed as maps
$$ 
\ev^+\times \ev^-  \,:\;   \cB^{+,-}_{\rm GW}(A)  \; \to \;  \C^+\times M\times \C^-\times M , 
$$
where we identify $\C^+\sqcup\C^- = \CP^1 \less S^1$ and restrict to the domain
$$
\{0\}\times \cB^{+,-}_{\rm GW}(A) \,:=\;
\bigl(\{0\}\times\cB_{\rm GW}(A)\bigr) \cap \cB^{+,-}_{\rm SFT}(A) 
\;=\; \{0\}\times \bigl( (\ov\ev^+)^{-1}(\C^+\times M)\cap  (\ov\ev^-)^{-1}(\C^-\times M) \bigr) .
$$
\end{enumilist}
\end{ass}

\begin{rmk} \label{rmk:ass2} \rm  
\begin{enumilist}
\item
While not needed for our proof, we state the following properties for intuition:

\smallskip\noindent
{\it
The ${\rm Aut}(\CP^1,z_0^-,z_0^+)$-orbits of smooth maps $v: \CP^1 \to \CP^1\times M$ which represent the class $[\CP^1]+ A$ form a dense subset $\cB_{\rm dense}(A)\subset \cB_{\rm GW}(A)$. On this subset, the section is given by $\s_{\rm GW}([v]) =  [(v,\dbar_{\Ti J} v)]$.
Moreover, $[0,\infty)\times \cB_{\rm dense}(A) \subset \cB_{\rm SFT}(A)$ is a dense subset that intersects the boundary $\partial\cB_{\rm SFT}(A)$ exactly in $\{0\}\times \cB_{\rm dense}(A)$, and on which the section is given by $\s_{\rm SFT}(R,[v]) = [(v, \dbar_{\Ti J_H^R} v )]$. 
On these dense subsets, $\ov\ev^\pm([v])$ resp.\ $\ov\ev^\pm(R,[v])$ is the evaluation as in \eqref{SFTstretchEval}.

The intersection of the zero sets with the dense subsets 
$\s_{\rm GW}^{-1}(0)\cap \cB_{\rm dense}(A) \cong \cM_{\rm GW}(A)$
and
$\s_{\rm SFT}^{-1}(0)\cap [0,\infty) \times \cB_{\rm dense}(A) \cong \cM_{\rm SFT}(A)$ 
are naturally identified with the Gromov-Witten moduli space \eqref{eq:GWeval} and
SFT moduli space in \eqref{SFTstretchmod}.
}
\smallskip

\item

The polyfold description $\s_{\rm GW}: \cB_{\rm GW}(A)\to \cE_{\rm GW}(A)$ is developed for the homology classes $[\CP^1]+ A\in H_2(\CP^1\times M)$ in \cite{hwz-gw}, with the submersion property shown in \cite[Ex.5.1]{Ben-fiber}. The properties for $\omega(A)\leq 0$ in 
Assumption~\ref{ass:iota,h}~(ii) 
follow from the fact that nonconstant pseudoholomorphic curves have positive symplectic area, and linear Cauchy-Riemann operators on trivial bundles (arising from linearization at constant maps) are surjective. 
The construction of $\s_{\rm SFT}$ starts by recognizing that the family of almost complex manifolds in Remark~\ref{rmk:drama2}~(ii) for $R<\infty$ is equivalent to a degeneration of the almost complex structure on $\CP^1\times M$ along the equator $S^1\subset\CP^1$. This can be described by an $R$-dependent bundle and section over $[0,\infty)\times \cB_{\rm GW}(A)$. The construction for $R\to\infty$ then proceeds analogous to \cite[\S 3]{fh-sft}, with buildings consisting of a top and bottom floor curve in $\C^\pm\times M$ and intermediate floors given by curves in $\R\times S^1\times M$. Thus 
Assumption~\ref{ass:iota,h}~(iii) 
and the compatibility with $\cB_{\rm GW}(A)$ in (iv) hold by construction.
The polyfold and bundle structure are again obtained as in \cite{fh-sft} by extending the constructions in \cite{hwz-gw} with local models for punctures and neck-stretching from \cite[\S 3]{fh-2}, using the implanting method in \cite[\S 3,\S 5]{fh-1}. The remaining properties are proven as outlined in Remark~\ref{rmk:ass}~(ii). 
\end{enumilist}
\end{rmk}

Given any such polyfold descriptions of the moduli spaces of pseudoholomorphic curves, we now extend the fiber product descriptions of the moduli spaces 
$$
\cM^{(\iota)}(p_-,p_+; A) \;\cong\;  \cM(p_-,M)\; \leftsub{\{z_0^+\}\times\ev}{\times}_{\ev^+} \cM_{\rm GW / SFT}(A) \leftsub{\ev^-}{\times}_{\{z_0^-\}\times\ev} \;\cM(M,p_+)
$$ 
in \S\ref{ssec:iota-mod} and \S\ref{ssec:h-mod} to obtain ambient polyfolds which contain compactifications of the moduli spaces. 
Towards this we define for each $p_-,p_+\in{\rm Crit}(f)$ and $A\in H_2(M)$ the topological spaces
\begin{align*}
\tilde{\cB}^\iota(p_-,p_+; A) &\,:=\; \left\{ (\ul{\t}_-, \ul{v} ,\ul{\t}_+) \in \bM(p_-,M) \times \, \cB_{\rm GW}(A) \, \times \bM(M,p_+)  \,\big|\,
( z_0^\pm , \ev(\ul{\t}_\pm) ) = \ov\ev^\pm(\ul{v})  \;  \right\}  \\
&\,\phantom{:}=\; \left\{ (\ul{\t}_-, \ul{v} ,\ul{\t}_+) \in \bM(p_-,M) \times \, \cB^{+,-}_{\rm GW}(A) \, \times \bM(M,p_+)  \,\big|\, ( 0 , \ev(\ul{\t}_\pm) ) = \ev^\pm(\ul{v})  \;  \right\},  \\
\tilde{\cB}(p_-,p_+; A) &\,:=\; \left\{ (\ul{\t}_-, \ul{w},\ul{\t}_+)\in \bM(p_-,M) \times \cB_{\rm SFT}(A)\times \bM(M,p_+)  \,\big|\,
( z_0^\pm ,\ev(\ul{\t}_\pm) )  = \ov\ev^\pm(\ul{w}) \right\} \\
&\,\phantom{:}=\; \left\{ (\ul{\t}_-, \ul{w},\ul{\t}_+)\in \bM(p_-,M) \times \cB^{+,-}_{\rm SFT}(A)\times \bM(M,p_+)  \,\big|\,  (0,\ev(\ul{\t}_\pm) )  = \ev^\pm(\ul{w})  \right\}, 
\end{align*}
where the last equality stems from the identification at the end of Remark~\ref{rmk:drama2}~(ii).
Then the abstract fiber product constructions in \cite{Ben-fiber} will be used as in Lemma~\ref{lem:PSS poly} to obtain the following polyfold description for compactifications of the moduli spaces in \S\ref{ssec:iota-mod} and \S\ref{ssec:h-mod}.

\begin{lem} \label{lem:iota,h poly}
Given any $p_-,p_+\in{\rm Crit}(f)$ and $A\in H_2(M)$, there exist open subsets $\cB^\iota(p_-,p_+; A) \subset \tilde{\cB}^\iota(p_-,p_+; A)_1$ and $\cB(p_-,p_+; A) \subset \tilde{\cB}(p_-,p_+; A)_1$ which contain the smooth levels $\tilde{\cB}^{(\iota)}(p_-,p_+;A)_\infty$ of the fiber products and inherit natural polyfold structures with smooth level of the interior
\begin{align*}
\partial_0\cB^\iota(p_-,p_+; A)_\infty &\;=\; \cM(p_-,M) \; \leftsub{\{z_0^+\}\times\ev}{\times}_{\ev^+} \;  \cB^{+,-}_{\rm GW}(A)_\infty \; \leftsub{\ev^-}{\times}_{\{z_0^-\}\times\ev} \;  \cM(M,p_+) ,  \\
\partial_0\cB(p_-,p_+; A)_\infty &\;=\; \cM(p_-,M) \; \leftsub{\{z_0^+\}\times\ev}{\times}_{\ev^+} \;  \partial_0\cB^{+,-}_{\rm SFT}(A)_\infty \; \leftsub{\ev^-}{\times}_{\{z_0^-\}\times\ev} \; \cM(M,p_+)  , 
\end{align*}
and a scale-smooth inclusion
$$
\phi_\iota \,:\; \cB^\iota(p_-,p_+; A) \;\hookrightarrow\; \cB(p_-,p_+; A), \qquad
(\ul{\t}_-, \ul{v},\ul{\t}_+) \;\mapsto\; (\ul{\t}_-, 0,\ul{v},\ul{\t}_+) . 
$$
Moreover, pullback of the sections and bundles $\sigma_{\rm GW/SFT}: \cB_{\rm GW/SFT}(A) \to \cE_{\rm GW/SFT}(A)$ under the projection $\cB(p_-,p_+; A) \to \cB_{\rm GW/SFT}(A)$ induces sc-Fredholm sections of strong polyfold bundles $\s_{(p_-,p_+; A)}: \cB(p_-,p_+;A) \to \cE(p_-,p_+;A)$ of index $I(p_-,p_+; A)$ as in \eqref{h index} and
$\s^\iota_{(p_-,p_+; A)}: \cB^\iota(p_-,p_+;A) \to \cE^\iota(p_-,p_+;A)$ of index $I^\iota(p_-,p_+; A)=I(p_-,p_+; A)-1$ as in \eqref{iota index}.
Further, these are related via the inclusion $\phi_\iota$ by natural orientation preserving identification 
$\s^\iota_{(p_-,p_+; A)}\cong\phi_\iota^*\s_{(p_-,p_+; A)}$. 

The zero sets of these sc-Fredholm sections contain\footnote{
As in Remark~\ref{rmk:ass}, this identification is stated for intuition and will ultimately not be used in our proofs.}
the moduli spaces from \S\ref{ssec:iota-mod} and \S\ref{ssec:h-mod}, 
\begin{align*}
{\s^\iota_{(p_-,p_+; A)}}^{-1}(0) &\;=\;
\bM(p_-,M) \; \leftsub{\{z_0^+\}\times\ev}{\times}_{\ov\ev^+} \;  \s_{\rm GW}^{-1}(0) 
\; \leftsub{\ov\ev^-}{\times}_{\{z_0^-\}\times\ev} \;  \bM(M,p_+) 
\;\supset\; \cM(p_-,p_+; A),  \\
{\s_{(p_-,p_+; A)}}^{-1}(0) 
&\;=\; 
\bM(p_-,M) \; \leftsub{\{z_0^+\}\times\ev}{\times}_{\ov\ev^+} \;  \s_{\rm SFT}^{-1}(0) \; \leftsub{\ov\ev^-}{\times}_{\{z_0^-\}\times\ev} \; \bM(M,p_+) 
\;\supset\; \cM^\iota(p_-,p_+; A) .
\end{align*}
Finally, each zero set $\s^{(\iota)}_{(p_-,p_+; A)}\,\!\! ^{-1}(0)$ is compact,
and given any $p_\pm\in\Crit(f)$ and $C\in\R$, there are only finitely many $A\in H_2(M)$ with $\omega(A)\leq C$ and nonempty zero set ${\s^{(\iota)}_{(p_-,p_+; A)}} \,\!\! ^{-1}(0)\neq\emptyset$. 
\end{lem}

\begin{proof}
The inclusion $\phi_{\i}$ is sc$^\infty$ since the map $\cB_{\rm GW}(A) \hookrightarrow \cB_{\rm SFT}(A), \ul{v} \mapsto (0,\ul{v})$ is a sc$^\infty$ inclusion by Assumption~\ref{ass:iota,h}~(iii). 
Apart from further relations involving $\phi_\iota$, the proof is directly analogous to the fiber product construction in Lemma~\ref{lem:PSS poly}, using Assumption~\ref{ass:iota,h} -- in particular the sc$^\infty$ and $\s_{\rm SFT}$-compatibly submersive evaluation map \eqref{eq:SFTevalB} on the open subset $\cB^{+,-}_{\rm SFT}(A)\subset \cB_{\rm SFT}(A)$. 
This yields polyfold structures on open sets $\cB^{\i}(p_-,p_+;A)\subset \tilde{\cB}^\iota(p_-,p_+; A)_1$ and $\cB(p_-,p_+;A)\subset \tilde{\cB}(p_-,p_+; A)_1$ as well as the pullback sc-Fredholm sections $\s_{(p_-,p_+;A)}=\pr_{\rm SFT}^*\s_{\rm SFT}$ and $\s^\iota_{(p_-,p_+;A)}=\pr_{\rm GW}^*\s_{\rm GW}$ under the projections $\pr_{\rm GW/SFT} : \cB^{(\i)}(p_-,p_+;A) \rightarrow \cB_{\rm GW/SFT}(A)$.
Here we have $\pr_{\rm GW} = \pr_{\rm SFT} \circ \phi_{\i}$, so the bundle $\cE^\iota(p_-,p_+;A)=\pr_{\rm GW}^*\cE_{\rm GW}(A)$ and section $\s^\iota_{(p_-,p_+;A)}=\pr_{\rm GW}^*\s_{\rm GW}$ are naturally identified with the pullback bundle $\phi_\iota^*\cE(p_-,p_+;A) = \pr_{\rm GW}^*\cE_{\rm SFT}(A)|_{\{0\}\times\cB_{\rm GW}(A)}$ and section 
$\phi_\iota^*\s_{(p_-,p_+;A)} = \pr_{\rm GW}^*\s_{\rm SFT}|_{\{0\} \times\cB_{\rm GW}(A)}$ using Assumption~\ref{ass:iota,h}~(iii). 
Finally, the index of the induced section $\s_{(p_-,p_+;A)}$, and similarly of $\s^\iota_{(p_-,p_+;A)}$, is computed by \cite[Cor.7.3]{Ben-fiber} as
\begin{align*}
\text{\rm ind}(\s_{(p_-,p_+;A)}) &\;=\; 
\text{\rm ind}(\s_{\rm SFT}) + \dim \bM(p_-,M) + \dim \bM(M,p_+) - 2 \dim(\CP^1\times M) \\
&\;=\; 2c_1(A) + \dim M + 5 + |p_-| + \dim M - |p_+| -4 - 2 \dim M  \\
&\;=\; 2c_1(A) + |p_-|  - |p_+| + 1
\;=\; I(p_-,p_+;A) .
\end{align*}

\vspace{-5mm}

\end{proof}

\begin{table}
\centering
 \begin{tabular}{||c | c | c ||}  
 \hline 
 Notation & Description & Definition\\ [0.5ex] 
 \hline\hline
$\cB_{\rm dense}(A)$   & \thead{elements are equivalence classes under reparameterization by \\ ${\rm Aut}(\CP^1,z_0^-,z_0^+)$ of smooth maps $\CP^1 \to \CP^1\times M$ in class $[\CP^1] + A$ }
 & \thead{Remark~\ref{rmk:ass2} (i)} \\
 \hline
$\cB_{\rm GW}(A)$  &  \thead{a polyfold with dense subset $\cB_{\rm dense}(A)$, which contains the \\ Gromov-compactification $\bM_{\rm GW}( A)$ of the moduli space in \eqref{eq:GWeval}} 
& \thead{Assumption~\ref{ass:iota,h} }\\ 
 \hline
$\cB_{\rm SFT}(A)$  &  \thead{a polyfold with dense subset $[0,\infty)\times \cB_{\rm dense}(A)$, which contains the \\ SFT-compactification $\bM_{\rm SFT}( A)$ of the moduli space in \eqref{SFTstretchmod} }
 & \thead{Assumption~\ref{ass:iota,h}
} \\
 \hline
$\cB_{\scriptscriptstyle \rm GW / SFT}^{+,-}(A)$ & \thead{the open subsets of $\cB^\pm_{\rm \scriptscriptstyle GW / SFT}(A)$ containing the curves whose \\ evaluation at two marked point lands in $\mathbb{C}^{\pm} \times M$; see Remark~\ref{rmk:drama2}}
& \thead{Assumption~\ref{ass:iota,h} (iv)} \\
 \hline
$\tilde{\cB}^\iota(p_-,p_+; A)$  & \thead{elements are triples of two half-infinite broken Morse trajectories from \\ $p_-$ and to $p_+$  and a curve in $\cB_{\rm GW}(A)$ whose evaluations at the \\ marked points agrees with the endpoints of the Morse trajectories} & \thead{before Lemma~\ref{lem:iota,h poly}}\\
 \hline
 $\tilde{\cB}(p_-,p_+; A)$& \thead{elements are triples of two half-infinite broken Morse trajectories from \\ $p_-$ and to $p_+$  and a curve in $\cB_{\rm SFT}(A)$ whose evaluations at the \\ marked points agrees with the endpoints of the Morse trajectories} & \thead{before Lemma~\ref{lem:iota,h poly}}\\
 \hline
$\cB^\iota(p_-,p_+; A)$ & \thead{open subset of $\tilde{\cB}^\iota(p_-,p_+; A)_1$ containing $\cM^\iota(p_-,p_+; A)$ \\ over which $\s^\iota_{(p_-,p_+;A)}$ is sc-Fredholm }  & \thead{Lemma~\ref{lem:iota,h poly}} \\ [1ex] 
 \hline
 $\cB(p_-,p_+; A)$ & \thead{open subset of $\tilde{\cB}(p_-,p_+; A)_1$ containing $\cM(p_-,p_+; A)$ \\ over which $\s_{(p_-,p_+;A)}$ is sc-Fredholm }  & \thead{Lemma~\ref{lem:iota,h poly}} \\ [1ex] 
 \hline
\end{tabular}
\caption{Summary of the polyfolds and their subsets introduced in this section
}
\label{tab:asspolyfolds2}
\end{table}

Given this compactification and polyfold description of the moduli spaces $\cM(\alpha)\subset\sigma_{\alpha}^{-1}(0)$ and $\cM^\iota(\alpha)\subset {\sigma^\iota_{\alpha}}^{-1}(0)$ for all tuples in the indexing set 
$$
\cI\,:=\; \bigl\{ \alpha=(p_-,p_+;A) \,\big|\, p_-,p_+\in\Crit(f), A\in H_2(M) \bigr\} ,
$$
we can again apply \cite[Theorems~18.2,18.3,18.8]{HWZbook} to the sc-Fredholm sections $\sigma_\alpha$ and $\sigma_\alpha^\iota$ and 
obtain Corollary~\ref{cor:regularize} verbatim for these collections of moduli spaces. 
In \S\ref{sec:algebra} we will moreover make use of the fact that $\sigma_\alpha^\iota=\phi_\iota^*\sigma_\alpha$ arises from restriction of $\sigma_\alpha$, so admissible perturbations of $\sigma_\alpha$ pull back to admissible perturbations of $\sigma_\alpha^\iota$.
For now, we choose perturbations independently and thus as in Definition~\ref{def:PSS} obtain perturbation-dependent, and not yet algebraically related, $\Lambda$-linear maps.

\begin{dfn} \label{def:iota,h}
Given admissible sc$^+$-multisections
$\underline\kappa=( \kappa_{(p_-,p_+;A)} )_{p_\pm\in\Crit(f), A\in H_2}$
in general position to $(\sigma_{(p_-,p_+;A)} )$ and 
$\underline\kappa^\iota=( \kappa^\iota_{(p_-,p_+;A)} )_{p_\pm\in\Crit(f), A\in H_2}$
in general position to $(\sigma^\iota_{(p_-,p_+;A)} )$ as in Corollary~\ref{cor:regularize}, 
we define the maps $h_{\ul\kappa}:CM \to CM$ and  $\iota_{\ul\kappa^\iota}:CM \to CM$ to be the $\Lambda$-linear extensions of
$$
h_{\underline\kappa} \la p_- \ra := \hspace{-4mm} \sum_{\begin{smallmatrix}
\scriptstyle  p_+, A  \\ \scriptscriptstyle I(p_-,p_+;A)=0 
\end{smallmatrix} }  \hspace{-4mm}
\# Z^{\underline\kappa}(p_-,p_+;A)  \cdot T^{\o(A)} \la p_+ \ra
, \qquad
\iota_{\underline\kappa^\iota} \la p_- \ra := \hspace{-4mm} \sum_{\begin{smallmatrix}
\scriptstyle  p_+, A  \\ \scriptscriptstyle I^\iota(p_-,p_+;A)=0 
\end{smallmatrix} }  \hspace{-4mm}
\# Z^{\underline\kappa^\iota}(p_-,p_+;A)  \cdot T^{\o(A)} \la p_+ \ra . 
$$
\end{dfn}

The proof that the coefficients of these maps lie in the Novikov field $\Lambda$ is verbatim the same as Lemma~\ref{lem:pss-novikov}, based on the compactness properties in Lemma~\ref{lem:iota,h poly}.

\begin{rmk} \label{rmk:iota,h orient} \rm
The determination in Corollary~\ref{cor:regularize} of $\# Z^{\underline\kappa}(p_-,p_+;A), \# Z^{\underline\kappa^\iota}(p_-,p_+;A) \in \Q$ that is used in Definition~\ref{def:iota,h} requires an orientation of the sections $\s_{(p_-,p_+;A)}$ and $\s^\iota_{(p_-,p_+;A)}$. As in Remark~\ref{rmk:PSS orient} this is determined via the fiber product construction in Lemma~\ref{lem:iota,h poly} from the orientations of the Morse trajectory spaces in Remark~\ref{rmk:MorseOrient}~(i) and the orientations of $\s_{\rm GW}, \s_{\rm SFT}$ given in Assumption~\ref{ass:iota,h}. 
In practice, we will construct the perturbations $\ul\k,\ul\k^\iota$ by pullback of perturbations $\ul \lambda=(\lambda_A)_{A\in H_2(M)}$ of the SFT-sections $\s_{\rm SFT}$ and their restriction $\ul \lambda^\iota$ to $\{0\}\times\cB_{\rm GW}(A)\subset\partial \cB_{\rm SFT}(A)$. So we can specify the orientations of the regularized zero sets by expressing them as transverse fiber products of oriented spaces over $\CP^1\times M$ or $\C^\pm \times M$, 
\begin{align*}
Z^{\ul\k^{\iota}}(p_-,p_+;A) &\;=\; \bM(p_-,M) \; \leftsub{\ev_0^+}{\times}_{\ov\ev^+} \;  Z^{\ul\l^\iota}(A) 
\; \leftsub{\ov\ev^-}{\times}_{\ev_0^-} \;  \bM(M,p_+) , \\
&\;=\; \bM(p_-,M) \; \leftsub{\ev_0^+}{\times}_{\ev^+} \;  \bigl( Z^{\ul\l^\iota}(A) \cap \cB^{+,-}_{\rm GW}(A) \bigr) \; \leftsub{\ev^-}{\times}_{\ev_0^-} \;  \bM(M,p_+) , \\
Z^{\ul\k}(p_-,p_+;A) &\;=\; \bM(p_-,M) \; \leftsub{\ev_0^+}{\times}_{\ev^+} \;  \bigl( Z^{\ul\l}(A) \cap \cB^{+,-}_{\rm SFT}(A) \bigr) \; \leftsub{\ev^-}{\times}_{\ev_0^-} \;  \bM(M,p_+) , 
\end{align*}
using $\ov\ev^\pm: \cB_{\rm GW}(A)\to\CP^1\times M$ resp.\ 
$\ev^\pm: \cB^{+,-}_{\rm GW/SFT}(A)\to\C^\pm\times M$
and the Morse evaluations
$\ev_0^\pm: \bM(\ldots) \to \CP^1\times M, \ul\tau\mapsto (z_0^\pm, \ev(\ul\tau) )$
resp.\ 
$\ev_0^\pm: \bM(\ldots) \to \C^\pm\times M, \ul\tau\mapsto (0, \ev(\ul\tau) )$.
%
\end{rmk}

\section{Algebraic relations via coherent perturbations}
\label{sec:algebra}

In this section we prove parts (i)--(iii) of Theorem~\ref{thm:main}, that is the algebraic properties which relate the maps $PSS:CM\to CF$, $SSP:CF\to CM$ constructed in \S\ref{sec:PSS}, and the maps $\iota:CM\to CM$, $h:CM\to CM$ constructed in \S\ref{sec:iota,h}.
More precisely, we will make so-called ``coherent'' choices of perturbations
in \S\ref{ssec:chain}, \S\ref{ssec:iso}, and \S\ref{ssec:homotopy} which guarantee that 
(i) $\iota$ is a chain map, (ii) $\iota$ is a $\L$-module isomorphism, and (iii) $h$ is a chain homotopy between the composition $SSP\circ PSS$ and $\iota$. 

\subsection{Coherent polyfold descriptions of moduli spaces}  \label{ssec:coherent}

The general approach to obtaining not just counts as discussed in \S\ref{ssec:poly} but well-defined algebraic structures from moduli spaces of pseudoholomorphic curves is to replace them by compact manifolds with boundary and corners (or generalizations thereof which still carry `relative virtual fundamental classes') in such a manner that their boundary strata are given by Cartesian products of each other.
In the context of polyfold theory, this requires a description of the compactified moduli spaces $\bM(\alpha)=\sigma_\alpha^{-1}(0)$ as zero sets of a ``coherent collection'' of sc-Fredholm sections $\bigl(\sigma_\alpha:\cB(\alpha)\to\cE(\alpha)\bigr)_{\alpha\in\cI}$ of strong polyfold bundles. 
Here ``coherence'' indicates a well organized identification of the boundaries $\partial\cB(\alpha)$ with unions of Cartesian products of other polyfolds in the collection $\cI$, which is compatible with the bundles and sections.

As a first example, the moduli spaces $\cM^\iota(p_-,p_+;A)$ in \S\ref{ssec:iota-mod} which yield the map $\iota:CM \to CM$ are given polyfold descriptions $\s^\iota_{(p_-,p_+; A)}: \cB^\iota(p_-,p_+;A) \to \phi_\iota^* \cE(p_-,p_+;A)$ in Lemma~\ref{lem:iota,h poly} that arise as fiber products with polyfolds $\cB_{\rm GW}(A)$ without boundary. Thus their coherence properties stated below follow from properties of the fiber product in \cite{Ben-fiber} and the boundary stratification of the Morse trajectory spaces in Theorem~\ref{thm:Morse}. 
We state this result to illustrate the notion of coherence. The full technical statement -- on the level of ep-groupoids and including compatibility with bundles and sections -- can be found in the second bullet point of Lemma~\ref{lem:iota chain}.

\begin{lem} \label{lem:iota boundary}
For any $p_\pm\in{\rm Crit}(f)$ and $A\in H_2(M)$ the smooth level of the first boundary stratum of the fiber product $\cB^\iota(p_-,p_+; A)$ in Lemma~\ref{lem:iota,h poly} is naturally identified with
$$
\partial_1\cB^\iota(p_-,p_+; A)_\infty \;\cong\; \bigcup_{q\in{\rm Crit}(f)} \cM(p_-,q) \times \partial_0\cB^\iota(q,p_+; A)_\infty
\quad
\sqcup  \bigcup_{q\in{\rm Crit}(f)} \partial_0\cB^\iota(p_-,q; A)_\infty \times \cM(q,p_+).
$$
\end{lem}

\begin{proof}
By the fiber product construction \cite[Cor.7.3]{Ben-fiber} of $\cB^{\i}(p_-,p_+;A)$ in Lemma~\ref{lem:iota,h poly}, the degeneracy index
 satisfies $d_{\cB^{\i}(p_-,p_+;A)}(\ul{\t}^-,\ul{v},\ul{\t}^+) = d_{\bM(p_-,M)}(\ul{\t}^-) + d_{\cB_{\rm GW}(A)}(\ul{v}) + d_{\bM(M,p_+)}(\ul{\t}^+)$, and the smooth level is 
 $\cB^{\i}(p_-,p_+;A)_\infty =
 \bM(p_-,M) \; \leftsub{\{z_0^-\}\times\ev}{\times}_{\ev^-} \;  \cB_{\rm GW}^{+,-}(A)_\infty \; \leftsub{\ev^+}{\times}_{\{z_0^+\}\times\ev} \;  \bM(M,p_+)$.
The polyfold $\cB_{\rm GW}(A)$ and its open subset $\cB_{\rm GW}^{+,-}(A)$ are boundaryless by Assumption~\ref{ass:iota,h}~(iii), which means $d_{\cB_{\rm GW}(A)}=d_{\cB_{\rm GW}^{+,-}(A)} \equiv 0$. Hence we have $d_{\cB^{\i}(p_-,p_+;A)}(\ul{\t}^-,\ul{v},\ul{\t}^+) = 1$ if and only if $\ul{\t}^- \in \partial_1 \bM(p_-,M)$ and $\ul{\t}^+ \in \partial_0 \bM(M,p_+)$ or the other way around. These two cases are disjoint but analogous, so it remains to show that the first case consists of points in the union $\bigcup_{q\in{\rm Crit}(f)} \cM(p_-,q) \times \partial_0\cB^\iota(q,p_+; A)_{\infty}$. For that purpose recall the identification $\partial_1 \bM(p_-,M)= \bigcup_{q\in{\rm Crit}(f)} {\cM}(p_-,q) \times {\cM}(q,M)$ in Theorem~\ref{thm:Morse}, which is compatible with the evaluation $\ev: {\cM}(p_-,q) \times {\cM}(q,M) \to M, (\tau_1,\tau_2) \mapsto \ev(\tau_2)$ by construction, and thus
\begin{align*}
&
\partial_1\bM(p_-,M) \; \leftsub{\{z_0^-\}\times\ev}{\times}_{\ev^-} \;  \cB_{\rm GW}^{+,-}(A)_\infty \; \leftsub{\ev^+}{\times}_{\{z_0^+\}\times\ev} \;  \partial_0\bM(M,p_+) \\
&\quad =
\Bigl(\textstyle \bigcup_{q\in{\rm Crit}(f)} {\cM}(p_-,q) \times {\cM}(q,M) \Bigr) \; \leftsub{\{z_0^-\}\times\ev}{\times}_{\ev^-} \;  \cB_{\rm GW}^{+,-}(A)_\infty \; \leftsub{\ev^+}{\times}_{\{z_0^+\}\times\ev} \;  \cM(M,p_+) \\
&\quad =
\textstyle \bigcup_{q\in{\rm Crit}(f)} {\cM}(p_-,q) \times  \Bigl( {\cM}(q,M) \; \leftsub{\{z_0^-\}\times\ev}{\times}_{\ev^-} \;  \cB_{\rm GW}^{+,-}(A)_\infty \; \leftsub{\ev^+}{\times}_{\{z_0^+\}\times\ev} \;  \cM(M,p_+) \Bigr) \\
&\quad = \textstyle \bigcup_{q\in{\rm Crit}(f)} \cM(p_-,q) \times \partial_0\cB^\iota(q,p_+; A)_{\infty}
\end{align*}
Here we also used the identification of the interior smooth level in Lemma~\ref{lem:iota,h poly}. 
\end{proof}

Next, the polyfold description in Lemma~\ref{lem:iota,h poly} for the moduli spaces  $\cM(p_-,p_+;A)$ in \S\ref{ssec:h-mod}, which yield the map $h:CM \to CM$, are obtained as fiber products of the Morse trajectory spaces with polyfold descriptions $\s_{\rm SFT}: \cB_{\rm SFT}(A)\to \cE_{\rm SFT}(A)$ of SFT moduli spaces given in \cite{fh-sft-full}.
We will state as assumption only those parts of their coherence properties that are relevant to our argument in \S\ref{ssec:homotopy} for the chain homotopy $\iota - SSP\circ PSS = \rd \circ h + h \circ \rd$. Here the contributions to $\rd \circ h + h \circ \rd$ will arise from boundary strata of the Morse trajectory spaces, whereas $\iota - SSP\circ PSS$ arises from the following identification of the boundary of the polyfold $\cB^{+,-}_{\rm SFT}(A)$, which is given as open subset of $\cB_{\rm SFT}(A)$ in Assumption~\ref{ass:iota,h}~(iv).
\footnote{See also the end of Remark~\ref{rmk:drama2}~(ii) for the motivation of $\cB^{+,-}_{\rm SFT}(A)$ as open subset that intersects the boundary strata
$\lim_{R\to\infty} \{R\} \times\cB_{\rm GW}(A)\subset\partial\cB_{\rm SFT}(A)$ in the buildings which have one marked point in each of the components mapping to $\C^\pm\times M$, and no marked points mapping to intermediate cylinders $\R\times S^1\times M$.}

\begin{rmk} \rm \label{rmk:face}
In the following we will use the word ``face'' loosely for Cartesian products of polyfolds such as 
$\cF=\cB^+_{\rm SFT} (\g;A_+)\times \cB^-_{\rm SFT} (\g; A_-)$ and their immersions into the boundary of another polyfold such as $\partial\cB^{+,-}_{\rm SFT}(A)$. We also refer to the image of the immersion $\cF \hookrightarrow \partial\cB^{+,-}_{\rm SFT}(A)$ as a face of $\cB^{+,-}_{\rm SFT}(A)$. Compared with the formal definition of faces in \cite[Definitions~2.21,11.1,16.13]{HWZbook}, ours are disjoint unions of faces. 
They describe the interaction between the moduli spaces - roughly speaking:
\begin{enumilist}
\item
The $R\to\infty$ boundary parts of $\cB_{\rm SFT}(A)$ in which the marked points separate are covered by immersions of products of the PSS and SSP polyfolds. This structure arises from generalizing the SFT compactification in \cite{sft-compactness} to buildings of not necessarily holomorphic maps. The parts of the boundary described here are given by buildings whose top and bottom floors are given by maps to $\C^\pm\times M$ and intermediate floors given by maps to $\R\times S^1\times M$. The immersions then arise from stacking a building in $\cB^+_{\rm SFT}(\g;A_+)$ (with top floor in $\C\times M$) on top of a building in $\cB^-_{\rm SFT}(\g;A_-)$ (with bottom floor in $\C^-\times M$). Here a lack of injectivity arises at buildings with middle floors in $\R\times S^1\times M$ from ambiguity in splitting such building into two parts. 

\item
The immersions restrict to a disjoint cover of the top boundary stratum of $\cB^{+,-}_{\rm SFT}(A)$ by embeddings. This restriction is given by the buildings with a single floor -- guaranteeing injectivity by avoiding the ambiguous middle floors in $\R\times S^1\times M$. 
\item
The immersions are compatible -- simply by construction -- with the evaluation maps, bundles, and sections for the boundary components at $R=\infty$, and the boundary $\cB_{\rm GW}(A)$ at $R=0$. 
\end{enumilist}

\end{rmk}

\begin{ass}\label{ass:iso}
The collection of oriented sc-Fredholm sections of strong polyfold bundles 
$\s^\pm_{\rm SFT}: \cB^{\pm}_{\rm SFT}(\g; A)\to \cE^\pm_{\rm SFT}(\g;A)$, 
$\s_{\rm GW}: \cB_{\rm GW}(A)\to \cE_{\rm GW}(A)$, 
$\s_{\rm SFT}: \cB_{\rm SFT}(A)\to \cE_{\rm SFT}(A)$ 
for $\g\in\cP(H)$ and $A\in H_2(M)$ 
together with the evaluation maps $\ov\ev^\pm: \cB^{\pm}_{\rm SFT}(\g;A)\to \ov{\C^\pm} \times M$, 
$\ov\ev^\pm: \cB_{\rm GW}(A)\to \CP^1 \times M$, 
$\ov\ev^\pm: \cB_{\rm SFT}(A)\to \ov{\CP^1_\infty} \times M$, 
and their sc$^\infty$ restrictions on open subsets, 
$\ev^\pm: \cB^{\pm,\C}_{\rm SFT}(\g;A)\to \C^\pm \times M$,  
$\ev^\pm: \cB^{+,-}_{\rm GW/SFT}(A)\to \C^\pm \times M$ 
from Assumptions~\ref{ass:pss}, \ref{ass:iota,h} has the following coherence properties.
\begin{enumilist}
\item For each $\g \in \cP(H)$ and $A_-,A_+ \in H_2(M)$ such that $A_- + A_+ = A$, there is a sc$^\infty$ immersion
$$
l_{\g,A_{\pm}} \,:\; \cB^+_{\rm SFT}(\g;A_+) \times \cB^-_{\rm SFT}(\g;A_-) \;\rightarrow\; \partial \cB_{\rm SFT}(A)
$$
whose restriction to the interior $\partial_0\cB^+_{\rm SFT}(\g;A_+) \times \partial_0\cB^-_{\rm SFT}(\g;A_-)\subset \cB^{+,\C}_{\rm SFT}(\g;A_+) \times \cB^{-,\C}_{\rm SFT}(\g;A_-)$ is an embedding into the boundary of the open subset $\cB^{+,-}_{\rm SFT}(A) \subset \cB_{\rm SFT}(A)$.
They map into the limit set $\lim_{R\to\infty} \{R\} \times\cB_{\rm GW}(A)$ from Assumption~\ref{ass:iota,h}(iii), so cover most of the boundary\footnote{
The extra boundary faces of $\cB_{\rm SFT}(A)$ arise from both marked points mapping to the same component in the $R\to\infty$ neck stretching limit. These will not be relevant to our construction of coherent perturbations.
}
$$
\partial\cB_{\rm SFT}(A) \;\supset\; \{0\}\times \cB_{\rm GW}(A)  \;\; \sqcup \;\;
\textstyle\bigcup_{\begin{smallmatrix}
\scriptscriptstyle \g \in \cP(H)\\
\scriptscriptstyle A_- + A_+ = A\\
  \end{smallmatrix}} l_{\g,A_{\pm}}\bigl( \cB^{+}_{\rm SFT}(\g;A_+) \times \cB^{-}_{\rm SFT}(\g;A_-) \bigr).
$$

\item
The union of the images $l_{\g,A_{\pm}}\bigl(\cB^{+,\C}_{\rm SFT}(\g;A_+) \times \cB^{-,\C}_{\rm SFT}(\g;A_-)\bigr)\subset\partial \cB^{+,-}_{\rm SFT}(A)$ for all admissible choices of $\g,A_{\pm}$ is the intersection of $\cB^{+,-}_{\rm SFT}(A)$ with $\lim_{R\to\infty} \{R\} \times\cB_{\rm GW}(A)\subset\partial \cB_{\rm SFT}(A)$, i.e.\
\begin{align*} \textstyle
\partial\cB^{+,-}_{\rm SFT}(A) &\;=\; \{0\}\times \cB^{+,-}_{\rm GW}(A)  \;\; \sqcup \;\;
\partial^{R=\infty}\cB^{+,-}_{\rm SFT}(A), \\
\text{where}\qquad\quad
\partial^{R=\infty}\cB^{+,-}_{\rm SFT}(A) &\;=\; \textstyle
\bigcup_{\begin{smallmatrix}
\scriptscriptstyle \g \in \cP(H)\\
\scriptscriptstyle A_- + A_+ = A\\
  \end{smallmatrix}} l_{\g,A_{\pm}}\bigl( \cB^{+,\C}_{\rm SFT}(\g;A_+) \times \cB^{-,\C}_{\rm SFT}(\g;A_-) \bigr).
\end{align*}
When restricted to the interiors, this yields a disjoint cover of the top boundary stratum,
$$ \textstyle
\partial_1\cB^{+,-}_{\rm SFT}(A) \;=\; \{0\}\times \cB^{+,-}_{\rm GW}(A)  \;\; \sqcup \;\;
\bigsqcup_{\begin{smallmatrix}
\scriptscriptstyle \g \in \cP(H)\\
\scriptscriptstyle A_- + A_+ = A\\
  \end{smallmatrix}} l_{\g,A_\pm}\bigl( \partial_0\cB^+_{\rm SFT} (\g;A_+)\times \partial_0\cB^-_{\rm SFT} (\g; A_-) \bigr). 
$$

\item The immersions $l_{\g,A_\pm}$ are compatible with the evaluation maps, bundles, and sections -- as required for the construction \cite{fh-sft-full} of coherent perturbations for SFT, that is:
\begin{itemize}
\item[(a)]
The boundary restriction of the evaluation maps $\ov\ev^\pm|_{\{0\}\times \cB_{\rm GW}(A)\subset\partial\cB_{\rm SFT}(A)}$ coincides with the Gromov-Witten evaluation maps $\ov\ev^\pm:\cB_{\rm GW}(A)\to \ov{\CP^1_\infty}$, and the same holds for their sc$^\infty$ restriction 
$\ev^+\times \ev^-|_{\{0\}\times \cB^{+,-}_{\rm GW}(A)\subset\partial\cB^{+,-}_{\rm SFT}(A)}
= \ev^+\times \ev^- : \cB^{+,-}_{\rm GW}(A) \to \C^+ \times M\times \C^-\times M$ with values in $\C^\pm\subset\ov{\CP^1_\infty}=\C^+\sqcup S^1 \sqcup \C^-$. 
The restriction of $\ov\ev^\pm:\cB_{\rm SFT}(A)\to \ov{\CP^1_\infty}$ to each boundary face $\im l_{\g,A_\pm}\subset\partial\cB_{\rm SFT}(A)$ takes values in $\ov{\C^\pm}\subset \ov{\CP^1_\infty}$, and its pullback under $l_{\g,A_\pm}$ coincides with  
$\ov\ev^\pm: \cB^\pm_{\rm SFT} (\g;A_\pm) \to \ov{\C^\pm}\times M$. 
Moreover, pullback of the restricted sc$^\infty$ evaluations $\ev^+\times\ev^-:\cB^{+,-}_{\rm SFT}(A)\to \C^+\times M\times\C^-\times M$ under $l_{\g,A_\pm}$ coincides with 
$\ev^+\times\ev^-: \cB^{+,\C}_{\rm SFT} (\g;A_+)\times \cB^{-,\C}_{\rm SFT} (\g; A_-) \to \C^+\times M \times  \C^-\times M$. 
\item[(b)]
The restriction of $\sigma_{\rm SFT}$ to $\cF=\{0\}\times \cB_{\rm GW}(A) \subset \partial\cB_{\rm SFT}(A)$ equals to $\sigma_{\rm GW}$ via a natural identification $\cE_{\rm SFT}(A)|_\cF \cong \cE_{\rm GW}(A)$. This identification reverses the orientation of sections. 
\item[(c)]
The restriction of $\sigma_{\rm SFT}$ to each face
$\cF=\cB^+_{\rm SFT} (\g;A_+)\times \cB^-_{\rm SFT} (\g; A_-)  \subset \partial\cB_{\rm SFT}(A)$ is related by pullback to $\s^+_{\rm SFT}\times \s^-_{\rm SFT}= \sigma_{\rm SFT}\circ l_{\g,A_\pm}$ via a natural identification $l_{\g,A_\pm}^*\cE_{\rm SFT}(A) \cong \cE^+_{\rm SFT} (\g; A_+) \times  \cE^-_{\rm SFT} (\g; A_-)$. This identification preserves the orientation of sections.
\end{itemize}
\end{enumilist}
\end{ass}

\subsection{Coherent perturbations for chain map identity} \label{ssec:chain}

In this section we prove Theorem~\ref{thm:main}~(i), that is we construct $\iota_{\ul\kappa^\iota}$ in Definition~\ref{def:iota,h} as a chain map on the Morse complex \eqref{eq:CM} with differential $d:CM\to CM$ given by \eqref{eq:MorseDiff}. 
This requires the following construction of the perturbations $\ul\kappa^\iota$ that is coherent in the sense that it is compatible with the boundary identifications of the polyfolds $\cB^\iota(p_-,p_+;A)$ in Lemma~\ref{lem:iota boundary}. 
Here we will indicate smooth levels by adding $\infty$ as superscript -- denoting e.g.\ $\cX^{\iota,\infty}_{p_-,p_+; A}$ as the smooth level of an ep-groupoid representing $\cB^\iota(p_-,p_+;A)_\infty$. 

\begin{lem} \label{lem:iota chain}
There is a choice of $( \kappa^\iota_\alpha )_{\alpha\in\cI}$ in Corollary~\ref{cor:regularize} for $\cI=\{(p_-,p_+;A)\,|\, p_\pm\in\Crit(f), A\in H_2(M) \}$ that is coherent
w.r.t.\ the identifications in Lemma~\ref{lem:iota boundary} in the following sense. 
\begin{itemlist}
\item
Each $\kappa_{\alpha}^\iota:\cW^\iota_{\alpha}\to\Q^+$ for $\alpha\in\cI$ is an admissible sc$^+$-multisection of a strong bundle $P_\alpha: \cW^\iota_{\alpha} \to \cX^\iota_{\alpha}$ that is in general position to a sc-Fredholm section functor $S^\iota_{\alpha}: \cX^\iota_{\alpha} \to \cW^\iota_{\alpha}$ which represents $\sigma^\iota_{\alpha}|_{\cV_{\alpha}}$ on an open neighbourhood $ \cV_{\alpha}\subset \cB^\iota(\alpha)$ of the zero set $\sigma_{\alpha}^{-1}(0)$. 
\item
The identification of top boundary strata in Lemma~\ref{lem:iota boundary} holds for the representing ep-groupoids, 
$$\textstyle
\partial_1\cX^{\iota,\infty}_{p_-,p_+; A} \;\cong\; \bigcup_{q\in{\rm Crit}(f)} \cM(p_-,q) \times \partial_0\cX^{\iota,\infty}_{q,p_+; A} \quad \sqcup \quad \bigcup_{q\in{\rm Crit}(f)} \partial_0\cX^{\iota,\infty}_{p_-,q; A} \times \cM(q,p_+) , 
$$
and the oriented section functors $S^\iota_\alpha: \cX^\iota_\alpha \to \cW^\iota_\alpha$ are compatible with these identifications in the sense that the restriction of $S^\iota_{p_-,p_+; A}$ to any face $\cF^\infty_{(p_-,q_-),\alpha'} := \cM(p_-,q_-) \times \partial_0\cX^{\iota,\infty}_{\alpha'} \subset  \partial_1\cX^{\iota,\infty}_{p_-,p_+; A}$ resp.\ 
$\cF^\infty_{\alpha',(q_+,p_+)}:= \partial_0\cX^{\iota,\infty}_{\alpha'} \times \cM(q_+,p_+)\subset  \partial_1\cX^{\iota,\infty}_{p_-,p_+; A}$ for another $\alpha'\in\cI$ coincides on the smooth level with the pullback $S^\iota_\alpha|_{\cF^\infty}=\pr_{\cF}^*S^\iota_{\alpha'}|_{\cF^\infty}$ of $S^\iota_{\alpha'}$ via the projection 
$\pr_{\cF}:  \cF=\cF_{(p_-,q_-),\alpha'} := \cM(p_-,q_-) \times \partial_0\cX^{\iota}_{\alpha'}  \to \cX^{\iota}_{\alpha'}$ 
resp.\ 
$\pr_{\cF}:  \cF=\cF_{\alpha',(q_+,p_+)}:= \partial_0\cX^{\iota}_{\alpha'} \times \cM(q_+,p_+)
\to \cX^{\iota}_{\alpha'}$. 
\item
Each restriction $\kappa_{\alpha}^\iota|_{P_\alpha^{-1}(\cF^\infty)}$ to a face 
$\cF^\infty=\cF^\infty_{(p_-,q_-),\alpha'}$ resp.\ $\cF^\infty=\cF^\infty_{\alpha',(q_+,p_+)}$ is given by pullback $\kappa_{\alpha}^\iota|_{P_\alpha^{-1}(\cF^\infty)}=\kappa^\iota_{\alpha'}\circ \pr_\cF^*$ via the identification $P_\alpha^{-1}(\cF^\infty)\cong \pr_\cF^*\cW^\iota_{\alpha'}|_{\partial_0\cX^{\iota,\infty}_{\alpha'}}$ and natural map $\pr_\cF^*: \pr_\cF^*\cW^\iota_{\alpha'} \to \cW^\iota_{\alpha'}$.
\end{itemlist}
For any such choice of $\underline\kappa^\iota=( \kappa^\iota_\alpha )_{\alpha\in\cI}$, the resulting map $\iota_{\underline\kappa^\iota}:CM \to CM$  in Definition~\ref{def:iota,h} satisfies
$\iota_{\underline\kappa^\iota} \circ \rd + \rd \circ \iota_{\underline\kappa^\iota} = 0$.
By setting $\iota\la p\ra :=(-1)^{|p|} \iota _{\underline\kappa^\iota}\la p\ra$ 
we then obtain a chain map $\iota:C_*M\to C_*M$, that is $\iota \circ \rd = \rd \circ \iota$. 
\end{lem}

\begin{proof} 
We will first assume the claimed coherence and discuss the algebraic consequences up to signs, then construct the coherent data, and finally use this construction to compute the orientations.  

\medskip
\noindent
{\bf Construction of chain map:}
Assuming $\iota_{\underline\kappa^\iota} \circ \rd + \rd \circ \iota_{\underline\kappa^\iota} = 0$, recall that $d$ decreases the degree on the Morse complex \eqref{eq:CMgraded} by $1$. Thus $\iota:C_*M\to C_*M$ defined as above satisfies for any $q\in{\rm Crit}(f)$ 
\begin{align*}
( \iota \circ \rd - \rd \circ \iota ) \la q \ra
 \;=\;
(-1)^{|q|-1} \iota_{\underline\kappa^\iota} ( \rd \la q \ra )  -  \rd (  (-1)^{|q|} \iota_{\underline\kappa^\iota} \la q \ra ) 
 \;=\;
(-1)^{|q|-1} \bigl( \iota_{\underline\kappa^\iota} \circ \rd  +  \rd \circ \iota_{\underline\kappa^\iota} \bigr) \la q \ra  
\;=\;0 . 
\end{align*}
By $\Lambda$-linearity this proves $\iota\circ \rd = \rd\circ\iota$ on $C_*M$.

\medskip
\noindent
{\bf Proof of identity:}
To prove $\iota_{\underline\kappa^\iota} \circ \rd + \rd \circ \iota_{\underline\kappa^\iota} = 0$ note that both $\iota_{\underline\kappa^\iota}$ and $\rd$ are $\L$-linear, so the claimed identity is equivalent to the collection of identities $(\iota_{\underline\kappa^\iota} \circ \rd) \la p_-\ra + (\rd \circ \iota_{\underline\kappa^\iota}) \la p_-\ra = 0$ for all generators $p_-\in{\rm Crit}(f)$. 
That is we wish to verify 
$$
\sum_{\begin{smallmatrix}
\scriptstyle q, p_+, A  \\ \scriptscriptstyle I^\iota(q,p_+;A)=0  \\ \scriptscriptstyle |q|=|p_-|-1
 \end{smallmatrix} }
\hspace{-4mm}
\#\cM(p_-,q) \cdot \#Z^{\ul\kappa^\iota}(q,p_+; A) \cdot T^{\o(A)} \la p_+ \ra
\; + \; \hspace{-5mm}
\sum_{\begin{smallmatrix}
\scriptstyle  q, p_+, A  \\ \scriptscriptstyle I^\iota(p_-,q;A)=0   \\ \scriptscriptstyle |p_+|=|q|-1
 \end{smallmatrix} }
 \hspace{-4mm}
\#Z^{\ul\kappa^\iota}(p_-,q; A) \cdot \#\cM(q,p_+) \cdot T^{\o(A)} \la p_+ \ra 
\; = 0.
$$
Here, by the index formula \eqref{iota index}, both sides can be written as sums over $p_+\in\Crit(f)$ and $A\in H_2(M)$ for which $I^\iota(p_-,p_+;A)=1$. Then it suffices to prove for any such pair $\alpha=(p_-,p_+;A)$ with $I^\iota(\alpha)=1$
\begin{equation} \label{iota claim} \textstyle
\sum_{|q|=|p_-|-1}  \#\cM(p_-,q) \cdot \#Z^{\ul\kappa^\iota}(q,p_+; A) \;\; + \;\;
\sum_{|q|=|p_+|+1} \#Z^{\ul\kappa^\iota}(p_-,q; A) \cdot \#\cM(q,p_+) \;\; = \;\; 0 .
\end{equation}
This identity will follow by applying Corollary~\ref{cor:regularize}~(v) to the sc$^+$-multisection $\kappa_{\alpha}:\cW^\iota_{\alpha}\to\Q^+$.
Its perturbed zero set is a weighted branched $1$-dimensional orbifold $Z^{\ul\kappa^\iota}(\alpha)$,  
whose boundary is given by the intersection with the smooth level\footnote{
Here and in the following we suppress indications of the smooth level, as the perturbed zero sets automatically lie in the smooth level; see Remark~\ref{rmk:levels}.
} 
of the top boundary stratum $\partial_1\cB^\iota(\alpha) \cap \cV_{\alpha} = |{\partial_1\cX^\iota_{\alpha}}|$. 
By coherence (and with orientations discussed below) this boundary is
\begin{align*}
& \partial Z^{\ul\kappa^\iota}(\alpha) \;=\; Z^{\ul\kappa^\iota}(\alpha) \cap |\partial_1\cX^\iota_{\alpha} | \\
& \;=\; \textstyle 
\bigcup_{q\in{\rm Crit}(f)} Z^{\ul\kappa^\iota}(\alpha) \cap  \bigl(\cM(p_-,q) \times |\partial_0\cX^\iota_{q,p_+; A}| \bigr)
\;\sqcup\; 
\bigcup_{q\in{\rm Crit}(f)} Z^{\ul\kappa^\iota}(\alpha) \cap \bigl( |\partial_0\cX^\iota_{p_-,q; A}| \times \cM(q,p_+)\bigr) \\
& \;=\; \textstyle 
\bigcup_{q\in{\rm Crit}(f)} \cM(p_-,q) \times \bigl( Z^{\ul\kappa^\iota}(q,p_+; A) \cap  |\partial_0\cX^\iota_{q,p_+; A}| \bigr) \\
&\qquad\qquad\qquad\qquad\qquad\qquad\qquad\qquad\qquad\quad
\textstyle
\sqcup\; \bigcup_{q\in{\rm Crit}(f)}  \bigl( Z^{\ul\kappa^\iota}(p_-,q; A) \cap |\partial_0\cX^\iota_{p_-,q; A}| \bigr) \times \cM(q,p_+) , \\
& \;=\; \textstyle 
\bigcup_{|q|=|p_-|-1} \cM(p_-,q) \times Z^{\ul\kappa^\iota}(q,p_+; A) \quad
\sqcup \quad  \bigcup_{|q|=|p_+|+1}  Z^{\ul\kappa^\iota}(p_-,q; A) \times \cM(q,p_+) .
\end{align*}
Here the first summand of the third identification on the level of object spaces, 
\begin{align*}
& \bigl\{([\tau],x) \in \cM(p_-,q) \times \partial_0 X_{q,p_+; A} \subset \partial_1 X_{\alpha} \,\big|\, \kappa_{\alpha}(S^\iota_{\alpha}([\tau],x))>0 \bigr\} \\
&\cong 
 \bigl\{([\tau],x) \in \cM(p_-,q) \times \partial_0 X_{q,p_+; A} \,\big|\, \kappa_{q,p_+; A}(S^\iota_{q,p_+; A}(x)) >0 \bigr\} \\
&= 
\cM(p_-,q) \times \bigl\{ x \in \partial_0 X_{q,p_+; A} \,\big|\, \kappa_{q,p_+; A}(S^\iota_{q,p_+; A}(x)) >0  \bigr\}, 
\end{align*}
follows if we assume coherence of sections and multisections on the faces $\cF_{(p_-,q),\alpha'}\subset \partial_1 \cX^\iota_\alpha$, 
$$
\kappa_{\alpha}(S^\iota_{\alpha} ([\tau],x)) 
\;=\; \kappa_{\alpha}(S^\iota_{q,p_+; A}(x))
\;=\; \kappa_{q,p_+; A}(S^\iota_{q,p_+; A}(x)) .
$$
The second summand is identified similarly by assuming coherence on the faces $\cF_{\alpha',(q_-,p_+)}\subset \partial_1 \cX^\iota_\alpha$.

Finally, the fourth identification in $\partial Z^{\ul\kappa^\iota}(\alpha)$ for $\alpha=(p_-,p_+;A)$ with $I^\iota(\alpha)=1$ follows from index and regularity considerations as follows. Corollary~\ref{cor:regularize}~(iii),(iv) guarantees that the perturbed solution spaces $Z^{\ul\kappa^\iota}(\alpha')$ are nonempty only for Fredholm index $I^\iota(\alpha')\geq0$, and for $I^\iota(\alpha')=0$ are contained in the interior, $Z^{\ul\kappa^\iota}(\alpha')\subset\partial_0\cB(\alpha')$.
The Morse trajectory spaces $ \cM(p_-,q)$ resp.\ $\cM(q,p_+)$ are nonempty only for $|p_-|-|q|\geq 1$ resp.\ $|q|- |p_+| \geq 1$, so the perturbed solution spaces in the Cartesian products have Fredholm index \eqref{iota index}
$$
I^\iota(q,p_+; A) = 2c_1(A) + |q| - |p_+|   = I^\iota(p_-,p_+;A) + |q| - |p_-| =  1 + |q| - |p_-| \leq 0,
$$
and analogously $I^\iota(p_-,q; A) = I^\iota(p_-,p_+; A) + |p_+|-|q| \leq 0$. By the above regularity of the perturbed solution spaces this implies that the unions on the left hand side of the fourth identification are over $|q|=|p_-|-1$ resp.\ $|q|=|p_+|+1$ as in \eqref{iota claim}, 
and for these critical points we have the inclusions $Z^{\ul\kappa^\iota}(q,p_+; A) \subset \partial_0\cB(q,p_+; A)$ and $Z^{\ul\kappa^\iota}(p_-,q; A) \subset \partial_0\cB(p_-,q; A)$ that verify the equality. 

This finishes the identification of the boundary $\partial Z^{\ul\kappa^\iota}(\alpha)$. Now Corollary~\ref{cor:regularize}~(v) asserts that the sum of weights over this boundary is zero -- when counted with signs that are induced by the orientation of $Z^{\ul\kappa^\iota}(\alpha)$. So in order to prove the identity \eqref{iota claim} we need to compare the boundary orientation of $\partial Z^{\ul\kappa^\iota}(\alpha)$ with the orientations on the faces. We will compute the relevant signs  in \eqref{eq:orient-iota} below, after first making coherent choices of representatives $S^\iota_\alpha: \cX^\iota_\alpha \to \cW^\iota_\alpha$ of the oriented sections $\s^\iota_\alpha$, and 
constructing coherent sc$^+$-multisections $\kappa_{\alpha}^\iota:\cW^\iota_{\alpha}\to\Q^+$ for ${\alpha\in\cI}$. 

\medskip
\noindent
{\bf Coherent ep-groupoids, sections, and perturbations:}
Recall that the fiber product construction in Lemma~\ref{lem:iota,h poly} defines each bundle $\cW^\iota_\alpha=\pr_\alpha^*\cW^{\scriptscriptstyle \rm GW}_A$ for $\alpha=(p_-,p_+;A)\in\cI$ as the pullback of a strong bundle $\cW^{\scriptscriptstyle \rm GW}_A\to\cX^{\scriptscriptstyle \rm GW}_A$ under a projection of ep-groupoids -- with abbreviated notation 
$\ev_0^\pm:=\{z_0^\pm\}\times\ev :\bM(\ldots) \to \CP^1\times M$ --
$$ 
\pr_{p_-,p_+;A}\, : \;  
\cX^\iota_{p_-,p_+;A} \;=\; \bM(p_-,M) \; \leftsub{\ev^-_0}{\times}_{\ov\ev^-} \; \cX^{\scriptscriptstyle \rm GW}_A \; \leftsub{\ov\ev^+}{\times}_{\ev^+_0} \; \bM(M,p_+) \;\;\longrightarrow\;\;\cX^{\scriptscriptstyle \rm GW}_A  . 
$$ 
Moreover, the section $S^\iota_\alpha= S^{\scriptscriptstyle \rm GW}_A\circ \pr_\alpha$ is induced by the section $S^{\scriptscriptstyle \rm GW}_A:\cX^{\scriptscriptstyle \rm GW}_A\to\cW^{\scriptscriptstyle \rm GW}_A$ which cuts out the Gromov-Witten moduli space $\bM_{\rm GW}(A)=|(S^{\scriptscriptstyle \rm GW}_A)^{-1}(0)|$.
Then the identification of the top boundary stratum proceeds exactly as the proof of Lemma~\ref{lem:iota boundary}.
Coherence of the bundles and sections follows from coherence of the projections $\pr_\alpha: \cX^\iota_\alpha \to \cX^{\scriptscriptstyle \rm GW}_A$ in the sense that $\pr_\alpha|_{\cF^\infty}= \pr_{\alpha'}\circ \pr_{\cF}$ 
for all smooth levels of faces $\cF\supset \cF^\infty\subset \partial_1\cX^\iota_\alpha$ and their projections
$\pr_{\cF}:  \cF=\cF_{(p_-,q_-),\alpha'} \to \cX^\iota_{\alpha'}$ resp.\ $\pr_{\cF}:  \cF=\cF_{\alpha',(q_-,p_+)} \to \cX^\iota_{\alpha'}$. 
For example, the face $\cF=\cF_{(p_-,q_-),(q_-,p_+;A)}$ with $\cF^\infty\subset  \partial_1\cX^\iota_{p_-,p_+;A}$ identifies
\begin{align*}
\bigl( [\tau], (\tau_-,[\ul v],\tau_+) \bigr) &\;\in\; \cF^\infty_{(p_-,q_-),(q_-,p_+;A)} 
\;=\; \cM(p_-,q_-) \times \partial_0\cX^{\iota,\infty}_{q_-,p_+;A}  \\
&\;=\; \cM(p_-,q_-) \times \cM(q_-,M) \; \leftsub{\ev^-_0}{\times}_{\ov\ev^-} \; \cX^{\scriptscriptstyle \rm GW, \infty}_A \; \leftsub{\ov\ev^+}{\times}_{\ev^+_0} \; \cM(M,p_+) \\
\text{with}\qquad
\bigl( ( [\tau], \tau_-) ,[\ul v],\tau_+ \bigr)
&\;\in\; \bM(p_-,M)_1 \; \leftsub{\ev^-_0}{\times}_{\ov\ev^-} \; \cX^{\scriptscriptstyle \rm GW , \infty}_A \; \leftsub{\ov\ev^+}{\times}_{\ev^+_0} \; \cM(M,p_+)  \;\subset\; \partial_1\cX^{\iota,\infty}_{p_-,p_+;A} , 
\end{align*}
and $\pr_{p_-,p_+;A}\bigl( ( [\tau], \tau_-) ,[\ul v],\tau_+ \bigr) = [\ul v]\in \cX^{\scriptscriptstyle \rm GW}_A$ 
coincides with 
$(\pr_{q_-,p_+;A}\circ \pr_\cF) \bigl( [\tau], (\tau_-,[\ul v],\tau_+) \bigr)
= \pr_{q_-,p_+;A} (\tau_-,[\ul v],\tau_+ ) =  [\ul v]\in \cX^{\scriptscriptstyle \rm GW}_A$.
Now any choice of sc$^+$-multisections $(\l^{\scriptscriptstyle \rm GW}_A:\cW^{\scriptscriptstyle \rm GW}_A\to\Q^+)_{A\in H_2(M)}$ induces a coherent collection of sc$^+$-multisections $\bigl(\kappa^\iota_\alpha:= \l^{\scriptscriptstyle \rm GW}_A\circ \pr_\alpha^* : \pr_\alpha^*\cW^{\scriptscriptstyle \rm GW}_A\to\Q^+\bigr)_{\alpha\in\cI}$ by composition with the natural maps $\pr_\alpha^*: \pr_\alpha^*\cW^{\scriptscriptstyle \rm GW}_A \to \cW^{\scriptscriptstyle \rm GW}_A$ covering $\pr_\alpha: \cX^\iota_\alpha \to \cX^{\scriptscriptstyle \rm GW}_A$. 
Indeed, $\pr_\alpha|_{\cF}= \pr_{\alpha'}\circ \pr_{\cF}$ lifts to $\pr_\alpha^*|_{P_\alpha^{-1}(\cF^\infty)}= \pr_{\alpha'}^*\circ \pr_{\cF}^*$ so that
$$
\kappa_{\alpha}^\iota|_{P_\alpha^{-1}(\cF^\infty)}
\;=\;  \l^{\scriptscriptstyle \rm GW}_A\circ \pr_\alpha^*|_{P_\alpha^{-1}(\cF^\infty)}
\;=\; \l^{\scriptscriptstyle \rm GW}_A\circ \pr_{\alpha'}^*\circ \pr_{\cF}^*
\;=\;\kappa^\iota_{\alpha'}\circ \pr_\cF^* . 
$$

\medskip
\noindent
{\bf Construction of admissible Gromov-Witten perturbations:}
It remains to choose the sc$^+$-multisections $(\l^{\scriptscriptstyle \rm GW}_A:\cW^{\scriptscriptstyle \rm GW}_A\to\Q^+)_{A\in H_2(M)}$ so that the induced coherent collection $\ul\kappa^\iota=\bigl(\l^{\scriptscriptstyle \rm GW}_A\circ \pr_\alpha^* \bigr)_{\alpha\in\cI}$ is admissible and in general position. To do so, for each $A \in H_2(M)$ we apply Theorem~\ref{thm:transversality} to the sc-Fredholm section functor $S^{\scriptscriptstyle\rm GW}_A : \cX^{\scriptscriptstyle\rm GW}_A \rightarrow \cW^{\scriptscriptstyle\rm GW}_A$, the sc$^\infty$ submersion $\ov\ev^- \times \ov\ev^+ : \cX^{\scriptscriptstyle\rm GW}_A \rightarrow \CP^1\times M \times \CP^1\times M$, and the collection of Cartesian products of stable and unstable manifolds $\{z_0^-\}\times W_{p_-}^- \times \{z_0^+\}\times W_{p_+}^+$ for all pairs of critical points $p_-,p_+ \in \Crit(f)$.

After fixing a pair controlling compactness $(N_A,\cU_A)$ for each $A\in H_2(M)$, Theorem~\ref{thm:transversality} yields $(N_A,\cU_A)$-admissible sc$^+$-multisections $\l^{\scriptscriptstyle\rm GW}_A : \cW^{\scriptscriptstyle\rm GW}_A \rightarrow \mathbb{Q}^+$ in general position to $S_A^{\scriptscriptstyle\rm GW}$ for each $A\in H_2(M)$. Moreover, they can be chosen such that restriction of evaluations to the perturbed zero set 
$\ov\ev^- \times \ov\ev^+: Z^{\l_A^{\scriptscriptstyle\rm GW}}\to \CP^1\times M\times \CP^1 \times M$ is transverse to all of the products of unstable and stable submanifolds $\{z_0^-\}\times W_{p_-}^- \times \{z_0^+\}\times W_{p_+}^+$ for $p_-,p_+\in\Crit(f)$. 
Note that these embedded submanifolds cover the images of all evaluation maps on the compactified Morse trajectory spaces $\ev_0^-\times \ev_0^+: \bM(p_-,M)\times \bM(M,p_+) \to \CP^1\times M\times \CP^1\times M$, by construction of the evaluations $\ev:\bM(\ldots)\to M$ in \eqref{eval}, which determine $\ev_0^\pm(\ul\tau)=(z_0^\pm,\ev(\ul\tau))$. 
Thus we obtain transverse fiber products 
$\bM(p_-,M) \; \leftsub{\ev^-_0}{\times}_{\ov\ev^-} Z^{\l_A^{\scriptscriptstyle\rm GW}} \leftsub{\ov\ev^+}{\times}_{\ev^+_0} \; \bM(M,p_+)$ for every $\alpha\in\cI$. 
This translates into the pullbacks $\k^\iota_\alpha = \l^{\scriptscriptstyle \rm GW}_A\circ \pr_\alpha^*$ being in general position to the pullback sections $S^\iota_\alpha$ for $\alpha\in\cI$. Moreover, $\k^\iota_\alpha$ is admissible with respect to a pullback of $(N_A,\cU_A)$, so the perturbed zero set is a compact weighted branched orbifold for each $\alpha=(p_-,p_+;A)$, 
$$
``\bigl|(S^\iota_\alpha + \kappa^\iota_\alpha)^{-1}(0)\bigr|\text{''} \;=\;
Z^{\ul\kappa^\iota}(\alpha) \;=\;
\bM(p_-,M) \; \leftsub{\ev^-_0}{\times}_{\ov\ev^-} Z^{\l_A^{\scriptscriptstyle\rm GW}} \leftsub{\ov\ev^+}{\times}_{\ev^+_0} \; \bM(M,p_+) . 
$$
This finishes the construction of coherent perturbations.

\medskip
\noindent
{\bf Computation of orientations:} 
 To prove the identity \eqref{iota claim} it remains to compute the effect of the orientations in Remark~\ref{rmk:iota,h orient} on the algebraic identity in Corollary~\ref{cor:regularize}~(v) that arises from the boundary $\partial Z^{\ul\kappa^\iota}(\alpha)$ of the 1-dimensional weighted branched orbifolds arising from regularization of the moduli spaces with index $I^\iota(\alpha)=I^\iota(p_-,p_+;A)=1$. Here $Z^{\l_A^{\scriptscriptstyle\rm GW}}$ is of even dimension and has no boundary since the Gromov-Witten polyfolds in Assumption~\ref{ass:iota,h} have no boundary, and the index of $\s_{\rm GW}$ is even. 
For the Morse trajectory spaces, the boundary strata are determined in Theorem~\ref{thm:Morse}, with relevant orientations computed in Remark~\ref{rmk:MorseOrient}. Thus for $I^\iota(\alpha)=|p_-|-|p_+| +2c_1(A)=1$ we can compute orientations -- at the level of well defined finite dimensional tangent spaces at a solution; in whose neighbourhood the evaluation maps are guaranteed to be scale-smooth -- 
\begin{align}
\partial Z^{\ul\kappa^\iota}(\alpha) &\;=\;
\partial_1\bM(p_-,M) \; \leftsub{\ev}{\times}_{\ev}\; Z^{\l_A^{\scriptscriptstyle\rm GW}}\; \leftsub{\ev}{\times}_{\ev} \; \partial_0\bM(M,p_+)  \nonumber\\
&\quad\sqcup\;
(-1)^{\dim \bM(p_-,M)} \;
\partial_0\bM(p_-,M) \; \leftsub{\ev}{\times}_{\ev} \; Z^{\l_A^{\scriptscriptstyle\rm GW}}\; \leftsub{\ev}{\times}_{\ev} \; \partial_1\bM(M,p_+) \nonumber \\
&\;=\;
\bigl( \textstyle\bigsqcup_{q\in\Crit f} \cM(p_-,q) \times \cM(q,M) \bigr) \; \leftsub{\ev}{\times}_{\ev}\; Z^{\l_A^{\scriptscriptstyle\rm GW}} \;\leftsub{\ev}{\times}_{\ev} \; \cM(M,p_+) \label{eq:orient-iota} \\
&\quad\sqcup\;
(-1)^{|p_-| + |p_+|+1}
\cM(p_-,M) \; \leftsub{\ev}{\times}_{\ev} \; Z^{\l_A^{\scriptscriptstyle\rm GW}}\; \leftsub{\ev}{\times}_{\ev} \; \bigl( \textstyle \bigsqcup_{q\in\Crit f} \cM(M,q) \times \cM(q,p_+) \bigr) \nonumber\\
&\;=\;
\textstyle\bigsqcup_{q\in\Crit f} \cM(p_-,q) \times Z^{\ul\k^\iota}(q,p_+;A)
\;\;\sqcup\;\;  \bigsqcup_{q\in\Crit f} Z^{\ul\k^\iota}(p_-,q;A) \times \cM(q,p_+) .  \nonumber
\end{align}
Here the signs in the first equality arise from the ambient Cartesian product 
$\partial ( \bM_-\times Z \times\bM_+) \subset (-1)^{\dim(\bM_-\times Z)}  \bM_-\times Z \times\bM_+$; in the second equality we used Remark~\ref{rmk:MorseOrient}; and in the final equality we use $|p_-|+|p_+|+1 \equiv I^\iota(\alpha)=1 \equiv 0$ modulo 2.
This finishes the computation of the oriented boundaries $\partial Z^{\ul\kappa^\iota}(\alpha)$ for $I^\iota(\alpha)=1$ that proves \eqref{iota claim} and thus yields a chain map.  
\end{proof}

\subsection{Admissible perturbations for isomorphism property} \label{ssec:iso}

In this section we prove Theorem~\ref{thm:main}~(ii), i.e.\ construct $\iota=(-1)^*\iota_{\ul\kappa^\iota}:C_*M\to C_*M$ in Definition~\ref{def:iota,h} and Lemma~\ref{lem:iota chain} as a $\L$-module isomorphism on the chain complex $CM=CM_\L$ over the Novikov field as in \eqref{eq:CM}. 
This requires a construction of the perturbations $\ul\kappa^\iota$ that preserves the properties of the zero sets in Remark~\ref{rmk:iota-energy} for nonpositive symplectic area $\omega(A)\leq 0$. 

\begin{lem} \label{lem:iota triangular}
The coherent collection of sc$^+$-multisections $\ul\kappa^\iota$ in Lemma~\ref{lem:iota chain} can be chosen such that 
$\#Z^{\ul\kappa^\iota}(p_-,p_+;A)=0$ for $A\in H_2(M)\less\{0\}$ with $\o(A)\leq 0$, or for $A=0$ and $p_-\neq p_+$, and $\#Z^{\ul\kappa^\iota}(p,p;0)\neq 0$. 
As a consequence, $\iota=(-1)^*\iota_{\ul\kappa^\iota} : CM_{\L} \to CM_{\L}$ is a $\L$-module isomorphism.
\end{lem}

\begin{proof}
The sc$^+$-multisections $\ul\kappa^\iota$ in Lemma~\ref{lem:iota chain} are obtained from choices of sc$^+$-multisections $(\kappa_A:\cW^\iota_A\to\Q^+)_{A\in H_2(M)}$ that are in general position to sc-Fredholm sections $S_A:\cX^{\scriptscriptstyle \rm GW}_A\to\cW_A$ which cut out the Gromov-Witten moduli space $\bM_{\rm GW}(A)=|S_A^{-1}(0)|$, and such that moreover the evaluation maps restricted to the perturbed zero sets, $\ov\ev^-\times\ov\ev^+: Z({\kappa_A}) \to \CP^1\times M\times \CP^1\times M$ are transverse to the unstable and stable manifolds $\{z^-_0\}\times W^-_{p_-}\times \{z^+_0\}\times W^+_{p_+}\subset \CP^1\times M\times \CP^1\times M$ for any pair of critical points $p_-,p_+\in\Crit(f)$. 

We will first consider $\alpha=(p_-,p_+;A)\in\cI$ for nontrivial homology classes $A\in H_2(M)\less\{0\}$ with nonpositive symplectic area $\o(A)\leq 0$. Recall from Remark~\ref{rmk:iota-energy} that these moduli spaces are empty $|S_A^{-1}(0)|=\emptyset$, so as in Corollary~\ref{cor:regularize} we can choose empty neighbourhoods $\emptyset =|\cU_A|\subset |\cX^{\scriptscriptstyle \rm GW}_A|$ to control compactness. Then the perturbed zero set 
$Z({\kappa_A})=|\{ x\in X_A \,|\, \kappa_A(S_A(x)) >0 \}| \subset |\cU_A|$ is forced to be empty, i.e.\ $\kappa_A\circ S_A \equiv 0$. This is an allowed choice in Lemma~\ref{lem:iota chain} since evaluation maps from an empty set are trivially transverse to any submanifold. 
This choice induces for any $p_\pm\in\Crit(f)$ in $\alpha=(p_-,p_+;A)$ an induced sc$^+$-multisection $\kappa^\iota_\alpha =\kappa_A\circ \pr_\alpha^*:\cW^\iota_\alpha\to\Q^+$. Its perturbed zero set is 
$$
Z^{\kappa^\iota_\alpha}(\alpha) = \bigl| \bigl\{ (\ul\tau^- , x , \ul\tau^+ ) \in  X_{\alpha} \,\big|\, 
\kappa^\iota_{\alpha}\bigl( S^\iota_\alpha (\ul\tau^- , x , \ul\tau^+ ) \bigr) >0 
\bigr\} \bigr| =\emptyset
$$
since the coherence in Lemma~\ref{lem:iota chain} implies
$\kappa^\iota_{\alpha}\circ S^\iota_{\alpha} = \kappa_A\circ \pr_\alpha^* \circ S^\iota_{\alpha} = \kappa_A\circ S_A \circ \pr_\alpha \equiv 0$, or more concretely 
$\kappa^\iota_{(p_-,p_+;A)}\bigl( S^\iota_{p_-,p_+;A} (\ul\tau^- , x , \ul\tau^+ ) \bigr)
= \kappa_A(S_A(x)) = 0$. 
Thus we have ensured vanishing counts $\#Z^{\ul\kappa^\iota}(p_-,p_+;A)=0$ for $A\in H_2(M)\less\{0\}$ with $\o(A)\leq 0$ whenever $I^\iota(p_-,p_+;A)=0$.

Next we consider $A=0\in H_2(M)$ and recall from Remark~\ref{rmk:iota-energy} and Assumption~\ref{ass:iota,h}~(ii) that the Gromov-Witten moduli space $\bM_{\rm GW}(0) = Z(\kappa_0)$ is already compact and transversely cut out. Thus the trivial sc$^+$-multisection 
$\kappa_0:\cW_0\to \Q^+$, given by $\kappa_0(0_x)=1$ on zero vectors $0_x\in (\cW_0)_x$ and $\kappa_0|_{(\cW_0)_x \less \{0_x\}} \equiv 0$, is an admissible sc$^+$-multisection in general position to $S_0:\cX^{\scriptscriptstyle \rm GW}_0\to \cW_0$. 
Recall moreover that the evaluation maps on the unperturbed zero set are 
$$
\ov\ev^-\times \ov\ev^+ : Z(\k_0)\simeq \CP^1\times\CP^1\times M \to \CP^1\times M \times\CP^1\times M, \qquad (z^-,z^+,x)\mapsto (z^-,x,z^+,x) . 
$$ 
In the $\CP^1$-factors this is submersive so transverse to the fixed points $(z_0^-,z_0^+)\in\CP^1\times\CP^1$. 
In the $M$-factors this is the diagonal map, which is transverse to the unstable and stable manifolds $W^-_{p_-}\times W^+_{p_+}\subset M\times M$ for any pair $p_-,p_+\in\Crit(f)$ by the Morse-Smale condition on the metric on $M$ chosen in \S\ref{sec:Morse}. 
Thus the trivial multisection $\kappa_0$ is in fact an allowed choice in Lemma~\ref{lem:iota chain}. 
Now with this choice, the tuples $(p_-,p_+;0)\in\cI$ for which we need to compute 
\begin{align*}
\#Z^{\ul\kappa^\iota}(p_-,p_+;0) & \;=\; 
\#\ \left| \left\{ (\ul\tau^- , [v] , \ul\tau^+ ) \in  \bM(p_-,M) \times Z({\kappa_0}) \times \bM(M,p_+)
\, \big|\, \bigl( z_0^\pm, \ev(\ul\tau^\pm) \bigr) = \ov\ev^\pm([v])  \right\} \right| \\
&\;\cong\;  \#\ \left| \left\{ (\ul\tau^- , \ul\tau^+ ) \in  \bM(p_-,M) \times \bM(M,p_+)
\,\left|\, \ev(\ul\tau^-)= \ev(\ul\tau^+) \right. \right\} \right|
\end{align*}
are those with $0=I^\iota(p_-,p_+;0)= 2 c_1(0) + |p_-| - |p_+|$, i.e.\ $|p_-|=|p_+|$. 
These are the fiber products identified in Remark~\ref{rmk:MorseOrient}~(ii) as either empty or a one point set,
\begin{align*}
\bM(p_-,M) \leftsub{\ev}{\times}_\ev \bM(M,p_+)
\;=\; \begin{cases}
\quad \emptyset &; p_-\neq p_+ , \\
(\tau^-\equiv p_-, \tau^+\equiv p_+) &; p_-= p_+ . 
\end{cases}
\end{align*}
Thus we have counts $\#Z^{\ul\kappa^\iota}(p_-,p_+;0)=0$ for $p_-\neq p_+$ and 
$\#Z^{\ul\kappa^\iota}(p,p;0) \neq 0$ for each $p\in\Crit(f)$.

Finally, we will use these computations of $\#Z^{\ul\kappa}(p_-,p_+;A)$ for $\o(A)\leq 0$ to prove that the resulting map $\iota:=(-1)^* \iota_{\ul\kappa^\iota} : CM_{\L} \to CM_{\L}$ is a $\L$-module isomorphism. For that purpose we choose an arbitrary total order of the critical points $\Crit(f)=\{p_1, \ldots, p_\ell\}$ and for $i,j\in\{1,\ldots,\ell\}$ denote the coefficients of 
$\i(\langle p_j \rangle) = \sum_{i=1}^\ell \lambda^{ij} \langle p_i \rangle$ by $\lambda^{ij}  \in \Lambda$.
We claim that the $(\ell \times \ell)$-matrix with entries $\lambda^{ij}=\sum_{r\in\G} \lambda^{ij}_r T^r$ satisfies the conditions of Lemma~\ref{lem:invertibilitymatrixnovikov}. To check this recall that we have by construction in Definition~\ref{def:iota,h} and change of signs in Lemma~\ref{lem:iota chain}
$$ \textstyle
\lambda^{ij}_r \;=\; \sum_{\begin{smallmatrix}
 A \in H_2(M) , \o(A)=r \\  I^\iota(p_j,p_i;A)=0  \end{smallmatrix} }  
\;(-1)^{|p_j|} \;\#Z^{\ul\k^\iota}(p_j,p_i; A) . 
$$
For $r<0$ we obtain $\lambda^{ij}_r=0$ since each coefficient $\#Z^{\ul\k^\iota}(p_j,p_i; A)=0$ vanishes for $\o(A)=r<0$. 
For $r=0$ and $i\neq j$ we also have $\lambda^{ij}_0=0$ since $\#Z^{\ul\k^\iota}(p_j,p_i; A)=0$ also holds for $\o(A)=0$ and $p_j\neq p_i$. Finally, for $r=0$ and $i=j$ we use $\#Z^{\ul\k^\iota}(p_j,p_i; A)=0$ for $A\neq 0$ with $\o(A)=0$ to compute $\lambda^{ii}_0 = \#Z^{\ul\k^\iota}(p_j,p_i; 0) \neq 0$.  
This confirms that Lemma~\ref{lem:invertibilitymatrixnovikov} applies, and thus $\i\cong (\lambda_{ij})_{1\leq i,j\leq\ell}$ is invertible. This finishes the proof.
\end{proof}

\subsection{Coherent perturbations for chain homotopy} \label{ssec:homotopy}

In this section we prove Theorem~\ref{thm:main}~(iii) by constructing $h_{\ul\kappa}: CM \rightarrow CM$ in Definition~\ref{def:iota,h} as a chain homotopy between $SSP_{\ul\kappa^+} \circ PSS_{\ul\kappa^-}$ and $\iota_{\ul\kappa^\iota}$ from Definitions~\ref{def:PSS},\ref{def:iota,h}, with appropriate sign adjustments as in Lemma~\ref{lem:iota chain}. 
This requires a coherent construction of perturbations $\ul\kappa,\ul\kappa^\iota,\ul\kappa^-,\ul\kappa^+$ over the indexing sets 
\begin{align*}
\cI=\cI^\iota &\,:=\;  \bigl\{ \alpha=(p_-,p_+,A) \,\big|\, p_-,p_+\in{\rm Crit}(f) , A\in H_2(M) \bigr\} ,\\
\cI^+ &\,:=\; \bigl\{ \alpha=(p,\g,A) \,\big|\, p\in{\rm Crit}(f) , \g\in\cP(H) , A\in H_2(M) \bigr\}, \\
\cI^- &\,:=\;  \bigl\{ \alpha=(\g,p,A) \,\big|\, p\in{\rm Crit}(f) , \g\in\cP(H) , A\in H_2(M) \bigr\} . \end{align*}
Here we will use notation from Lemma~\ref{lem:Cartesian} for Cartesian products of multisections.

\begin{lem} \label{lem:h chain}
There is a choice of $\ul\kappa^+=(\kappa^+_\alpha)_{\alpha\in\cI^+},\ul\kappa^-=(\kappa^-_\alpha)_{\alpha\in\cI^-},\ul\kappa^\iota=(\kappa^\iota_\alpha)_{\alpha\in\cI},\ul\kappa=(\kappa_\alpha)_{\alpha\in\cI}$ in Definitions~\ref{def:PSS},\ref{def:iota,h} that is coherent in the following sense. 
\begin{enumilist}
\item
Each $\kappa_{\alpha}^{\cdots}:\cW_{\alpha}^{\cdots}\to\Q^+$ for $\alpha\in\cI^+\sqcup\cI^-\sqcup\cI^\iota\sqcup\cI$ is an admissible sc$^+$-multisection of a strong bundle $P^{\cdots}_\alpha: \cW^{\cdots}_{\alpha} \to \cX^{\cdots}_{\alpha}$ that is in general position to a sc-Fredholm section functor $S^{\cdots}_{\alpha}: \cX^{\cdots}_{\alpha} \to \cW^{\cdots}_{\alpha}$ which represents $\sigma^{\cdots}_{\alpha}|_{\cV^{\cdots}_{\alpha}}$ on an open neighbourhood $ \cV^{\cdots}_{\alpha}\subset \cB^{\cdots}(\alpha)$ of the zero set ${\sigma^{\cdots}_{\alpha}}^{-1}(0)$. 
The tuple $\ul\k^\iota=(\kappa^\iota_\alpha)_{\alpha\in\cI^\iota}$ satisfies the conclusions of Lemma~\ref{lem:iota chain} and \ref{lem:iota triangular}.
\item
The smooth level of the first boundary stratum of $\cX_{p_-,p_+, A}$ for every $(p_-,p_+,A)\in\cI$ is naturally identified -- on the level of object spaces, and compatible with morphisms -- with 
\begin{align}\textstyle  \nonumber
\partial_1\cX_{p_-,p_+, A}^\infty &\;\cong\; 
\partial_0\cX^{\iota,\infty}_{p_-,p_+, A} \quad \sqcup \bigcup_{\g\in\cP(H), A=A_-+A_+} 
 \partial_0\cX^{+,\infty}_{p_-,\g, A_+} \times  \partial_0\cX^{-,\infty}_{\g,p_+, A_-} \\
&\qquad\sqcup\quad 
\bigcup_{q\in{\rm Crit}(f)} \cM(p_-,q) \times \partial_0\cX^\infty_{q,p_+, A} \quad \sqcup \quad \bigcup_{q\in{\rm Crit}(f)} \partial_0\cX^\infty_{p_-,q, A} \times \cM(q,p_+) , 
 \label{eq:homotopy bdy}
\end{align}
and the oriented section functors $S_\alpha^{\cdots}$ are compatible with these identifications in the sense that the restriction of $S_{p_-,p_+, A}$ to any of these faces $\cF^\infty\subset \partial_1\cX^\infty_{p_-,p_+, A}$ is given by pullback $S_{p_-,p_+, A}|_{\cF^\infty}=\pr_\cF^*S_\cF$ of another sc-Fredholm section of a strong bundle over an ep-groupoid $S_\cF: \cX_\cF\to\cW_\cF$ given by 
$S_{q,p_+, A}$, $S_{p_-,q, A}$, $S^\iota_{p_-,p_+, A}$, resp.\ 
$$
S_\cF= S^+_{p_-,\g, A_+} \times S^-_{\g,p_+, A_-} \,:\; 
\cX^+_{p_-,\g, A_+} \times \cX^-_{\g,p_+, A_-}  \;\to\; \cW^+_{p_-,\g, A_+} \times \cW^-_{\g,p_+, A_-} 
$$
via the projection $\pr_\cF: \cF\to \cX_\cF$ given by the natural maps
\begin{align*}
\cM(p_-,q) \times \partial_0\cX_{q,p_+, A} &\;\to\;\cX_{q,p_+, A} ,\qquad
\partial_0\cX^\iota_{p_-,p_+, A}  \;\to\; \cX^\iota_{p_-,p_+, A}, \\
\partial_0\cX_{p_-,q, A} \times \cM(q,p_+) &\;\to\; \cX_{p_-,q, A}, \qquad
\partial_0\cX^+_{p_-,\g, A_+} \times  \partial_0\cX^-_{\g,p_+, A_-}  \;\to\;  \cX^+_{p_-,\g, A_+} \times  \cX^-_{\g,p_+, A_-} .
\end{align*}

\item
Each restriction $\kappa_{\alpha}|_{P_\alpha^{-1}(\cF^\infty)}$ for $\alpha=(p_-,p_+,A)\in\cI$ to one of the faces $\cF^\infty\subset\partial_1\cX_\alpha$ is given via the identification $P_\alpha^{-1}(\cF^\infty)\cong \pr_\cF^*\cW_\cF|_{\partial_0\cX_\cF}$ and natural map $\pr_\cF^*: \pr_\cF^*\cW_\cF \to \cW_\cF$ by 
$$
\kappa_{\alpha}|_{P_\alpha^{-1}(\cF^\infty)} \;=\; 
\begin{cases}
\kappa_{q,p_+, A} \circ \pr_\cF^*
&\qquad\text{for}\;\; \cF=\cM(p_-,q) \times \partial_0\cX_{q,p_+, A} , \\
\kappa_{p_-,q, A} \circ \pr_\cF^*
&\qquad\text{for}\;\; \cF=\partial_0\cX_{p_-,q, A} \times \cM(q,p_+) ,\\
 \kappa^\iota_{p_-,p_+, A} \circ \pr_\cF^*
&\qquad\text{for}\;\; \cF=\partial_0\cX^\iota_{p_-,p_+, A}  , \\
(\kappa^+_{p_-,\g, A_+}\cdot \kappa^-_{\g,p_+, A_-}) \circ \pr_\cF^*  
&\qquad\text{for}\;\; \cF=\partial_0\cX^+_{p_-,\g, A_+} \times  \partial_0\cX^-_{\g,p_+, A_-} .\end{cases}
$$
\end{enumilist}
For any such choice of $\underline\kappa^\iota=( \kappa^\iota_\alpha )_{\alpha\in\cI}$,
the resulting maps $PSS_{\underline\kappa^+}, SSP_{\underline\kappa^-}, \iota_{\underline\kappa^\iota}, h_{\underline\kappa}$ in Definitions~\ref{def:PSS}, \ref{def:iota,h} satisfy
$(-1)^{|p|}\iota_{\underline\kappa^\iota}\la p \ra  =  (-1)^{|p|} SSP_{\underline\kappa^-} \bigl( PSS_{\underline\kappa^+} \la p \ra \bigr) 
+ h_{\underline\kappa} ( \rd \la p \ra ) + \rd ( h_{\underline\kappa} \la p \ra )$,
where $\rd$ is the Morse differential from \S\ref{sec:Morse}.
By setting $\iota\la p\ra :=(-1)^{|p|} \iota_{\underline\kappa^\iota}\la p\ra$ as in Lemma~\ref{lem:iota chain},  $PSS\la p\ra :=(-1)^{|p|} PSS_{\underline\kappa^+}\la p\ra$, $SSP:=SSP_{\underline\kappa^-}$, and $h:=h_{\underline\kappa}$ we then obtain a chain homotopy between $\iota$ and $SSP \circ PSS$, that is $\iota - SSP \circ PSS = \rd \circ h + h \circ d$. 
\end{lem}

\begin{proof}
This proof is similar to Lemma~\ref{lem:iota chain}, with more complicated combinatorics of the boundary faces due to the boundary of $\cB_{\rm SFT}$ described in Assumption~\ref{ass:iso}, and presented in different order: We will first make the coherent constructions and then deduce the algebraic consequences.

\medskip
\noindent
{\bf Coherent ep-groupoids and sections:}
To construct coherent representatives $S_\alpha^{\cdots}: \cX_\alpha^{\cdots} \to \cW_\alpha^{\cdots}$ for $\alpha\in\cI^+\sqcup\cI^-\sqcup\cI^\iota\sqcup\cI$ as claimed in (ii) recall that the fiber product construction in Lemma~\ref{lem:iota,h poly} defines each bundle $\cW_\alpha=\pr_\alpha^*\cW^{\scriptscriptstyle \rm SFT}_A$ for $\alpha=(p_-,p_+,A)\in\cI$ as the pullback of a strong bundle $P_A:\cW^{\scriptscriptstyle \rm SFT}_A\to\cX^{\scriptscriptstyle \rm SFT}_A$ under the natural projection of ep-groupoids 
$$
\pr_{p_-,p_+,A}\, : \;  
\cX_{p_-,p_+,A} \;=\; \bM(p_-,M) \; \leftsub{\ev^+_0}{\times}_{\ev^+} \; \cX^{\scriptscriptstyle \rm SFT}_A \; \leftsub{\ev^-}{\times}_{\ev^-_0} \; \bM(M,p_+) \;\;\longrightarrow\;\;\cX^{\scriptscriptstyle \rm SFT}_A  .
$$ 
Here $\ev_0^\pm: \bM(\ldots) \to \C^\pm\times M, \ul\tau\mapsto (0, \ev(\ul\tau) )$ arise from Morse evaluation \eqref{eval}.
The ep-groupoid $\cX^{\scriptscriptstyle \rm SFT}_A\subset\Ti\cX^{\scriptscriptstyle \rm SFT}_A$ is a  full subcategory -- determined by the open subset $\cB^{+,-}_{\rm SFT}(A)=(\ov\ev^+)^{-1}(\C^+\times M) \cap (\ov\ev^-)^{-1}(\C^-\times M) \subset\cB_{\rm SFT}(A)$ --
of an ep-groupoid $\Ti\cX^{\scriptscriptstyle \rm SFT}_A$ from Assumption~\ref{ass:iota,h} that represents $\cB_{\rm SFT}(A)$ and thus contains the compactified SFT neck stretching moduli space $\bM_{\rm SFT}(A)=|(S^{\scriptscriptstyle \rm SFT}_A)^{-1}(0)|$ as zero set of a sc-Fredholm section $S^{\scriptscriptstyle \rm SFT}_A:\Ti\cX^{\scriptscriptstyle\rm SFT}_A\to \Ti\cW^{\scriptscriptstyle \rm SFT}_A$.
We will work with both groupoids: Multisection perturbations are constructed over $\Ti\cX^{\scriptscriptstyle \rm SFT}_A$ since we need a compact zero set to specify the admissibility that guarantees preservation of compactness under perturbations -- both for $S^{\scriptscriptstyle \rm SFT}_A$ and its fiber product restrictions $S_\alpha$. On the other hand, $|\Ti\cX^{\scriptscriptstyle \rm SFT}_A|=\cB_{\rm SFT}(A)$ has more complicated boundary than $\cB^{+,-}_{\rm SFT}(A)$ -- due to the distribution of marked points into building levels -- and does not support a sc$^\infty$ evaluation map. Thus we discuss coherence only over subgroupoids $\cX^{\scriptscriptstyle \rm SFT}_A\subset\Ti\cX^{\scriptscriptstyle \rm SFT}_A$ with the boundary stratification of $\cB^{+,-}_{\rm SFT}(A)$, and which support sc$^\infty$ functors $\ev^\pm: \cX^{\scriptscriptstyle \rm SFT}_A\to\C^\pm\times M$ representing the evaluation maps \eqref{eq:SFTevalB}. 
Here we may even use subgroupoids $\cX^{\scriptscriptstyle \rm SFT}_A$ representing a smaller open subset $(\ov\ev^+)^{-1}(\D_r^+\times M) \cap (\ov\ev^-)^{-1}(\D_r^-\times M) \subset\cB_{\rm SFT}(A)$ of preimages of the disks $\D_r^\pm:=\{z\in\C^\pm\,|\, |z|<r\}\subset\C^\pm$, which contain the standard marked points $z_0^\pm\cong 0\in\C^\pm$.\footnote{These disks should not be confused with the closed disks $D_\pm$ in the construction of $\CP^1_R$, as e.g.\ $\D^+\subset \C^+ \cong (D_+\sqcup [-R,0)\times S^1)/\sim_R$ is a precompact subset of the first hemisphere in $\CP^1_R\cong \C^+\cup S^1 \cup \C^-$ for any $R\geq0$.} 
The polyfold structure on the fiber products $\cX_\alpha$ in Lemma~\ref{lem:iota,h poly} is independent of the choice of open neighbourhood in $\cB_{\rm SFT}(A)$ of the subset satisfying the fiber product condition. 
After obtaining the subgroupoid $\cX^{\scriptscriptstyle \rm SFT}_A\subset\Ti\cX^{\scriptscriptstyle \rm SFT}_A$ from such an open subset, we obtain the bundle $\cW^{\scriptscriptstyle \rm SFT}_A=\Ti \cW^{\scriptscriptstyle \rm SFT}_A|_{\cX^{\scriptscriptstyle \rm SFT}_A}$ and section $S^{\scriptscriptstyle \rm SFT}_A|_{\cX^{\scriptscriptstyle \rm SFT}_A}:\cX^{\scriptscriptstyle \rm SFT}_A\to \cW^{\scriptscriptstyle \rm SFT}_A$ by restriction. 
Finally, each section $S_\alpha= S^{\scriptscriptstyle \rm SFT}_A\circ \pr_\alpha$ is induced by the above projection $\pr_\alpha:\cX_\alpha\to \cX^{\scriptscriptstyle \rm SFT}_A\subset \Ti\cX^{\scriptscriptstyle \rm SFT}_A$. 

Next, restriction to the boundary faces given in Assumption~\ref{ass:iso}~(i) induces representatives $S^{\scriptscriptstyle \rm GW}_A : \Ti \cX^{\scriptscriptstyle \rm GW}_A\to \Ti \cW^{\scriptscriptstyle \rm GW}_A$ resp.\ $S^\pm_{\g,A_\pm}: \Ti \cX^\pm_{\g,A_\pm}\to \Ti \cW^\pm_{\g,A_\pm}$ of the sections $\s_{\rm GW}: \cB_{\rm GW}(A)\to \cE_{\rm GW}(A)$ resp.\ $\s_{\rm SFT}: \cB^\pm_{\rm SFT}(\g; A_{\pm})\to \cE^\pm_{\rm SFT}(\g;A_{\pm})$ from Assumption~\ref{ass:pss} resp.\ \ref{ass:iota,h}. 
Moreover, the boundary of the open subset $(\ov\ev^+\times \ov\ev^-)^{-1}(\D_r^+\times M \times \D_r^-\times M)$ for $0<r\leq \infty$ (with $\D^\pm_\infty:=\C^\pm$) yields subgroupoids $\cX^{\scriptscriptstyle \rm GW}_A \subset \Ti \cX^{\scriptscriptstyle \rm GW}_A$ representing $(\ov\ev^+\times \ov\ev^-)^{-1}(\D_r^+\times M \times \D_r^-\times M)\subset \cB_{\rm GW}(A)$ resp.\ $\cX^\pm_{\g,A_\pm}\subset\Ti \cX^\pm_{\g,A_\pm}$ 
representing $(\ov\ev^\pm)^{-1}(\D_r^\pm\times M)\subset \cB^\pm_{\rm SFT}(A)$, along with restricted sections $S^{\scriptscriptstyle \rm GW}_A:\cX^{\scriptscriptstyle \rm GW}_A\to \cW^{\scriptscriptstyle \rm GW}_A=\Ti\cW^{\scriptscriptstyle \rm GW}_A|_{\cX^{\scriptscriptstyle \rm GW}_A}$ resp.\ $S^\pm_{\g,A_\pm}: \cX^\pm_{\g,A_\pm}\to \cW^\pm_{\g,A_\pm}=\Ti \cW^\pm_{\g,A_\pm}|_{\cX^\pm_{\g,A_\pm}}$. 
Then the evaluation maps restrict to sc$^\infty$ functors $\ev^\pm: \cX^{\scriptscriptstyle \rm GW}_A\to\D_r^\pm\times M$ resp.\ $\ev^\pm: \cX^\pm_{\g,A}\to\D_r^\pm\times M$, which yield -- again independent of $r>0$ -- the fiber product construction of $\cB^\pm(\alpha)$ in Lemma~\ref{lem:PSS poly}, and of $\cB^\iota(\alpha)$ in Lemma~\ref{lem:iota,h poly}.

Now the identification of the top boundary strata $\partial_1\cX^\infty_{p_-,p_+,A}$ will proceed similar to the proof of Lemma~\ref{lem:iota boundary} with $\cB_{\rm GW}(A)$ replaced by $\cB^{+,-}_{\rm SFT}(A)$, apart from the fact that the SFT polyfold has boundary. 
This boundary is identified in Assumption~\ref{ass:iso}~(ii) as
\begin{equation}\label{eq:SFTfaces} \textstyle
\partial_1\cX^{\scriptscriptstyle \rm SFT}_A \;\cong\; \cX^{\scriptscriptstyle \rm GW}_A  \;\; \sqcup \;\;
\bigsqcup_{\begin{smallmatrix}
\scriptscriptstyle \g \in \cP(H)\\
\scriptscriptstyle A_- + A_+ = A\\
  \end{smallmatrix}} \partial_0\cX^+_{\g,A_+} \times \partial_0\cX^-_{\g, A_-} .
\end{equation}
By the fiber product construction \cite[Cor.7.3]{Ben-fiber} of $\cB(p_-,p_+;A)$ in Lemma~\ref{lem:iota,h poly}, the degeneracy index satisfies $d_{\cB(p_-,p_+;A)}(\ul{\t}_-,\ul{u},\ul{t}_+) = d_{\overline{\cM}(p_-,M)}(\ul{\t}_-) + d_{\cB_{\rm SFT}(A}(\ul{u}) + d_{\overline{\cM}(M,p_+)}(\ul{\t}_+)$. Hence we have $d_{\cB(p_-,p_+;A)}(\ul{\t}_-,\ul{u},\ul{\t}_+) = 1$ if and only if the degeneracy index of exactly one of the three arguments $\ul{\t}_-,\ul{u},\ul{\t}_+$ is $1$ and the other two are $0$. This identifies $|\partial_1\cX_{p_-,p_+,A}|=\partial_1\cB(p_-,p_+; A)$ as in the first line of the displayed equation below.  
Then the subsequent identifications result by comparing the resulting expressions with the interiors in Lemma~\ref{lem:PSS poly}, \ref{lem:iota,h poly}. We obtain an identification that throughout is to be interpreted on the smooth level (as fiber product constructions drop some non-smooth points)
\begin{align*}
&\partial_1\cX_{p_-,p_+, A} \;\cong\;  \phantom{\sqcup\;\;}
\partial_0\bM(p_-,M) \; \leftsub{\ev_0^+}{\times}_{\ev^+} \; \partial_1\cX^{\scriptscriptstyle \rm SFT}_A \; \leftsub{\ev^-}{\times}_{\ev_0^-} \; \partial_0\bM(M,p_+)  \\
&\phantom{\partial_1\cX^{\scriptscriptstyle \rm SFT}_{p_-,p_+, A} \;\cong\; } \sqcup\;\;
\partial_1\bM(p_-,M) \; \leftsub{\ev_0^+}{\times}_{\ev^+} \; \partial_0\cX^{\scriptscriptstyle \rm SFT}_A \; \leftsub{\ev^-}{\times}_{\ev_0^-} \; \partial_0\bM(M,p_+) 
 \\
&\phantom{\partial_1\cX_{p_-,p_+, A} \;\cong\; } \sqcup\;\;
\partial_0\bM(p_-,M) \; \leftsub{\ev_0^+}{\times}_{\ev^+} \; \partial_0\cX^{\scriptscriptstyle \rm SFT}_A \; \leftsub{\ev^-}{\times}_{\ev_0^-} \; \partial_1\bM(M,p_+)  \\
&\quad\;=\;
\cM(p_-,M) \; \leftsub{\ev_0^+}{\times}_{\ev^+} \; \cX^{\scriptscriptstyle \rm GW}_A \; \leftsub{\ev^-}{\times}_{\ev_0^-} \; \cM(M,p_+)  \\
&\quad\qquad\sqcup\;\; \textstyle\bigcup_{\begin{smallmatrix}
\scriptscriptstyle \g \in \cP(H)\\
\scriptscriptstyle A_- + A_+ = A\\
  \end{smallmatrix}} \hspace{-2mm}
\cM(p_-,M) \; \leftsub{\ev_0^+}{\times}_{\ev^+} \;  \partial_0\cX^+_{\g,A_+} \times \partial_0\cX^-_{\g,A_-} \; \leftsub{\ev^-}{\times}_{\ev_0^-} \; \cM(M,p_+)  \\
&\quad\qquad \sqcup\;\;
\textstyle\bigcup_{q\in{\rm Crit}(f)} \cM(p_-,q) \times \cM(q,M) \; \leftsub{\ev_0^+}{\times}_{\ev^+} \; \partial_0\cX^{\scriptscriptstyle \rm SFT}_A \; \leftsub{\ev^-}{\times}_{\ev_0^-} \; \partial_0\bM(M,p_+) 
 \\
&\quad\qquad\sqcup\;\; \textstyle\bigcup_{q\in{\rm Crit}(f)} 
\cM(p_-,M) \; \leftsub{\ev_0^+}{\times}_{\ev^+} \; \partial_0\cX^{\scriptscriptstyle \rm SFT}_A \; \leftsub{\ev^-}{\times}_{\ev_0^-} \; \cM(M,q) \times\cM(q,p_+)  \\
&\quad\;=\; 
 \partial_0\cX^\iota_{p_-,p_+,A}
\;\;\sqcup\;\; \textstyle\bigcup_{\begin{smallmatrix}
\scriptscriptstyle \g \in \cP(H)\\
\scriptscriptstyle A_- + A_+ = A\\
  \end{smallmatrix}}
\partial_0\cX^+_{p_-,\g,A_+} \times \partial_0\cX^-_{\g,p_+, A_-} \\
&\quad\qquad\sqcup\;\;
\textstyle\bigcup_{q\in{\rm Crit}(f)} \cM(p_-,q) \times \partial_0\cX_{q,p_+, A} 
\quad
\sqcup\;\;  \textstyle\bigcup_{q\in{\rm Crit}(f)} \partial_0\cX_{p_-,q, A} \times \cM(q,p_+). 
\end{align*}
Here we also used the identification of evaluation maps in Assumption~\ref{ass:iso}~(iii)(a).
Then compatibility in (ii) of the oriented section functors $S_\alpha^{\cdots}$ with the identification of these (smooth levels of) faces $\cF^\infty\subset\partial_1\cX^\infty_{p_-,p_+,A}$ follows from compatibility of $\pr_{p_-,p_+,A}: \cX_{p_-,p_+,A}\to \cX^{\scriptscriptstyle \rm SFT}_A$ with the projections $\pr^\pm_\alpha: \cX^\pm_\alpha \to \cX^\pm_{\gamma,A_\pm}$ for $\alpha\in\cI^\pm$ used in Lemma~\ref{lem:PSS poly} and $\pr^\iota_\alpha: \cX^\iota_\alpha \to \cX^{\scriptscriptstyle \rm GW}_A$ used in Lemma~\ref{lem:iota,h poly}. More precisely, $S_{p_-,p_+,A}|_{\cF^\infty}=\pr_\cF^*S_\cF$ follows from compatibility of the sections in Assumption~\ref{ass:iso}~(iii) and
$$
\pr_{p_-,p_+,A}|_{\cF^\infty} \;=\; 
\begin{cases}
\pr^\iota_{p_-,p_+, A} \circ \pr_\cF
&\qquad\text{for}\;\; \cF=\partial_0\cX^\iota_{p_-,p_+, A}  , \\
(\pr^+_{p_-,\g, A_-}\times \pr^-_{\g,p_+, A_+}) \circ \pr_\cF  
&\qquad\text{for}\;\; \cF=\partial_0\cX^+_{p_-,\g, A_+} \times  \partial_0\cX^-_{\g,p_+, A_-} , \\
\pr_{q,p_+, A} \circ \pr_\cF
&\qquad\text{for}\;\; \cF=\cM(p_-,q) \times \partial_0\cX_{q,p_+, A} , \\
\pr_{p_-,q, A} \circ \pr_\cF
&\qquad\text{for}\;\; \cF=\partial_0\cX_{p_-,q, A} \times \cM(q,p_+) .
\end{cases}
$$

\medskip
\noindent
{\bf Construction of coherent perturbations:}
Next, we construct admissible sc$^+$-multisections $\kappa_{\alpha}^{\cdots}:\cW_{\alpha}^{\cdots}\to\Q^+$ for $\alpha\in\cI^+\cup\cI^-\cup\cI^\iota\cup\cI$ as claimed in (i), i.e.\ in general position to the respective sections $S_{\alpha}^{\cdots}:\cX_{\alpha}^{\cdots}\to:\cW_{\alpha}^{\cdots}$, while also coherent as claimed in (iii).
The existence of such coherent transverse perturbations will ultimately be guaranteed by an abstract perturbation theorem for coherent systems of sc-Fredholm sections. Since the SFT perturbation package \cite[\S 14]{fh-sft} has not yet been described for neck stretching, 
we give a detailed construction of the perturbations for our purposes. 
We proceed as in Lemma~\ref{lem:iota chain} and construct them all as pullbacks  
$\kappa^{\cdots}_\alpha:= \lambda^{\cdots}_A \circ (\pr^{\cdots}_\alpha)^*$ of a collection of sc$^+$-multisections on the SFT resp.\ Gromov-Witten polyfold bundles -- without Morse trajectories --
$$
\ul\lambda \;=\; \left(  \;\;
\begin{aligned}
&\bigl( \lambda^+_{\g,A} : \Ti\cW^+_{\g,A}\to\Q^+ \bigr)_{\g\in\cP(H),A\in H_2(M)} \qquad
&&\bigl( \lambda^{\scriptscriptstyle \rm GW}_A : \Ti\cW^{\scriptscriptstyle \rm GW}_A\to\Q^+ \bigr)_{A\in H_2(M)}\\
&\bigl( \lambda^-_{\g,A} : \Ti\cW^-_{\g,A}\to\Q^+ \bigr)_{\g\in\cP(H),A\in H_2(M)} \qquad
&&\bigl( \lambda^{\scriptscriptstyle \rm SFT}_A : \Ti\cW^{\scriptscriptstyle \rm SFT}_A\to\Q^+ \bigr)_{A\in H_2(M)}
\end{aligned}
\;\;
\right).
$$
For this to induce a coherent collection of sc$^+$-multisections as required in (iii),
\begin{align*}
&\bigl( \kappa^+_{p,\g,A}:=\lambda^+_{\g,A} \circ (\pr^+_{p,\g,A})^*  \bigr)_{(p,\g,A)\in\cI^+}, \qquad
&&\bigl( \kappa^\iota_{p_-,p_+,A}:=\lambda^{\scriptscriptstyle\rm GW}_{A} \circ (\pr^\iota_{p_-,p_+,A})^*  \bigr)_{(p_-,p_+,A)\in\cI^\iota} ,\\
&\bigl( \kappa^-_{\g,p,A}:=\lambda^-_{\g,A} \circ (\pr^-_{\g,p,A})^*  \bigr)_{(\g,p,A)\in\cI^-}, \qquad
&& \bigl( \kappa_{p_-,p_+,A}:=\lambda^{\scriptscriptstyle\rm SFT}_{A} \circ (\pr_{p_-,p_+,A})^*  \bigr)_{(p_-,p_+,A)\in\cI},
\end{align*}
it suffices to pick $\ul\lambda$ compatible with respect to the faces of the SFT neck stretching polyfolds $\cX^{\scriptscriptstyle\rm SFT}_{A}$ in \eqref{eq:SFTfaces}. More precisely, using the natural identifications of bundles from Assumption~\ref{ass:iso}~(iii), we will construct $\ul\l$ coherent in the sense that -- for some choice of $r>0$ in the construction of 
$|\cX^\pm_{\g,A}| = (\ov\ev^\pm)^{-1}(\D_r^\pm\times M)\subset \cB^\pm_{\rm SFT}(\g;A)$ and $\cW^\pm_{\g,A}=\Ti\cW^\pm_{\g,A}|_{\cX^\pm_{\g,A}}$ -- we have
\begin{align} \label{eq:coherent1}
 \lambda^{\scriptscriptstyle \rm SFT}_A(w) &\;=\; \lambda^{\scriptscriptstyle \rm GW}_A(w) 
\quad&& \forall \;
w\in \Ti\cW^{\scriptscriptstyle \rm GW}_A , \\
\lambda^{\scriptscriptstyle \rm SFT}_A ( (l_{\g,A_\pm})_*(w^+,w^-)) &\;=\;  \lambda^+_{\g,A_+}(w^+) \cdot \lambda^-_{\g,A_-}(w^-)
\quad&& \forall \;
(w^+,w^-) \in \cW^+_{\g, A_+} \times \cW^-_{\g, A_-},
\label{eq:coherent2}
\end{align}
where $l_{\g,A_\pm}$ is the map defined in Assumption~\ref{ass:iso}(i).
So to finish this proof it remains to choose the sc$^+$-multisections $\ul\lambda$ so that each induced sc$^+$-multisection in the induced coherent collection for $(\kappa^{\cdots}_\alpha)_{\alpha\in\cI^+\cup\cI^-\cup\cI^\iota\cup\cI}$ is admissible and in general position, while also satisfying the coherence requirements \eqref{eq:coherent1}, \eqref{eq:coherent2} and the 
requirements on $\ul\kappa^\iota$ in the proofs of Lemma~\ref{lem:iota chain} and \ref{lem:iota triangular}.
The construction of coherent perturbations for the SFT polyfolds analogously to \cite[\S 14]{fh-sft} proceeds by first choosing coherent compactness controlling data, i.e.\ pairs $(N,\cU)$ 
of auxiliary norms on all the bundles and saturated neighbourhoods of the compact zero sets in all the ep-groupoids $\Ti\cX^\pm_{\g,A} , \Ti\cX^{\scriptscriptstyle\rm GW}_A, \Ti\cX^{\scriptscriptstyle\rm SFT}_A$  (c.f.\ Definition~\ref{def:control}), which are compatible with the immersions to boundary faces in \eqref{eq:SFTfaces}. 
Then it constructs the perturbations $\l^{\scriptscriptstyle\rm GW}_A$ as in Lemma~\ref{lem:iota triangular} and also $\l^\pm_{\g,A_\pm}$ to be in general position, admissible w.r.t.\ the coherent data $(2 N, \cU)$, and coherent in the sense that continuous extension of \eqref{eq:coherent1}--\eqref{eq:coherent2} induces a well defined multisection ${\lambda^\partial_A: \Ti\cW^{\scriptscriptstyle\rm SFT}_A|_{\partial\cX^{\scriptscriptstyle\rm SFT}_A}\to \Q^+}$. 
Here coherence of the perturbations on the intersection of faces (see Remark~\ref{rmk:face}) is required to guarantee existence of scale-smooth extensions of $\lambda^\partial_A$ to multisections $\lambda^{\scriptscriptstyle\rm SFT}_A: \Ti\cW^{\scriptscriptstyle\rm SFT}_A\to \Q^+$. 
Coherence of the compactness controlling pairs guarantees that the multisection $\lambda^\partial_A$ over $\partial\cX^{\scriptscriptstyle\rm SFT}_A\subset \Ti\cX^{\scriptscriptstyle\rm SFT}_A$ satisfies the auxiliary norm bounds  $N(\lambda^\partial_A)\leq\frac 12$ 
and support requirements that guarantee compactness for extensions $\lambda^{\scriptscriptstyle\rm SFT}_A$ of $\lambda^\partial_A$ with $N(\lambda^{\scriptscriptstyle\rm SFT}_A)\leq 1$ and appropriate support requirements. 
Moreover, we may choose each of the extensions $\lambda^{\scriptscriptstyle\rm SFT}_A$ using Theorem~\ref{thm:transversality} to ensure -- as in Lemma~\ref{lem:iota chain} -- that the induced multisections $\kappa^{\cdots}_\alpha$ are in general position as well. The latter will automatically be admissible with respect to pullback of the pair controlling compactness.  
In more detail (but without specifying the auxiliary norm bounds) the inductive construction of perturbations in \cite{fh-sft-full} -- simplified to the subset of SFT moduli spaces considered here -- proceeds as follows:

\medskip\noindent{\bf Construction of $\boldsymbol{\lambda^{\scriptscriptstyle\rm GW}_A}$ and $\boldsymbol{\underline\kappa^\iota}$:}
Since the Gromov-Witten ep-groupoids $\Ti\cX_A^{\scriptscriptstyle\rm GW}$ are boundaryless by Assumption~\ref{ass:iota,h}~(iii), the sc$^+$-multisections $\l^{\scriptscriptstyle\rm GW}_A$ can be chosen independently of all other multisections. 
So we construct $\l^{\scriptscriptstyle\rm GW}_A$ as in the proofs of Lemma~\ref{lem:iota chain} and \ref{lem:iota triangular}, to ensure that the conclusions in these lemmas hold, as required by (i). 
This prescribes \eqref{eq:coherent1} on the boundary face $\Ti\cX^{\scriptscriptstyle\rm GW}_A\subset \partial \Ti\cX^{\scriptscriptstyle\rm SFT}_A$. 

Moreover, recall that $\l^{\scriptscriptstyle\rm GW}_A$ is obtained by applying Theorem~\ref{thm:transversality} to the sc-Fredholm section functors $S_A^{\scriptscriptstyle\rm GW}$, 
the sc$^\infty$ submersion $\ov\ev^+ \times \ov\ev^- : \Ti\cX^{\scriptscriptstyle\rm GW}_A \rightarrow \CP^1\times M \times \CP^1\times M$, and the collection of Cartesian products of stable and unstable manifolds $\{z_0^+\}\times W_{p_-}^- \times \{z_0^-\}\times W_{p_+}^+$.
As in the proof of Lemma~\ref{lem:iota chain} this ensures that the pullbacks $\ul\k^\i =(\k_{\alpha}^{\i}=\lambda^{\scriptscriptstyle\rm GW}_{A} \circ (\pr^\iota_\alpha)^*)_{\alpha \in \cI^{\i}}$ are in general position. Moreover, these pullbacks are admissible w.r.t.\ the pairs controlling compactness on $\cW^{\i}_{\alpha} \rightarrow \cX^{\i}_{\alpha}$ that result by pullback from the coherent compactness controlling pair on $\Ti\cW^{\scriptscriptstyle\rm GW}_A \rightarrow\Ti\cX^{\scriptscriptstyle\rm GW}_A$, which is constructed in a preliminary step as in \cite[\S 13]{fh-sft}.

\medskip\noindent{\bf Coherence for $\boldsymbol{\lambda^\pm_{\gamma,A}}$:}
The next step is to construct sc$^+$-multisections $\lambda^\pm_{\gamma,A}:\Ti\cW^\pm_{\g, A}\to\Q^+$ over the SFT ep-groupoids $\Ti\cX^\pm_{\g,A}$ of planes with limit orbit $\g\in\cP(H)$ from Assumption~\ref{ass:pss}, which then induce the perturbations $\underline\kappa^\pm$ for the $PSS/SSP$ moduli spaces. These constructions are independent of the choice of $\l^{\scriptscriptstyle\rm GW}_A$ since the corresponding boundary faces of $\Ti\cX^{\scriptscriptstyle\rm SFT}_A$ do not intersect by Assumption~\ref{ass:iso}~(ii). However, to enable the subsequent construction of $\lambda^{\scriptscriptstyle\rm SFT}_A$ as extension of the boundary values prescribed in \eqref{eq:coherent1} and \eqref{eq:coherent2}, we need to make sure that each sc$^+$-multisection $(\lambda^+_{\gamma,A_+}\cdot \lambda^-_{\gamma,A_-})\circ (l_{\g,A_\pm})^{-1}_*$ 
is well defined on the (open subset of) face $\cF_{\g,A_\pm}:= l_{\g,A_\pm}(\cX^+_{\g,A_+}\times \cX^-_{\g,A_-})\subset\partial\Ti\cX^{\scriptscriptstyle\rm SFT}_A$ and
coincides with the other sc$^+$-multisections $(\lambda^+_{\gamma',A'_+}\cdot \lambda^-_{\gamma',A'_-})\circ (l_{\g',A'_\pm})^{-1}_*$ on their intersection 
$\cF_{\g,A_\pm}\cap \cF_{\g',A'_\pm}$. 
Then this yields a well defined sc$^+$-multisection on $\bigcup \cF_{\g,A_\pm} = \partial \cX^{\scriptscriptstyle\rm SFT}_A\subset\partial\Ti\cX^{\scriptscriptstyle\rm SFT}_A$. 
To describe these intersections we note that \cite{fh-sft} constructs the ep-groupoids $\Ti\cX^\pm_{\g,A_\pm}$ with coherent boundaries -- involving ep-groupoids 
$(\cX^{\scriptscriptstyle\rm Fl}_{\g^-,\g^+,B})_{\g^\pm\in\cP(H),B\in H_2(M)}$ which contain the moduli spaces of Floer trajectories between periodic orbits $\g^\pm$, as well as further ep-groupoids for Floer trajectories carrying a marked point. 
We will avoid dealing with the latter by specifying values $r<\infty$ when pulling back perturbations from the ep-groupoids $\cX^\pm_{\g,A}\subset\Ti\cX^\pm_{\g,A}$ given by $|\cX^\pm_{\g,A}| = (\ov\ev^\pm)^{-1}(\D_r^\pm\times M)\subset \cB^\pm_{\rm SFT}(\g;A)$, as this will prevent the appearance of marked Floer trajectories even in the closure. 
For any fixed value $0<r\leq \infty$, the $j$-th boundary stratum is given by $j$ Floer trajectories breaking off, 
\begin{align} \label{eq:Xpmbdy}
\partial_j \cX^+_{\g,A} &\;= \textstyle\bigsqcup_{\begin{smallmatrix}
\scriptscriptstyle \g^0,\ldots,\g^j=\g \in \cP(H)\\
\scriptscriptstyle A_+ + B_1+\ldots + B_j = A
  \end{smallmatrix}} 
\partial_0\cX^+_{\g^0,A_+} \times \partial_0\cX^{\scriptscriptstyle\rm Fl}_{\g^0,\g^1,B_1} \times \ldots \times \partial_0\cX^{\scriptscriptstyle\rm Fl}_{\g^{j-1},\g^j,B_j} , \\
\partial_{k-j} \cX^-_{\g,A} &\;= \textstyle\bigsqcup_{\begin{smallmatrix}
\scriptscriptstyle \g=\g^j,\ldots,\g^k \in \cP(H)\\
\nonumber
\scriptscriptstyle B_{j+1}+\ldots+B_k + A_- = A
\end{smallmatrix}}
\partial_0\cX^{\scriptscriptstyle\rm Fl}_{\g^j,\g^{j+1},B_{j+1}} \times \ldots \times \partial_0\cX^{\scriptscriptstyle\rm Fl}_{\g^{k-1},\g^k,B_k} \times 
\partial_0\cX^-_{\g^k,A_-} .
\end{align}
Now, for example, 
$\partial_0 \cX^+_{\g^0,A_+}\times \partial_0\cX^{\scriptscriptstyle\rm Fl}_{\g^0,\g^1,B}
\times\partial_0 \cX^-_{\g^1,A_-}$ is both a subset of 
$\partial_0 \cX^+_{\g^0,A_+} \times\partial_1 \cX^-_{\g^0,A_-+B} \subset \partial \bigl(\cX^+_{\g^0,A_+} \times \cX^-_{\g^0,A_-+B} \bigr)$
and of $\partial_1 \cX^+_{\g^1,A_++B} \times\partial_0 \cX^-_{\g^1,A_-}\subset \partial \bigl(\cX^+_{\g^1,A_++B} \times \cX^-_{\g^1,A_-} \bigr)$, and the embeddings $l_{\g^0,A^0_\pm}$ and $\l_{\g^1,A^1_\pm}$ for the two splittings 
$A_+ + (A_-+B) = A^0_++A^0_- =A=A^1_++A^1_- =(A_++B)+A_-$ coincide under this identification. 
Generally, the boundary of the Floer ep-groupoids is given by broken trajectories, and this yields a disjoint cover of $\partial^{R=\infty}\cX^{\scriptscriptstyle\rm SFT}_A \subset \partial \cX^{\scriptscriptstyle\rm SFT}_A$,
\begin{align*}
\partial^{R=\infty}\cX^{\scriptscriptstyle\rm SFT}_A  
\;=\; \hspace{-2mm}\bigsqcup_{\begin{smallmatrix}
\scriptscriptstyle \g^0,\ldots,\g^k \in \cP(H)\\
\scriptscriptstyle A_+ + B_1 + \ldots + B_k + A_- = A
  \end{smallmatrix}} \hspace{-2mm}
l_{\ul\g, A_\pm, \ul B}\bigl( 
\partial_0 \cX^+_{\g^0,A_+}\times \partial_0\cX^{\scriptscriptstyle\rm Fl}_{\g^0,\g^1,B_1}
\times \ldots \times \partial_0\cX^{\scriptscriptstyle\rm Fl}_{\g^{k-1},\g^k,B_k}
\times
\partial_0 \cX^-_{\g^k,A_-} \bigr) ,
\end{align*}
in which the embeddings $l_{\ul\g, A_\pm, \ul B}$ coincide with each of the embeddings
$l_{\g^j,A^j_\pm}$ for $0\leq j\leq k$ and $A^j_+=A_++\sum_{i\leq j} B_i$, $A^j_-=A_- +\sum_{i>j} B_i$ -- when restricted to the subsets
$$
\partial_0 \cX^+_{\g^0,A_+}\times
\partial_0\cX^{\scriptscriptstyle\rm Fl}_{\g^0,\g^1,B_1}
\times \ldots \times \partial_0\cX^{\scriptscriptstyle\rm Fl}_{\g^{k-1},\g^k,B_k}
\times
\partial_0 \cX^-_{\g^k,A_-}
\;\subset\; 
\partial_j \cX^+_{\g^j,A^j_+}\times \partial_{k-j} \cX^-_{\g^j,A^j_-} . 
$$
Now on these subsets we require coherence
$\lambda^+_{\g^j,A^j_+} \cdot \lambda^-_{\g^j,A^j_-} = \lambda^+_{\g^{j'},A^{j'}_+} \cdot \lambda^-_{\g^{j'},A^{j'}_-}$ for all $0\leq j\neq j' \leq k$, as this is equivalent to \eqref{eq:coherent2} being well defined on $\im l_{\ul\g, A_\pm, \ul B} = \bigcap_{j=0}^k \cF_{\g^j,A^j_\pm}$. 
This will be achieved by constructing the sc$^+$-multisections $(\lambda^\pm_{\g,A_\pm})$ to have product structure on the boundary -- 
where the bundles $P_{\g,A}:\cW^\pm_{\g,A}\to \cX^\pm_{\g,A}$ are restricted to various faces of $\partial\cX^\pm_{\g,A}$ -- 
\begin{align} \label{eq:lambdapm}
\lambda^+_{\g^j,A^j_+}\big|_{P_{\g^j,A^j_+}^{-1}\bigl(
\cX^+_{\g^0,A_+}\times
\cX^{\scriptscriptstyle\rm Fl}_{\g^0,\g^1,B_1}
\times \ldots \times \cX^{\scriptscriptstyle\rm Fl}_{\g^{j-1},\g^j,B_j}
\bigr)} 
&=\; 
\lambda^+_{\g^0,A_+}\cdot
\lambda^{\scriptscriptstyle\rm Fl}_{\g^0,\g^1,B_1}
\cdot \ldots \cdot \lambda^{\scriptscriptstyle\rm Fl}_{\g^{j-1},\g^j,B_j} , \\ 
\nonumber
\lambda^-_{\g^j,A^j_-}\big|_{
P_{\g^j,A^j_-}^{-1}\bigl(
\cX^{\scriptscriptstyle\rm Fl}_{\g^j,\g^{j+1},B_{j+1}}
 \times \ldots \times \cX^{\scriptscriptstyle\rm Fl}_{\g^{k-1},\g^k,B_k}
\times
\cX^-_{\g^k,A_-}
\bigr)}
&=\; 
\lambda^{\scriptscriptstyle\rm Fl}_{\g^j,\g^{j+1},B_{j+1}} \cdot
 \ldots \cdot\lambda^{\scriptscriptstyle\rm Fl}_{\g^{k-1},\g^k,B_k}
\cdot
\lambda^-_{\g^k,A_-}, 
\end{align}
for a collection of sc$^+$-multisections $\lambda^{\scriptscriptstyle\rm Fl}_{\g^-,\g^+,B}: \cW^{\scriptscriptstyle \rm Fl}_{\g^-,\g^+,B} \to \Q$ over the Floer ep-groupoids $\cX^{\scriptscriptstyle\rm Fl}_{\g^-,\g^+,B}$. 
While this guarantees coherence on each overlap of embeddings $\im l_{\ul\g, A_\pm, \ul B} \subset \cF_{\g^j,A^j_\pm} \cap\cF_{\g^{j'},A^{j'}_\pm}$, 
$$
\lambda^+_{\g^j,A^j_+} \cdot \lambda^-_{\g^j,A^j_-} 
\;=\;
\lambda^+_{\g^0,A_+}\cdot
\lambda^{\scriptscriptstyle\rm Fl}_{\g^0,\g^1,B_1}
\cdot \ldots \cdot \lambda^{\scriptscriptstyle\rm Fl}_{\g^{k-1},\g^k,B_j}
\cdot
\lambda^-_{\g^k,A_-}
\;=\; \lambda^+_{\g^{j'},A^{j'}_+} \cdot \lambda^-_{\g^{j'},A^{j'}_-} , 
$$
we are now faced with the challenge of satisfying the coherence conditions in \eqref{eq:lambdapm}. 
These conditions uniquely determine the boundary restrictions 
$\lambda^\pm_{\g,A_\pm}\big|_{P_{\g,A_\pm}^{-1}(\partial\cX^\pm_{\g,A_\pm})}$
via the identification of the boundaries with Cartesian products of interiors in \eqref{eq:Xpmbdy}. 
Thus \eqref{eq:Xpmbdy} on Cartesian products involving boundary strata poses coherence conditions on the choice of $\lambda^{\scriptscriptstyle\rm Fl}_\beta$ for $\beta\in\cI^{\scriptscriptstyle\rm Fl}:= \cP(H)\times\cP(H)\times H_2(M)$.

\medskip\noindent{\bf Construction of $\boldsymbol{\lambda^{\scriptscriptstyle\rm Fl}_{\gamma^-,\gamma^+,B}}$:}
To achieve the coherence in \eqref{eq:lambdapm}, \cite{fh-sft-full} first constructs the sc$^+$-multisections $(\lambda^{\scriptscriptstyle\rm Fl}_\beta)_{\beta\in\cI^{\scriptscriptstyle\rm Fl}}$ by iteration over the maximal degeneracy $k_\beta:=\max\{k\in\N_0\,|\, (S^{\scriptscriptstyle\rm Fl}_\beta)^{-1}(0) \cap \partial_k \cX^{\scriptscriptstyle\rm Fl}_\beta \neq \emptyset\}$ 
of unperturbed solutions (which is finite by Gromov compactness):
We first consider classes $\beta$ with $k_\beta=-\infty$. For these, 
the section $S^{\scriptscriptstyle\rm Fl}_\beta$ has no zeros so is already transverse, so that $\lambda^{\scriptscriptstyle\rm Fl}_\beta$ can be chosen as the trivial perturbation. (The trivial multivalued section functor $\l:\cW\to\Q^+$ is given by $\l(0)=1$ and ${\l(w\neq 0)=0}$.) 
Next, we consider $\beta$ with $k_\beta=0$. For these, 
the section $S^{\scriptscriptstyle\rm Fl}_\beta$ has all zeros in the interior, so that $\lambda^{\scriptscriptstyle\rm Fl}_\beta$ can be chosen admissible and trivial on the boundary -- by applying Corollary~\ref{cor:regularize}~(i) with a neighbourhood of the unperturbed zero set in the interior, $|(S^{\scriptscriptstyle\rm Fl}_\beta)^{-1}(0)|\subset \cV_\beta \subset |\partial_0\cX^{\scriptscriptstyle\rm Fl}_\beta|$. 
Once the iteration has constructed $\lambda^{\scriptscriptstyle\rm Fl}_\beta$ for all $\beta$ with $k_\beta \leq n$ for some $n\in\N_0$, we proceed to consider 
$\beta=(\g^-,\g^+,B)\in\cI^{\scriptscriptstyle\rm Fl}$ with
$k_\beta= n+1$.
For these, 
the restriction $\lambda^{\scriptscriptstyle\rm Fl}_\beta|_{P_\beta^{-1}(\partial\cX^{\scriptscriptstyle\rm Fl}_\beta)}$ to the boundary $\partial\cX^{\scriptscriptstyle\rm Fl}_\beta = \bigcup_{\g^-=\g^0,\g^1,\ldots,\g^{k-1},\g^k=\g^+,B=B_1+\ldots+B_k} \partial_0\cX^{\scriptscriptstyle\rm Fl}_{\g^0,\g^1,B_1}
\times \ldots \times \partial_0\cX^{\scriptscriptstyle\rm Fl}_{\g^{k-1},\g^k,B_k}$
is prescribed by the previous iteration steps
$\lambda^{\scriptscriptstyle\rm Fl}_\beta\big|_{P_\beta^{-1}(\cX^{\scriptscriptstyle\rm Fl}_{\g^0,\g^1,B_1}
\ldots \times \cX^{\scriptscriptstyle\rm Fl}_{\g^{k-1},\g^k,B_k})} :=
\lambda^{\scriptscriptstyle\rm Fl}_{\g^0,\g^1,B_1}
 \ldots \cdot \lambda^{\scriptscriptstyle\rm Fl}_{\g^{k-1},\g^k,B_k}$
on all boundary faces that contain unperturbed solutions in their closure. Indeed, existence of a solution in $\cX^{\scriptscriptstyle\rm Fl}_{\g^0,\g^1,B_1}\times \ldots \times\cX^{\scriptscriptstyle\rm Fl}_{\g^{k-1},\g^k,B_k}$ implies $k_{\g^{i-1},\g^i,B_i}\geq 0$ for $i=1,\ldots,k$, and the Cartesian product of solutions of maximal degeneracy yields $1 + k_{\g^0,\g^1,B_1} + \ldots + k_{\g^{k-1},\g^k,B_k} \leq k_\beta$. Thus these prescriptions are made for $0\leq k_{\g^{i-1},\g^i,B_i} \leq k_\beta - 1
=n$, and 
on boundary faces with no solutions in their closure we prescribe the trivial perturbation throughout.

This yields a well defined sc$^+$-multisection $\lambda^{\scriptscriptstyle\rm Fl}_\beta|_{P_\beta^{-1}(\partial\cX^{\scriptscriptstyle\rm Fl}_\beta)}$ by coherence in the prior iteration steps, so that $\lambda^{\scriptscriptstyle\rm Fl}_\beta$ can be constructed by applying the extension result \cite[Thm.15.5]{HWZbook} which provides general position and admissibility with respect to a pair controlling compactness that extends the pair which was chosen on the boundary in prior iteration steps.
    
\medskip\noindent{\bf Construction of $\boldsymbol{\lambda^\pm_{\gamma,A}}$ and $\boldsymbol{\underline\kappa^\pm}$:}
With the Floer perturbations in place, \cite{fh-sft-full} next constructs the collections of sc$^+$-multisections $(\lambda^\pm_{\gamma,A})_{\g\in\cP(H),A\in H_2(M)}$ to satisfy \eqref{eq:lambdapm} by iteration over degeneracy $k_{\gamma,A}:=\max\{k\in\N_0\,|\, (S^\pm_{\gamma,A})^{-1}(0) \cap \partial_k \Ti\cX^\pm_{\gamma,A} \neq \emptyset\}$. For $k_{\gamma,A}=-\infty$ one takes $\lambda^\pm_{\gamma,A}$ to be trivial. For $k_{\gamma,A}=0$ one applies Theorem~\ref{thm:transversality} to the sc-Fredholm section functor $S_{\g,A}^\pm:\Ti\cX_{\g,A}^\pm\to\Ti\cW_{\g,A}^\pm$, the map $\ov\ev^\pm : \Ti\cX^\pm_{\g,A} \to \ov{\C^\pm}\times M$, and the collection of stable resp.\ unstable manifolds $\{0\}\times W_p^\pm$ for all critical points $p \in \Crit(f)$. These satisfy the assumptions as the zero set $|(S_{\g,A}^\pm)^{-1}(0)|$ is compact and the preimages $(\ov\ev^\pm)^{-1}(\{0\}\times W_p^\pm)$ lie within the open subset $\cX^\pm_{\g,A}\subset\Ti\cX^\pm_{\g,A}$ on which $\ov\ev^\pm$ restricts to a sc$^\infty$ submersion $\ev^\pm : \cX^\pm_{\g,A} \to \C^\pm\times M$. 
We can moreover prescribe $\lambda^\pm_{\gamma,A}|_{P_{\g,A}^{-1}(\partial\Ti\cX^\pm_{\g,A})}$ to be trivial, since in the absence of solutions the trivial perturbation is in general position. 
Then Theorem~\ref{thm:transversality} provides $\lambda^\pm_{\gamma,A}$ that is supported in the interior and transverse to each submanifold $\{0\}\times W_p^\pm$ in the sense that these submanifolds are transverse to the evaluation from the perturbed zero set 
\begin{equation}\label{eq:transev}
\ev^\pm \,: \;  \bigl| \{x\in \cX^\pm_{\gamma,A} \,|\, \l^\pm_{\gamma,A}(S^\pm_{\gamma,A}(x))>0\} \bigr| \;\to\; \C^\pm\times M . 
\end{equation}
Now suppose that admissible $\lambda^\pm_{\gamma',A'}$ in general position have been constructed for $k_{\gamma',A'}\leq k\in\N_0$, and satisfy both the transversality in \eqref{eq:transev} and the coherence condition \eqref{eq:lambdapm} over the ep-groupoids $|\cX^\pm_{\g',A'}|=(\ov\ev^\pm)^{-1}(\D^\pm_{r_k}\times M)$ with $r_k:= 2+ 2^{-k}$. 
Then for $k_{\gamma,A}=k+1$ we will construct $\lambda^\pm_{\gamma,A}$ to satisfy \eqref{eq:lambdapm} over $(\ov\ev^\pm)^{-1}(\D^\pm_{r_{k+1}}\times M)$ by first noting that the previous iteration -- and requiring triviality on boundary faces without solutions -- determines a well defined sc$^+$-multisection $\lambda^\pm_{\gamma,A}|_{P_{\g,A}^{-1}(\partial\cX^\pm_{\gamma,A})}$ over the $r=r_k$ boundary $\partial\cX^\pm_{\gamma,A} \simeq \bigcup_{\g',A=A_\pm+B} \partial_0\cX^\pm_{\g',A_\pm} \times \cX^{\scriptscriptstyle\rm Fl}_{\g',\g,B}$. 
For faces (w.r.t.\ $\partial\cX^\pm_{\g,A}$) with solutions it is given by $\lambda^\pm_{\gamma,A}\big|_{P_{\g,A}^{-1}(\cX^\pm_{\g',A_\pm} \times \cX^{\scriptscriptstyle\rm Fl}_{\g',\g,B})} =
\lambda^\pm_{\g',A_\pm} \times \lambda^{\scriptscriptstyle\rm Fl}_{\g',\g,B}$ where
$k_{\gamma,A}\geq 1+ k_{\g',A_\pm} + k_{\g',\g,B}$. 
This is well defined at $(x^\pm,\ul x,\ul x')\in \partial_0\cX^\pm_{\g',A_\pm} \times \cX^{\scriptscriptstyle\rm Fl}_{\g',\g'',B'}\times \cX^{\scriptscriptstyle\rm Fl}_{\g'',\g,B-B'}$, which appears both as 
$(x^\pm, (\ul x,\ul x'))\in \partial_0\cX^\pm_{\g',A_\pm} \times \partial\cX^{\scriptscriptstyle\rm Fl}_{\g',\g,B}$ and
$((x^\pm,\ul x),\ul x')\in \partial\cX^\pm_{\g'',A_\pm+B'} \times \cX^{\scriptscriptstyle\rm Fl}_{\g'',\g,B-B'}$, 
by the coherence of the Floer multisections and the prior iteration: 
For vectors in the respective fibers 
$(w^\pm, w, w') \in P_{\g',A_\pm}^{-1}(x^\pm) \times P_{\g',\g'',B'}^{-1}(\ul x) \times P_{\g'',\g,B-B'}^{-1}(\ul x')$ we have 
\begin{align*}
\lambda^\pm_{\g',A_\pm}(w^\pm) \cdot \lambda^{\scriptscriptstyle\rm Fl}_{\g',\g,B}(w,w')
& \;=\; \lambda^\pm_{\g',A_\pm}(w^\pm) \cdot \lambda^{\scriptscriptstyle\rm Fl}_{\g',\g'',B'}(w) \cdot
\lambda^{\scriptscriptstyle\rm Fl}_{\g'',\g,B-B'}(w') \\
&\;=\; \lambda^\pm_{\g'',A_\pm+B'}(w^\pm,w) \cdot
\lambda^{\scriptscriptstyle\rm Fl}_{\g'',\g,B-B'}(w'). 
\end{align*}
Moreover, $\ev^\pm: |\{x\in \partial\cX^\pm_{\gamma,A} \,|\, \l^\pm_{\gamma,A}(S^\pm_{\gamma,A}(x))>0\}|\to \C^\pm\times M$ is transverse to the submanifolds $\{0\}\times W_p^\pm$. 
However, this defines an admissible sc$^+$-multisection in general position only over the open subset of the boundary $ \partial\cX^\pm_{\gamma,A} = (\ev^\pm)^{-1}(\D^\pm_{r_k}\times M)\cap \partial\Ti\cX^\pm_{\gamma,A}$.
We multiply the given data by a scale-smooth cutoff function -- guaranteed by the existence of partitions of unity for the open cover $|\Ti\cX^\pm_{\g,A}|=(\ov\ev^\pm)^{-1}(\D^\pm_{r_k}\times M) \cup (\ov\ev^\pm)^{-1}((\ov{\C^\pm}\less\D^\pm_{r_{k+\frac 12}})\times M)$; see Remark~\ref{rmk:partitions} --
to obtain an admissible sc$^+$-multisection $\l^\partial_{\g,A}: \Ti\cW^\pm_{\gamma,A}|_{\partial\Ti\cX^\pm_{\g,A}}\to\Q^+$ which coincides with the prescribed data -- thus in general position and with evaluation transverse to each $\{0\}\times W_p^\pm$ -- over the closed subset $(\ev^\pm)^{-1}(\ov{\D^\pm_{r_{k+1}}}\times M)\cap \partial \Ti\cX^\pm_{\gamma,A}$. 
Then $\lambda^\pm_{\gamma,A}:\Ti\cW^\pm_{\gamma,A}\to \Q^+$ is constructed with these given boundary values using Theorem~\ref{thm:transversality} to achieve not just general position but also transversality as in \eqref{eq:transev}.  
By admissibility of the prior iteration and coherence of the pairs controlling compactness, $\l^\pm_{\gamma,A}$ can moreover be chosen admissible. 

As required in the coherence discussion, this determines right hand sides of \eqref{eq:coherent2} which agree on overlaps of different immersions $l_{\g,A_\pm}(\cX^+_{\g,A_+}\times \cX^-_{\g,A_-})$ for $r=2$. 
Thus it constructs a well defined sc$^+$-multisection on $\partial^{R=\infty}\cX^{\scriptscriptstyle\rm SFT}_A =
\bigcup l_{\g,A_{\pm}}(\cX^+_{\g,A_+} \times \cX^-_{\g,A_-}) \subset \partial\Ti\cX^{\scriptscriptstyle\rm SFT}_A$ that is admissible and has evaluations transverse to the submanifolds $\{0\}\times W_{p_-}^-\times \{0\}\times W_{p_+}^+$ for all pairs $p_-,p_+\in\Crit(f)$. 

Moreover, for $\alpha \in \cI^\pm$ we obtain a pair controlling compactness by pullback of the coherent pairs constructed as in \cite[\S 13]{fh-sft} on the bundles $\cW^\pm_{\gamma,A}$. 
Then the pullback multisections $\ul\k^\pm=(\k_{\alpha}^\pm=\lambda^\pm_{\g,A} \circ (\pr^\iota_\alpha)^*)_{\alpha\in\cI^\pm}$ are sc$^+$, admissible w.r.t.\ the pullback pair, and in general position by the arguments in the proof of Lemma~\ref{lem:iota chain}.

\medskip\noindent{\bf Construction of $\boldsymbol{\lambda^{\scriptscriptstyle\rm SFT}_A}$ and $\boldsymbol{\underline\kappa}$:}
The above constructions determine the right hand sides in the coherence requirements $\l_A^{\scriptscriptstyle\rm SFT}|_{P_A^{-1}(\Ti\cX^{\scriptscriptstyle\rm GW}_A)}=\l_A^{\scriptscriptstyle\rm GW}$ over $\Ti\cX^{\scriptscriptstyle\rm GW}_A\subset \Ti\cX^{\scriptscriptstyle\rm SFT}_A$ in  \eqref{eq:coherent1}, as well as 
$\l_A^{\scriptscriptstyle\rm SFT}|_{P_A^{-1}(\cF_{\g,A_\pm}(2))}=(\lambda^+_{\gamma,A_+}\cdot \lambda^-_{\gamma,A_-})\circ (l_{\g,A_\pm})^{-1}_*$
 on $\bigcup_{\g\in\cP(H),A_-+A_+=A} \cF_{\g,A_\pm}(2) \subset\partial\Ti\cX^{\scriptscriptstyle\rm SFT}_A$ in \eqref{eq:coherent2}, where we denote by
$\cF_{\g,A_\pm}(r):= l_{\g,A_\pm}(\cX^+_{\g,A_+}\times \cX^-_{\g,A_-})\subset\partial\Ti\cX^{\scriptscriptstyle\rm SFT}_A$ the image of the immersion $l_{\g,A_\pm}$ on the ep-groupoids representing $|\cX^\pm_{\g,A_\pm}| =(\ov\ev^{\pm})^{-1}(\D^\pm_r\times M)\subset\cB^\pm_{\scriptscriptstyle\rm SFT}(\g;A)$. 
By admissibility in the prior steps and existence of scale-smooth partitions of unity (see Remark~\ref{rmk:partitions}) these induce for every $A\in H_2(M)$ an admissible sc$^+$-multisection $\lambda^\partial_A: \Ti\cW^\pm_A|_{\partial\Ti\cX_A}\to\Q^+$ which coincides with the prescribed data 
over $\Ti\cX^{\scriptscriptstyle\rm GW}_A \sqcup \bigcup_{\g, A_\pm} \overline{\cF_{\g,A_\pm}(1)} \subset\partial\Ti\cX^{\scriptscriptstyle\rm SFT}_A$. Thus on this closed subset we have general position and transversality of the evaluation map
\begin{equation} 
\ev^+\times \ev^- \,:\; \bigl| \{x\in \partial\cX^{\scriptscriptstyle\rm SFT}_A \,|\, \l^\partial_A(S^{\scriptscriptstyle\rm SFT}_A(x))>0\} \bigr| \;\to\; \C^+ \times M\times \C^-\times M
\end{equation} 
to $\{0\}\times W_{p_-}^-\times \{0\}\times W_{p_+}^+$ for any pair of critical points $p_-,p_+\in\Crit(f)$. 
Then the admissible sc$^+$-multisection $\lambda^{\scriptscriptstyle\rm SFT}_A:\Ti\cW^{\scriptscriptstyle\rm SFT}_A\to \Q^+$ is constructed with these given boundary values -- and auxiliary norm and support prescribed by the coherent pairs controlling compactness -- using Theorem~\ref{thm:transversality} to achieve general position on all of $\Ti\cX^{\scriptscriptstyle\rm SFT}_A$ and extend transversality of the evaluation $\ev^+\times\ev^-$ to $\{0\}\times W_{p_-}^-\times \{0\}\times W_{p_+}^+$ to the entire perturbed zero set $|\{x\in\cX^{\scriptscriptstyle\rm SFT}_A \,|\, \l^{\scriptscriptstyle\rm SFT}_A(S^{\scriptscriptstyle\rm SFT}_A(x))>0\}|$, 
where $|\cX^{\scriptscriptstyle\rm SFT}_A| = (\ov\ev^+)^{-1}(\D^+_1\times M)\cap (\ov\ev^-)^{-1}(\D^-_1\times M)\subset\cB_{\scriptscriptstyle\rm SFT}(A)$. 

As in the proof of Lemma~\ref{lem:iota chain}, the transversality of the evaluation maps implies that the pullbacks  $\ul\k=(\k_{\alpha}=\lambda_{A} \circ (\pr^\iota_\alpha)^*)_{\alpha\in\cI}$ are in general position. They are also admissible with respect to the pullback of pairs controlling compactness. This finishes the construction of the sc$^+$-multisections claimed in (i) with the boundary restrictions required in (iii).

\medskip
\noindent
{\bf Proof of identity:}
By $\L$-linearity of all maps involved, it suffices to fix two generators $p_-,p_+\in{\rm Crit}(f)$ of $CM$ and check that $\iota_{\ul\k^\iota}\la p_- \ra$ and $(SSP_{\ul\k^-} \circ PSS_{\ul\k^+})\la p_- \ra + (-1)^{|p_-|} (\rd \circ h_{\ul\k})\la p_- \ra + (-1)^{|p_-|} (h_{\ul\k} \circ d) \la p_- \ra$ have the same coefficient in $\Lambda$ on $\la p_+\ra$. That is, we claim
\begin{align*}
\hspace{-3mm} \sum_{\begin{smallmatrix}
 A\in H_2(M)  \\ \scriptscriptstyle I^\iota(p_-,p_+;A)=0  \end{smallmatrix} }  \hspace{-5mm}
\#Z^{\ul\k^\iota}(p_-,p_+; A) \cdot T^{\o(A)}
&\; =
\hspace{-5mm}\sum_{\begin{smallmatrix}
 \g\in\cP(H), A_-,A_+\in H_2(M)  \\ \scriptscriptstyle I(p_-,\g;A_+)=I(\g,p_+;A_-)=0
 \end{smallmatrix} }
\hspace{-10mm}
\#Z^{\ul\k^+}(p_-, \g ; A_+) \; \#Z^{\ul\k^-}(\g, p_+; A_-) \cdot T^{\o(A_-) + \o(A_+)} \\
&\qquad + (-1)^{|p_-|}
\sum_{\begin{smallmatrix}
q\in\Crit(f), A\in H_2(M) \\ \scriptscriptstyle I(p_-,q;A)= |q|-|p_+|-1= 0
 \end{smallmatrix} }
\hspace{-5mm}
\#Z^{\ul\k}(p_-,q; A)  \; \#\cM(q,p_+) \cdot T^{\o(A)} \\
&\qquad+ (-1)^{|p_-|}
\sum_{\begin{smallmatrix}
 q\in\Crit(f), A\in H_2(M) \\ \scriptscriptstyle |p_-|-|q|-1 = I(q,p_+;A)=0  
 \end{smallmatrix} }
\hspace{-5mm}
\#\cM(p_-,q)  \; \#Z^{\ul\k}(q,p_+; A) \cdot T^{\o(A)} .
\end{align*}
Here the sums on the right hand side are over counts of pairs of moduli spaces of index $0$. 
From \S\ref{sec:Morse} we have $\cM(q,p_+)=\emptyset$ for $|q|-|p_+|-1<0$ and $\cM(p_-,q)=\emptyset$ for $|p_-|-|q|-1<0$, and general position of the sc$^+$-multisections $\ul\k^{\cdots}$ as in Corollary~\ref{cor:regularize}~(iii) implies $Z^{\ul\k^{\cdots}}(\ldots)=\emptyset$ for $I(\ldots)<0$. 
Thus the right hand side can be rewritten as sum over pairs of moduli spaces with indices summing to zero, and by \eqref{PSS index}, \eqref{iota index}, \eqref{h index} this is moreover equivalent to
\begin{align*}
0 &\;=\; I{(p_-,\g;A_+)} + I{(\g,p_+;A_-)} \;=\; I^\iota(p_-,p_+;A_-+A_+) \;=\; I(p_-,p_+;A_-+A_+) - 1 , \\
0 &\;=\; I{(p_-,q;A)} + |q|-|p_+|-1 \;=\; I(p_-,p_+;A) - 1 , \\
0 &\;=\; |p_-|-|q|-1 + I{(q,p_+;A)} \;=\; I(p_-,p_+;A) - 1 .
\end{align*}
So all sums can be rewritten with the index condition $I(p_-,p_+;A) = 1$ for $A=A_-+A_+\in H_2(M)$, and since the symplectic area is additive $\o(A_-)+\o(A_+)=\o(A_-+A_+)$, it suffices to show the following identity for each $\alpha=(p_-,p_+; A) \in \cI$ with $I(p_-,p_+;A) = 1$, 
\begin{align} \nonumber
(-1)^{|p_-|} \#Z^{\ul\k^\iota}(p_-,p_+; A)
&\; =\; (-1)^{|p_-|} \hspace{-4mm}
\sum_{\begin{smallmatrix}
\scriptscriptstyle \g\in\cP(H) \\ \scriptscriptstyle A_-+A_+=A 
 \end{smallmatrix} } \hspace{-3mm}
\#Z^{\ul\k^+}(p_-, \g ; A_+) \; \#Z^{\ul\k^-}(\g, p_+; A_-) \\
&\quad +   \hspace{-2mm}
\sum_{q\in\Crit(f) } \hspace{-1mm}
\#Z^{\ul\k}(p_-,q; A)  \; \#\cM(q,p_+)  \; + \hspace{-2mm}
\sum_{q\in\Crit(f) } \hspace{-1mm} \#\cM(p_-,q)  \; \#Z^{\ul\k}(q,p_+; A)  .  \label{h claim}
\end{align}
This identity will follow from Corollary~\ref{cor:regularize}~(v) applied to the weighted branched $1$-dimensional orbifold $Z^{\ul\kappa}(\alpha)$ that arises from an admissible sc$^+$-multisection $\kappa_{\alpha}:\cW_{\alpha}\to\Q^+$.
The boundary $\partial Z^{\ul\kappa}(\alpha)$ is given by the intersection with the top boundary stratum $\partial_1\cB(\alpha)\cap \cV_{\alpha} = |\partial_1\cX_{\alpha}|$, and will be determined here -- with orientations computed in \eqref{eq:orient} below.
\begin{align*}
& \partial Z^{\ul\kappa}(\alpha) \;=\; Z^{\ul\kappa}(\alpha) \cap |\partial_1\cX_{\alpha} | \\
&\;=\;  
 Z^{\ul\kappa}(\alpha) \cap   |\partial_0\cX^\iota_{p_-,p_+; A}| 
\quad \sqcup \bigcup_{\begin{smallmatrix}
\scriptscriptstyle \g\in\cP(H) \\ \scriptscriptstyle A_-+A_+=A 
 \end{smallmatrix} }
Z^{\ul\kappa}(\alpha)  \cap  | \partial_0\cX^+_{p_-,\g; A_+} \times  \partial_0\cX^-_{\g,p_+; A_-} |  \\
&\qquad\sqcup
 \bigcup_{q\in{\rm Crit}(f)} Z^{\ul\kappa}(\alpha) \cap  \bigl( \cM(p_-,q) \times | \partial_0\cX_{q,p_+; A}| \bigr) \quad \sqcup\bigcup_{q\in{\rm Crit}(f)} Z^{\ul\kappa}(\alpha) \cap \bigl( |\partial_0\cX_{p_-,q; A}| \times \cM(q,p_+) \bigr)  \\
& \;=\;
Z^{\ul\kappa^\iota}(p_-,p_+;A) 
\quad \sqcup \bigcup_{\g\in\cP(H), A=A_-+A_+} 
 Z^{\ul\kappa^+}(p_-,\g;A_+) \times Z^{\ul\kappa^-}(\g,p_+;A_-) \\
&\qquad\sqcup
 \bigcup_{q\in{\rm Crit}(f)} \cM(p_-,q) \times Z^{\ul\kappa}(q,p_+;A) 
 \quad \sqcup\bigcup_{q\in{\rm Crit}(f)}  Z^{\ul\kappa}(p_-,q;A)  \times \cM(q,p_+).
\end{align*}
Here the second identity uses coherence of the ep-groupoid as in \eqref{eq:homotopy bdy}. 
The third identity follows from coherence of sections $S_\alpha^{\cdots}$ and sc$^+$multisections $\kappa_\alpha^{\cdots}$ stated in (ii), (iii), and the fact from Corollary~\ref{cor:regularize}~(iv) that perturbed zero sets $Z^{\ul\k^{\cdots}}(\alpha)\subset|\partial_0\cX^{\cdots}_\alpha|$ are contained in the interior of the polyfolds when the Fredholm index is $0$. 
For the second summand we moreover use Lemma~\ref{lem:Cartesian} which ensures that each restriction $\kappa_{\alpha}|_{P_\alpha^{-1}(\cF)}$ to a face 
$\cF= \partial_0\cX^+_{p_-,\g; A_+} \times  \partial_0\cX^-_{\g,p_+; A_-} \subset \partial_1\cX_{p_-,p_+; A}$, given by $\k^+_{p_-,\g; A_+} \cdot  \k^-_{\g,p_+; A_-}$, is in general position to the section $S^+_{p_-,\g; A_+} \times  S^-_{\g,p_+; A_-}$.
Then its perturbed zero set $Z^{\ul\kappa^\iota}(p_-,p_+;A)\cap |\cF|$ is contained in the interior
$ \partial_0 |\cX^+_{p_-,\g; A_+} \times \cX^-_{\g,p_+; A_-} | = | \partial_0\cX^+_{p_-,\g; A_+} \times  \partial_0\cX^-_{\g,p_+; A_-} |$ as the complement of the pairs of points $(x^+,x^-)$ with
\begin{align*}
0 \;=\; \kappa_{p_-,p_+;A}(S_{p_-,p_+;A}(x^+,x^-))
&\;=\; 
(\k^+_{p_-,\g; A_+} \cdot  \k^-_{\g,p_+; A_-})\bigl(
(S^+_{p_-,\g; A_+} \times  S^-_{\g,p_+; A_-})(x^+,x^-)\bigr) \\
&\;=\;
\k^+_{p_-,\g; A_+}( S^+_{p_-,\g; A_+}(x^+)) \cdot 
\k^-_{\g,p_+; A_-}(S^-_{\g,p_+; A_-}(x^-) ) . 
\end{align*}
Since a product in $\Q^+=\Q\cap[0,\infty)$ is nonzero exactly when both factors are nonzero, this identifies the objects of the perturbed zero set of $\kappa_{p_-,p_+;A}$ with the product of perturbed zero objects for $\kappa^\pm$,
\begin{align*}
& \bigl\{ (x^+,x^-) \in \cF \,\big|\, \kappa_{p_-,p_+;A}(S_{p_-,p_+;A}(x^+,x^-))>0 \bigr\}  \\
&\;=\; \bigl\{ x^+ \in \cX^+_{p_-,\g; A_+} \,\big|\,  \k^+_{p_-,\g; A_+}( S^+_{p_-,\g; A_+}(x_+))   >0 \bigr\} 
\times
 \bigl\{ x^- \in \cX^-_{\g,p_+; A_-}  \,\big|\, \k^-_{\g,p_+; A_-}(S^-_{\g,p_+; A_-}(x^-) )  >0 \bigr\} .
\end{align*}
And the realization of this set is precisely $Z^{\ul\kappa^+}(p_-,\g;A_+) \times Z^{\ul\kappa^-}(\g,p_+;A_-)$, as claimed above.

\medskip
\noindent
{\bf Computation of orientations:} 
 To prove the identity \eqref{h claim} it remains to compute the effect of the orientations in Remark~\ref{rmk:iota,h orient} on the algebraic identity in Corollary~\ref{cor:regularize}~(v) that arises from the boundary $\partial Z^{\ul\kappa}(\alpha)$ of the 1-dimensional weighted branched orbifolds arising from regularization of the moduli spaces with index $I(\alpha)=I(p_-,p_+;A)=1$. Here $Z^{\l_A^{\scriptscriptstyle\rm SFT}}$ is of odd dimension 
with oriented boundary determined by the orientation relations in Assumption~\ref{ass:iso}~(iii)(b) and (c) as
$$ 
\partial_1 Z^{\l_A^{\scriptscriptstyle\rm SFT}} \;=\; Z^{\l_A^{\scriptscriptstyle\rm SFT}} \cap \partial_1 \cB_{\scriptscriptstyle\rm SFT}(A) 
\;=\; (-1) Z^{\l_A^{\scriptscriptstyle\rm GW}} 
\;\sqcup\; \textstyle
\bigsqcup_{\begin{smallmatrix}
\scriptscriptstyle \g \in \cP(H)\\
\scriptscriptstyle A_- + A_+ = A\\
  \end{smallmatrix}}  Z^{\l_{\g,A_+}^+} \times Z^{\l_{\g,A_-}^-} . 
$$
Moreover, the index of $\s_{\scriptscriptstyle\rm SFT}$ is $I(\alpha)=|p_-|-|p_+| +2c_1(A)+1=1$, so we compute orientations in close analogy to \eqref{eq:orient-iota} -- while also giving an alternative identification of the boundary components -- 
\begin{align}
\partial Z^{\ul\kappa}(\alpha) &\;=\;
\partial_1\bM(p_-,M) \; \leftsub{\ev}{\times}_{\ev}\; Z^{\l_A^{\scriptscriptstyle\rm SFT}}\; \leftsub{\ev}{\times}_{\ev} \; \partial_0\bM(M,p_+)  \nonumber\\
&\quad\sqcup\;
(-1)^{\dim \bM(p_-,M)} \;
\partial_0\bM(p_-,M) \; \leftsub{\ev}{\times}_{\ev} \; \partial_1 Z^{\l_A^{\scriptscriptstyle\rm SFT}}\; \leftsub{\ev}{\times}_{\ev} \; \partial_0\bM(M,p_+) \nonumber \\
&\quad\sqcup\;
(-1)^{\dim \bM(p_-,M)+1} \;
\partial_0\bM(p_-,M) \; \leftsub{\ev}{\times}_{\ev} \; \partial_0 Z^{\l_A^{\scriptscriptstyle\rm SFT}}\; \leftsub{\ev}{\times}_{\ev} \; \partial_1\bM(M,p_+) \nonumber \\
&\;=\;
\bigl(\; \textstyle\bigsqcup_{q\in\Crit f} \cM(p_-,q) \times \cM(q,M) \;\bigr) \; \leftsub{\ev}{\times}_{\ev}\; Z^{\l_A^{\scriptscriptstyle\rm SFT}} \;\leftsub{\ev}{\times}_{\ev} \; \cM(M,p_+) \label{eq:orient} \\
&\quad\sqcup\;
(-1)^{|p_-| + |p_+|}
\cM(p_-,M) \; \leftsub{\ev}{\times}_{\ev} \; Z^{\l_A^{\scriptscriptstyle\rm SFT}}\; \leftsub{\ev}{\times}_{\ev} \; \bigl(\; \textstyle \bigsqcup_{q\in\Crit f} \cM(M,q) \times \cM(q,p_+) \;\bigr) \nonumber\\
&\quad\sqcup\;
(-1)^{|p_-|} \;
\cM(p_-,M) \; \leftsub{\ev}{\times}_{\ev} \; \bigl(\; \textstyle\bigsqcup_{\g\in\cP(H),A=A_-+A_+} 
Z^{\l_{\g,A_+}^+} \times Z^{\l_{\g,A_-}^-} \; \bigr) \; \leftsub{\ev}{\times}_{\ev} \; \cM(M,p_+) \nonumber \\
&\quad\sqcup\;
(-1)^{|p_-|+1} \;
\cM(p_-,M) \; \leftsub{\ev}{\times}_{\ev} \; Z^{\l_A^{\scriptscriptstyle\rm GW}}\; \leftsub{\ev}{\times}_{\ev} \; \cM(M,p_+) \nonumber \\
&\;=\;
\textstyle\bigsqcup_{q\in\Crit f} \cM(p_-,q) \times Z^{\ul\k}(q,p_+;A)
\;\;\sqcup\;\;  \bigsqcup_{q\in\Crit f} Z^{\ul\k}(p_-,q;A)\times \cM(q,p_+) \nonumber\\
&\quad
\;\sqcup\;
(-1)^{|p_-|} \; \textstyle\bigsqcup_{\g\in\cP(H),A=A_-+A_+} 
Z^{\ul\k^+}(p_-,\g;A_+) \times Z^{\ul\k^-}(\g,p_+;A_-) \nonumber \\
&\quad\sqcup\;
(-1)^{|p_-|+1} \; Z^{\ul\k^\iota}(p_-,p_+;A) .  \nonumber
\end{align}
This computation should be understood in a neighbourhood of a solution, so in particular with scale-smooth evaluation maps to $\C^\pm\times M$.
Based on this, Corollary~\ref{cor:regularize}~(v) implies -- as claimed -- 
\begin{align*}
0 &\;=\; h_{\underline\kappa} ( \rd \la p_- \ra ) + \rd ( h_{\underline\kappa} \la p_- \ra ) +   (-1)^{|p_-|} SSP_{\underline\kappa^-} \bigl( PSS_{\underline\kappa^+} \la p_- \ra \bigr) 
- (-1)^{|p_-|}\iota_{\underline\kappa^\iota}\la p_- \ra  \\
&\;=\; \bigl(\; h \circ \rd + \rd \circ h  +   SSP\circ PSS - \iota \;\bigr) \la p_- \ra  .
\end{align*}

\vspace{-5mm}

\end{proof}

\appendix

\section{Summary of Polyfold Theory} \label{sec:polyfold}

This section gives an overview of the main notions of polyfold theory that are used in this paper. 
The following language is used to describe settings with trivial isotropy.\footnote{Trivial isotropy would be guaranteed in our settings by an almost complex structure $J$ for which there are no nonconstant $J$-holomorphic spheres.}

\begin{rmk}  \rm \label{rmk:Mpolyfolds}
\begin{itemlist}
\item[(i)]
An {\bf M-polyfold without boundary} is analogous to the notion of a Banach manifold: While the latter are locally homeomorphic to open subsets of a Banach space, an M-polyfold is locally homeomorphic to the image $\cO=\im\rho$ of a retract $\rho:U\to U$ of an open subset $U\subset E$ of a Banach space $E$. While $\rho$ is generally not classically differentiable, it is required to be scale-smooth (sc$^\infty$) with respect to a scale structure on $E$, which is indicated by $\E$.  

\item[(i')]
An {\bf M-polyfold}, as defined in \cite[Def.2.8]{HWZbook}, is a paracompact Hausdorff space $X$ together with an atlas of charts $\phi_\iota: U_\iota \to \cO_\iota\subset [0,\infty)^{s_\iota}\times  \E^\iota$ (i.e.\ homeomorphisms between open sets $U_\iota\subset X$ and sc-retracts $\cO_\iota$ such that $\cup_\iota U_\iota = X$), whose transition maps are sc-smooth.

For $k\in\N_0$ the $k$-th boundary stratum $\partial_k X$ is the set of all $x\in X$ of degeneracy index $d(x)=k$ given\footnote{
The degeneracy index $d(x)\in\N_0$ in \cite[Def.2.13,Thm.2.3]{HWZbook} is a priori independent of the choice of chart $\phi_\iota$ only for points in a dense subset $X_\infty\subset X$ specified in Remark~\ref{rmk:levels}. With that $d(x):=\max\{\limsup d(x_i) \,|\, X_\infty\ni x_i \to x \} $ is well defined for all $x\in X$ and can also be computed in any fixed chart.
} 
by the number of components equal to $0$ for the point in a chart $\phi_\iota(x)\in [0,\infty)^{s_\iota}\times \E^\iota$.
In particular, $\partial_0 X$ is the interior of $X$.

\item[(ii)]
A {\bf strong bundle} over an M-polyfold $X$, as defined in \cite[Def.2.26]{HWZbook}, is a sc-smooth surjection $P:W\to X$ with linear structures on each fiber $W_x=P^{-1}(x)$ for $x\in X$, and an equivalence class of compatible strong bundle charts, which in particular encode a sc-smooth subbundle $W\supset W^1\to X$ whose fiber inclusions $W_x^1\hookrightarrow W_x$ are compact and dense. 

\item[(iii)]
The notion of {\bf sc-Fredholm} for a scale smooth section $S:X\to W$ of a strong bundle in \cite[Def.3.8]{HWZbook} encodes elliptic regularity and a nonlinear contraction property \cite[Def.3.6,3.7]{HWZbook}. The latter is a stronger condition than the classical notion of linearizations being Fredholm operators, and is crucial to ensure an implicit function theorem; see \cite{counterex}. 
\end{itemlist}
\end{rmk}

A more detailed survey of these trivial isotropy notions can be found in \cite{usersguide}.
Then the generalization to nontrivial isotropy is directly analogous to the notion of smooth sections of orbi-bundles, in which orbifolds are realizations of \'etale proper groupoids \cite{moer}.

\begin{rmk} \rm  \label{rmk:polyfolds}
A sc-Fredholm section $\sigma:\cB\to\cE$ of a strong polyfold bundle as introduced in \cite[Def.16.16,16.40]{HWZbook} is a map between topological spaces together with an equivalence class of sc-Fredholm section functors $s:\cX\to \cW$ of strong bundles $\cW$ over ep-groupoids $\cX$, whose realization $|s|:|\cX|\to|\cW|$ together with homeomorphisms $|\cX|:={\rm Obj}_{\cX} / {\rm Mor}_{\cX}\cong\cB$ and $|\cW|\cong\cE$ induces $\sigma$. 
To summarize these notions we use conventions of \cite{HWZbook} in denoting object and morphism spaces as ${\rm Obj}_\cX=X$ and ${\rm Mor}_\cX=\bX$. These will be equipped with M-polyfold structures, so that the $k$-th boundary stratum of a polyfold $\cB\cong|\cX|$ is given as $\partial_k\cB \cong \partial_k X / \bX \subset |\cX|$ for all $k\in\N_0$.

\begin{itemlist}
\item[(i)]
An {\bf ep-groupoid} as in \cite[Def.7.3]{HWZbook} is a groupoid $\cX=(X,\bX)$ equipped with M-polyfold structures on the object and morphism sets such that 
all structure maps are local sc-diffeomorphisms and every $x\in X$ has a neighbourhood $V(x)$ such that $t:s^{-1}\bigl( {\rm cl}_X(V(x))\bigr) \to X$ is proper.
As in \cite[\S7.4]{HWZbook} we require that the realization $|\cX|$ is paracompact and thus metrizable. 

\item[(ii)]
A {\bf strong bundle} as in \cite[Def.8.4]{HWZbook} over the ep-groupoid $\cX$ is a pair $(P,\mu)$ of a strong bundle $P:W\to X$ and a strong bundle map $\mu:\bX\leftsub{s}{\times}_P W\to W$ so that $P$ lifts to a functor $\cP:\cW\to\cX$ from an ep-groupoid $\cW=(W,\bW)$ induced by $(P,\mu)$.
Then $\cP$ restricts to a functor $\cW^1\to\cX$ on the full subcategory whose object space is the sc-smooth subbundle $W^1\subset W$. 

\item[(iii)]
A  {\bf sc-Fredholm section functor} of the strong bundle $\cP:\cW\to \cX$ 
as in \cite[Def.8.7]{HWZbook} is a functor $S:\cX\to \cW$ that is sc-smooth on object and morphism spaces, satisfies $\cP\circ S = \id_\cX$, and such that $S:X\to W$ is sc-Fredholm on the M-polyfold $X$.
\end{itemlist}
\end{rmk}

Now a polyfold description of a compact moduli space $\bM$ is a sc-Fredholm section $\s:\cB\to\cE$ of a strong polyfold bundle with zero set $\s^{-1}(0)\cong \bM$. 
The polyfold descriptions used in this paper are obtained as fiber products of existing polyfolds and sc-Fredholm sections over them. This requires a technical shift in levels described in the following remark, and a notion of submersion below. 

\begin{rmk}\label{rmk:levels} \rm
Polyfolds carry a level structure $\cB_\infty \subset \ldots \subset \cB_1 \subset \cB_0=\cB$ as follows: 
For any M-polyfold $X$, in particular the object space of the ep-groupoid representing $\cB=|\cX|$, 
a sequence of dense subsets $X_\infty \subset \ldots \subset X_1 \subset X_0= X$ is induced by the 
scale structures $\E^\iota=(E^\iota_m)_{m\in\N_0}$ of the charts, 
that is $X_m=\bigcup_{\iota} \phi_\iota^{-1}(\cO_\iota\cap \R^{s_\iota}\times E^\iota_m)$.
Then $\cB_m:=\qu{X_m}{{\rm Mor}_\cX}$ is well defined since morphisms of $\cX$ -- locally represented by scale-diffeomorphisms -- preserve the levels on ${\rm Obj}_\cX=X$. 

The restriction $\sigma|_{\cB_m}$ of a sc-Fredholm section $\s:\cB\to\cE$ is again sc-Fredholm with values in $\cE_m$, and the choice of such a shift in levels is irrelevant for applications since the zero set $\s^{-1}(0)\subset\cB_\infty$ -- as well as the perturbed zero set for any admissible perturbation -- is always contained in the so-called ``smooth part'' that is densely contained in each level $\cB_\infty\subset\cB_m$. 

For a finite dimensional manifold or orbifold $M$ -- such as the Morse trajectory spaces in \S\ref{ssec:Morse} -- viewed as polyfold, the level structure is trivial $M_\infty=\ldots = M_1=M_0=M$. 
\end{rmk}

\begin{dfn} \label{def:submersion} \cite[Def.5.9]{Ben-fiber}
A sc$^\infty$ functor $f:\cX \to M$ from an ep-groupoid $\cX = (X, {\bf X})$ to a finite dimensional manifold $M$ is a {\bf submersion} if for all $x \in X_{\infty}$ the tangent map $\rD_x f : \rT^R_xX \rightarrow \rT_{f(x)}M$ is surjective, where $\rT^R_xX$ is the reduced tangent space \cite[Def.2.15]{HWZbook}.

Consider in addition a sc-Fredholm section functor $S : \cX \rightarrow \cW$. 
Then the sc$^\infty$ functor $f : \cX \rightarrow M$ is {\bf $S$-compatibly submersive} if for all $x \in X_{\infty}$ there exists a sc-complement 
$L\subset \rT_x^RX$ of $\ker(\rD_xf) \cap \rT_x^RX$ 
and a tame sc-Fredholm chart for $S$ at $x$ \cite[Def.5.4]{Ben-fiber} 
in which the change of coordinates $\psi : \cO \to [0,\infty)^s \times \R^{k-s} \times \mathbb{W}$ that puts $S$ in basic germ form -- which by tameness has the form $\psi(v,e) = (v,\overline{\psi}(e))$ for $(v,e)\in \cO \subset [0,\infty)^s \times \mathbb{E}$ and a linear sc-isomorphism $\overline{\psi}$ -- moreover satisfies $\overline{\psi}(L) \subset \{0\}^{k-s} \times \mathbb{W}$,
where the chart identifies $L \subset \rT_x^RX\cong\rT_0^R\cO= \{0\} \times \mathbb{E}$.

More generally, given a smooth submanifold $N \subset M$, the sc$^{\infty}$ functor $f$ is {\bf transverse} to $N$ if for all $x \in f^{-1}(N) \cap X_{\infty}$ we have $\rD_xf(\rT_x^RX) + \rT_{f(x)}N = \rT_{f(x)}M$, and $f$ is {\bf $S$-compatibly transverse} to $N$ if there exists a sc-complement $L$ of $(\rD_xf)^{-1}(\rT_{f(x)}(N)) \cap \rT_x^RX$ satisfying the above condition.
\end{dfn}  

The purpose of giving a moduli space a polyfold description is to utilize the perturbation theory for sc-Fredholm sections over polyfolds, which allows to ``regularize'' the moduli space by associating to it a well defined cobordism class of weighted branched orbifolds. 
(For a technical statement see Corollary~\ref{cor:regularize} and the references therein.)
Since the ambient space $|\cX|$ is almost never locally compact, this requires ``admissible perturbations'' of the section to preserve compactness of the zero set. This admissibility is determined by the following data introduced in \cite[Def.12.2,15.4]{HWZbook}.

\begin{dfn} \label{def:control}
A {\bf saturated open subset} $\cU\subset \cX$ of an ep-groupoid $\cX=(X,\bX)$ is an open subset $\cU\subset X$ with $\pi^{-1}(\pi(\cU))=\cU$, where $\pi:X \to |\cX|=\qu{X}{\bX}$ is the projection to the realization. 

A {\bf pair controlling compactness} for a sc-Fredholm section $S:\cX\to\cW$ of a strong bundle $P:\cW\to\cX$ consists of an auxiliary norm $N:\cW[1]\to[0,\infty)$ (see \cite[Def.12.2]{HWZbook}) and a saturated open subset $\cU\subset \cX$ that contains the zero set $S^{-1}(0)\subset\cU$, such that $\bigl| \{ x\in \cU \,|\, N(S(x))\leq 1 \} \bigr| \subset |\cX|$ has compact closure.

Given such a pair, a section $s:\cX\to\cW$ is {\bf $\boldsymbol{(N,\mathcal U)}$-admissible} if $N(s(x))\leq 1$ and $\supp s \subset \cU$. 
\end{dfn}

The construction of perturbations moreover requires scale-smooth partitions of unity, which will be guaranteed by the following standing assumptions. 

\begin{rmk}\label{rmk:partitions} \rm
Throughout this paper we assume that the realizations $|\cX|$ of ep-groupoids are paracompact, and the Banach spaces $E$ in all M-polyfold charts are Hilbert spaces. 
This guarantees the existence of scale-smooth partitions of unity by \cite[\S5.5,\S7.5.2]{HWZbook}.
In order to guarantee the same on every level $\cB_m$ as discussed in Remark~\ref{rmk:levels}, we moreover assume that each scale structure $\E=(E_m)_{m\in\N_0}$ consists of Hilbert spaces $E_m$. 
These assumptions hold in applications, such as the ones cited \cite{hwz-gw,fh-sft}. 
Then paracompactness and thus existence of scale-smooth partitions of unity on every level is guaranteed by \cite[Prop.7.12]{HWZbook}. 
\end{rmk}

When discussing coherence of perturbations of a system of sc-Fredholm sections, the boundaries are described in terms of Cartesian products of polyfolds, bundles, and sections. So we will make use of Cartesian products of multivalued perturbations as follows, to obtain multisections over the boundary as summarized in the subsequent remark.

\begin{lem} \label{lem:Cartesian}
Let  $S_1:\cX_1\to\cW_1$ and $S_2:\cX_2\to\cW_2$ be sc-Fredholm section of strong bundles $P_i:\cW_i\to\cX_i$ over ep-groupoids. Then the Cartesian product
$\cX_1\times\cX_2$ is naturally an ep-groupoid and 
$(S_1\times S_2) : \cX_1\times\cX_2 \to \cW_1\times\cW_2$ is a sc-Fredholm section of the strong bundle $P_1\times P_2$. 

Moreover, if $\lambda_i:\cW_i\to\Q^+$ are sc$^+$-multisections for $i=1,2$, then there is a well defined sc$^+$-multisection $\lambda_1\cdot \lambda_2:\cW_1\times\cW_2\to\Q^+$ given by $(\lambda_1 \cdot \lambda_2)(w_1,w_2) = \lambda_1(w_1) \cdot \lambda_2(w_2)$. If, for $i = 1,2,$ the sections $\lambda_i$ are $(N_i,\cU_i)$-admissible for some fixed pair controlling compactness as in Definition~\ref{def:control}, then $\lambda_1 \cdot \lambda_2$ is $(\max(N_1,N_2),\cU_1 \times\cU_2)$-admissible.
Finally, if $\lambda_i$ is in general position to $S_i$ for $i=1,2$ then $\lambda_1\cdot \lambda_2$ is in general position to $S_1\times S_2$. 
\end{lem}
\begin{proof}
A detailed treatment of sc-Fredholmness of the product section $S_1 \times S_2$ can be found in \cite[Lemma~7.2]{Ben-fiber}. The remaining statements follow easily from the definitions in \cite{HWZbook} (as do the statements in the first paragraph).

Recall in particular from \cite[Def.13.4]{HWZbook} that a sc$^+$-multisection on a strong bundle $P:\cW\to\cX$ is a functor $\lambda:
\cW\to\Q^+$ that is locally of the form $\lambda(w) = \sum_{\{ j \,|\, w = p_j(P(w))\}} q_j$, represented by  sc$^+$-sections $p_1,\ldots,p_k:\cV\to P^{-1}(\cV)$ (i.e.\ sc$^\infty$ sections of $\cW^1$; see \cite[Def.2.27]{HWZbook}) and weights $q_1,\ldots,q_k\in\Q\cap[0,\infty)$ with $\sum_j q_j = 1$. 
Then for local sections $p_j^i$ and weights $q_j^i$ representing $\lambda_i$ for $i=1,2$, the multisection $\lambda_1\cdot \lambda_2$ is locally represented by the sections $(p_j^1, p_{j'}^2)$ with weights $q_j^1q_{j'}^2$, and all admissibility and general position arguments are made at the level of these local sections. 

In particular, the $(N_i,\cU_i)$-admissibility can be phrased as the existence of local representations by sections with $N_i(p^i_j(x))\leq 1$ and $Z(S_i,p^i_j) := \{x\in \cV_i \,|\, \exists\, t\in[-1,1]: S_i(x) = t p^i_j(x) \}\subset\cU_i$. Then $(\max(N_1,N_2),\cU_1 \times\cU_2)$-admissibility uses the observation\\
$\displaystyle \{(x_1,x_2) \,|\, \exists\, t\in[-1,1]: (S_1,S_2)(x_1,x_2) = t (p^1_j(x_1),p^2_{j'}(x_2)
 \} \subset Z(S_1,p^1_{j})  \times Z(S_2,p^2_{j'}) \subset\cU_1\times\cU_2 $.
\end{proof}

\begin{rmk}\rm \label{rmk:faces}
Let $P : \cW\to\cX$ be a strong bundle over a tame ep-groupoid $\cX=(X,\bX)$. Then for every $x\in X_\infty$ there is a chart $\phi: U_x \to \cO$ from a locally uniformizing\footnote{
A neighbourhood $U_x\subset X$ forms a local uniformizer as in \cite[Def.7.9]{HWZbook} if the morphisms between points in $U_x$ are given by a local action of the isotropy group $G_x$. 
} 
neighbourhood $U_x\subset X$ of $x$ to a sc-retract $\cO\subset[0,\infty)^n\times\E$, with $\phi(x)=0$ lying in the intersection of the $n$ local faces $\cF_k:=\phi^{-1}(\{ (\ul v,e) \in [0,\infty)^n\times\E\,|\, v_k=0 \})$ which cover the boundary $\partial X\cap U_x = \bigcup_{k=1}^n \cF_k$.

Now a {\bf sc$\boldsymbol{^+}$-multisection over the boundary} is a functor $\lambda^{\partial} : P^{-1}(\partial \cX) \rightarrow \mathbb{Q}^+$ whose restriction $\lambda^{\partial}|_{ P^{-1}(\cF_k)}$ to each local face is a sc$^+$-multisection of the strong bundle $P^{-1}(\cF_k)\to\cF_k$. 
In the presence of a sc-Fredholm section $S:\cX\to\cW$, such a sc$^+$-multisection is {\bf in general position over the boundary} if for each intersection of faces $\cF_K:= \bigcap_{k\in K} \cF_k\subset \partial X$ the restriction of the perturbed multi-section $\lambda^\partial\circ S|_{\cF_K}: P^{-1}(\cF_K)\to\Q^+$ has surjective linearizations at all solutions. 
If, moreover, $(N,\cU)$ is a pair controlling compactness, then $\lambda^{\partial}$ is $(N,\cU)$-admissible if each restriction $\lambda^{\partial}|_{ P^{-1}(\cF_k)}$ is admissible w.r.t.\ the pair $(N|_{P^{-1}(\cF_k)},\cU\cap \cF_k)$. 

In our applications, as described in Assumption~\ref{ass:iso}, the local faces $\cF_k$ are images of open subsets of global face immersions $l_\cF : \cF \to \partial\cX$, where each $\cF$ is a Cartesian product of two polyfolds, and the restriction to the interior $l_\cF|_{\partial_0\cF}$ is an embedding into the top boundary stratum $\partial_1\cX$. The bundles over each face are naturally identified with the pullbacks $l_\cF^*\cW$, and then the pushforwards of sc$^+$-multisections $\lambda_\cF:l_\cF^*\cW \to \Q^+$ form a sc$^+$-multisection over the boundary $\lambda^{\partial} : P^{-1}(\bigcup \im\lambda_\cF) \rightarrow \mathbb{Q}^+$ if they agree on overlaps and self-intersections of the immersions $l_\cF$, at the boundary $\partial\cF$ of the faces. 
In this setting, general position of $\lambda^\partial$ is equivalent to general position of the multisections $\lambda_\cF$. 
\end{rmk}

The following perturbation theorem allows us to refine the construction of coherent perturbations in  \cite{fh-sft} for the SFT moduli spaces such that moreover the evaluation maps from the perturbed solution sets are transverse to the unstable and stable manifolds in the symplectic manifold. 
This is a generalization of the polyfold perturbation theorem over ep-groupoids and the extension of transverse perturbations from the boundary \cite[Theorems~15.4,15.5]{HWZbook} (with norm bound given by $h\equiv 1$ for simplicity). 
Another version of this -- with the submanifolds representing cycles whose Gromov-Witten invariants are then obtained as counts -- also appears in \cite{wolfgang, wolfgangorbifold}. 
We are working under the assumptions made in this section -- e.g.\ paracompactness -- without further mention.
The limitation to finitely many submanifolds in the extension result seems to be of technical nature; we expect that joint work of the first author with Dusa McDuff
 -- on coherent finite dimensional reductions of polyfold Fredholm sections -- 
 will establish the result for countably many submanifolds.

\begin{thm} \label{thm:transversality}
Suppose $S:\cX\to\cW$ is a sc-Fredholm section functor of a strong bundle $P : \cW\to\cX$ over a tame ep-groupoid $\cX$ with compact solution set $|S^{-1}(0)| \subset |\cX|$, 
and let $(N,\cU)$ be a pair controlling compactness. 
Moreover, let $e:\cX\to M$ be a sc$^0$-map to a finite dimensional manifold $M$ which has a sc$^{\infty}$ submersive restriction $e|_{\cV} : \cV \rightarrow M$ on a saturated open set $\cV \subset \cX$. 

Then, for any countable collection of smooth submanifolds $(C_i \subset M)_{i\in I}$ with $e^{-1}\bigl( \,\overline{\cup_{i \in I}(C_i)} \,\bigr) \subset \cV$, there exists an $(N,\cU)$-admissible sc$^+$-multisection $\l:\cW\to\Q^+$ so that $(S,\lambda)$ is in general position (see \cite[Definition~15.6]{HWZbook}) and the restriction $e|_{Z^\lambda}:Z^\lambda \to M$ to the perturbed zero set $Z^\lambda=|\{x\in X \,|\, \l(S(x))>0\}|$ is in general position\footnote{
General position to $C_i$ requires transversality to $C_i$ of each restriction $e|_{Z^\lambda\cap \cF_K}$ to the perturbed solution set within an intersection of local faces $\cF_K=\bigcap_{k\in K} \cF_k$ as defined in Remark~\ref{rmk:faces}, including for $\cF_\emptyset:=Z^\lambda$. 
} to the submanifolds $C_i$ for all $i\in I$. 

Moreover, suppose $I$ is finite and $\lambda^{\partial} : P^{-1}(\partial \cX) \rightarrow \mathbb{Q}^+$ for some $0<\alpha<1$ is an $(\frac 1\alpha N,\cU)$-admissible structurable sc$^+$-multisection in general position over the boundary
such that the restriction $e|_{Z^\partial}: Z^\partial \to M$ to the perturbed zero set in the boundary $Z^{\partial} := |\{ x\in \partial\cX \,|\,  \lambda^{\partial}( S(x)) > 0 \}|$ is in general position\footnote{
This requires general position of each restriction $e|_{Z^\lambda\cap \cF_k}$ to a local face $\cF_k\subset \partial\cX$ as defined in Remark~\ref{rmk:faces}.
} to the submanifolds $C_i$ for all $i\in I$. 
Then $\lambda$ above can be chosen with $\lambda|_{P^{-1}(\partial \cX) } = \lambda^{\partial}$.
\end{thm}

\begin{proof}
Our proof follows the perturbation procedure of \cite[Theorem~15.4]{HWZbook}, which proves the special case when there is no condition on a map $e : \cX \rightarrow M$, i.e.\ when $M = \{pt\}$ and $C_i = \{pt\}$. To obtain the desired transversality of $e$ to the submanifolds $C_i \subset \cM$ we will go through the proof and indicate adjustments in three steps:
A local stabilization construction, which adds a finite dimensional parameter space to cover the cokernels near a point $x\in S^{-1}(0)$; a local-to-global argument which combines the local constructions into a global stabilization which covers the cokernels near $S^{-1}(0)$; and a global Sard argument which shows that regular values yield transverse perturbations. 
Within these arguments we need to consider restrictions to any intersection of faces to ensure general position to the boundary, use submersivity of $e$ to achieve transversality to the $C_i$, and work with multisections due to isotropy. 
The statement with prescribed boundary values $\lambda^{\partial}$ generalizes the extension result \cite[Theorem~15.5]{HWZbook}, which hinges on the fact that general position over the boundary persists in an open neighbourhood -- something that is generally guaranteed only for finitely many transversality conditions; see the end of this proof. The first step in any construction of perturbations is the existence of local stabilizations which cover the cokernels, as follows.

\medskip
\noindent
 {\bf Local stabilization constructions:}
For every zero $x \in S^{-1}(0)$ of the unperturbed sc-Fredholm section 
we construct a finite dimensional parameter space $\mathbb{R}^l$ 
for $l=l_x\in\N_0$ and sc$^+$-multisection
\begin{align*}
\tilde{\Lambda}^x \,:\; \mathbb{R}^l \times \cW &\rightarrow \mathbb{Q}^+
, \qquad (t,w) \mapsto \Lambda^x_t(w)
\end{align*}
such that 
$\L^x_0$ is the trivial multisection, i.e.\ $\Lambda^x_0(0)=1$, $\Lambda^x_0(w)=0$ for $w\in \cW_x\less\{0\}$. 
This multisection $\tilde{\Lambda}^x$ is viewed as local perturbation near 
$(0,x)$ of a sc-Fredholm section functor $\ti S^x$ of a bundle $\ti P^x$, 
\begin{align*}
\tilde{S}^x\,:\; \mathbb{R}^l \times \cX &\rightarrow \mathbb{R}^l \times \cW
\qquad & \tilde{P}^x \,: \; \mathbb{R}^l \times \cW &\rightarrow \mathbb{R}^l \times \cX \\
(t,y) &\mapsto (t,S(y))
& (t,w) &\mapsto (t,P(w)) .
\end{align*}
It is constructed in \cite{HWZbook} to be structurable in the sense of \cite[Def.13.17]{HWZbook}, in general position 
in the sense that
the linearization $\rT_{(\tilde{S}^x,\tilde{\Lambda}^x)}(0,x) : \rT_0\mathbb{R}^l \times \rT^R_x X \rightarrow W_x$ is surjective\footnote{
This is shorthand for $\tilde S^x + p_j$ having surjective linearization for every section $p_j$ in a local representation of $\ti\L^x$ with $\ti S^x(0,x) = 0 = p_j(0,x)$, and restricted to the reduced tangent space $\rT^R_x X$. 
}
and admissible in the sense that
the domain support of $\tilde{\Lambda}^x$ is contained in 
$\cU$ and the auxiliary norm is bounded linearly, $N(\Lambda)(t,y)\leq c_x |t|$ for some constant $c_x$.  
In case $x \in \cV\cap S^{-1}(0)$ we refine this construction to require surjectivity of the restrictions 
\begin{equation} \label{eq:transverseonkernel}
\rT_{(\tilde{S}^x,\tilde{\Lambda}^x)}(0,x)|_{\rT_0\mathbb{R}^l \times K_x} \,:\; \rT_0\mathbb{R}^l \times K_x \;\to\; W_x , 
\end{equation}
where $K_x := \ker (\rD_x e|_{\rT_x^R X} ) \subset \rT_x^R X$ is the kernel of the linearization 
$\rD_xe : \rT_x^R X \rightarrow \rT_{e(x)}M$ restricted to the reduced tangent space. 
For that purpose note that $e$ is sc$^\infty$ near $x$ by assumption, so has a well defined linearization, and since its codomain is finite dimensional, its kernel has finite codimension. Moreover $\im\rD_x S\subset W_x$ has finite codimension by the sc-Fredholm property of $S$, and the reduced tangent space $\rT_x^R X\subset\rT_x X$ has finite codimension by the definition of M-polyfolds with corners. 
Thus we can find finitely many vectors $w^1,\ldots,w^l \in W_x$ which together with $\rD_x S(K_x)$ span $W_x$. These vectors are extended to sc$^+$-sections of the form $p^j(t,y)=\sum t_j w^j(y)$, multiplied with sc$^\infty$ cutoff functions of sufficiently small support, and pulled back by local isotropy actions to construct the functor $\tilde{\Lambda}^x$ as in \cite[Thm.15.4]{HWZbook}.
We claim that this yields the following local properties with respect to the sc$^\infty$ functor
$$
\ti e^x \,:\; \mathbb{R}^{l} \times \cV \to M, \qquad (t,y) \mapsto e(y) . 
$$

\medskip
\noindent
\ul{Local stabilization properties:} 
{\it There exists $\epsilon_x > 0$ and a locally uniformizing
neighborhood $Q(x) \subset X$ of $x$ whose closure is contained in $\cU$, such that
\begin{equation} \label{eq:localorbifold}
\Theta^x \,:\; \{ t\in\R^l \, | \, |t| < \epsilon_x \} \times Q(x) \;\to\; \Q^+, \qquad (t,y) \;\mapsto\;\Lambda_t^x \bigl( S(y) \bigr) = \ti\L^x(\ti S^x(t,y))
\end{equation}
is a tame ep$^+$-subgroupoid, and for $(t,y)\in\supp \Theta^x=\{(t,y)\,|\, \Theta^x(t,y)>0 \}\subset\R^l\times X$ the reduced linearizations $\rT^R_{(\tilde{S}^x,\tilde{\Lambda}^x)}(t,y):= \rT_{(\tilde{S}^x,\tilde{\Lambda}^x)}(t,y)|_{\rT_t\R^l\times \rT_{y}^R X}$ are surjective.
Moreover, if $x \in \cV$ then we may choose $Q(x) \subset \cV$ such that for all $(t,y) \in \supp \Theta^x$ we have surjections\footnote{
As before, this is shorthand for surjectivity on each reduced tangent space 
$\ker \rD_{(t,y)}( \tilde{S}^x + p_j)|_{\rT_t\R^l\times \rT_{y}^R X}$.
}
\begin{equation*}
\rD_{(t,y)}\ti e^x |_{N^x_{t,y}} \;:\; N^x_{t,y} := \ker \rT^R_{(\tilde{S}^x,\tilde{\Lambda}^x)}(t,y) \; \to\; \rT_{e(y)}M . 
\end{equation*}
In particular, the realization $|\supp \Theta^x|$ is a weighted branched orbifold and $\ti e^x$ induces a submersion $|\supp\Theta^x| \rightarrow M$ in the sense of Definition~\ref{def:submersion}. 
Moreover, for all $y \in S^{-1}(0)\cap U_x$ we have $(0,y)\in\supp \Theta^x$ so that the reduced linearizations
$\rT^R_{(\tilde{S}^x,\tilde{\Lambda}^x)}(0,y)$ and the restriction to their kernel $\rD_{(0,y)}\ti e^x |_{N^x_{0,y}}$ are surjective. These properties persist for $y\in S^{-1}(0)$ with $|y|\in|Q(x)|$.}

\medskip
The structure of $\supp \Theta^x$ and surjectivity of linearizations $\rT^R_{(\tilde{S}^x,\tilde{\Lambda}^x)}$ follows from the local implicit function theorem \cite[Theorems~15.2,15.3]{HWZbook}. 
Then the kernels $N^x_{t,y} = \ker \rT^R_{(\tilde{S}^x,\tilde{\Lambda}^x)}(t,y)$ represent the reduced tangent spaces at $|(t,y)|$ to the weighted branched orbifold $|\supp\Theta^x|$. 
Surjectivity of $\rD_{(0,x)}\ti e^x |_{N^x_{0,x}}$ holds since $\rD_{(0,x)}\ti e^x$ is surjective by assumption, and the preimage of any given vector in $\rT_{e(x)}M$ can be adjusted by vectors in $\ker\rD_{(0,x)}\ti e^x$ to lie in $N^x_{0,x} = \ker \rT^R_{(\tilde{S}^x,\tilde{\Lambda}^x)}(0,x)$, because $\rT_{(\tilde{S}^x,\tilde{\Lambda}^x)}(t,y)|_{\ker\rD_{(0,x)}\ti e^x}$ is surjective by  \eqref{eq:transverseonkernel}. 
Then $\ti e^x$ restricts to a map $|\supp\Theta^x| \rightarrow M$ that is classically smooth on each (finite dimensional) branch of $\supp\Theta^x$, and thus surjectivity of $\rD_{(t,y)}\ti e^x |_{N^x_{t,y}}$ is an open condition along each branch. Since $\supp \Theta^x$ is locally compact -- in particular with finitely many branches near $x$ -- we can then choose $\epsilon_x$ and $Q(x)$ sufficiently small to guarantee that each $\rD_{(t,y)}\ti e^x |_{N^x_{t,y}}$ is surjective. This proves submersivitiy in the sense of Definition~\ref{def:submersion}.

\medskip
\noindent
{\bf From local to global stabilization:}
In this portion of the proof, we proceed almost verbatim to the corresponding portion of \cite[Thm.15.4]{HWZbook}, with extra considerations to deduce submersivity of \eqref{eq:bigevsubmersive}.
By assumption, $|S^{-1}(0)|$ is compact and $|e|: |\cX|\to M$ is continuous. Then $|S^{-1}(0)| \cap |e^{-1}(C)|$ is compact since $C:= \overline{\cup_{i \in I}(C_i)} \subset M$ is closed. We moreover have the identity $|S^{-1}(0) \cap e^{-1}(C)|= |S^{-1}(0)| \cap |e^{-1}(C)|$ since both sets are saturated. 
Thus we have an open covering $\bigl(|Q(x)|\bigr)_{x \in S^{-1}(0)\cap e^{-1}(C)}$ by the open subsets chosen above, and can pick finitely many points $x_1,\ldots,x_r  \in S^{-1}(0) \cap e^{-1}(C)$ to obtain a finite open cover $|S^{-1}(0) \cap e^{-1}(C)| \subset \bigcup_{i=1}^r |Q(x_i)|$. 
Then $|S^{-1}(0)| \less \bigcup_{i=1}^r |Q(x_i)|$ is compact, with open cover by $\bigl(|Q(x)|\bigr)_{x \in S^{-1}(0)}$, so we may pick further $x_{r+1},\ldots,x_k \in S^{-1}(0)$ to obtain the covers
\begin{align} \label{eq:cover}
&\textstyle
|S^{-1}(0)| \;\subset\; \bigcup_{i=1}^k |Q(x_i)|, 
 \qquad\qquad\qquad\qquad
|S^{-1}(0) \cap e^{-1}(C)| \;\subset\; \bigcup_{i=1}^r |Q(x_i)|, \\
&\textstyle \nonumber
S^{-1}(0) \;\subset\; \tilde{Q} := \pi^{-1}\bigl(  \bigcup_{i=1}^k |Q(x_i)| \bigr) \;\subset\; \cU.
\end{align}
For each $x=x_i$ we constructed above a family of sc$^+$-multisections $\bigl(\Lambda_t^{x_i}: \cW\to\Q^+\bigr)_{t\in\R^{l_{x_i}}}$. These are summed up, using \cite[Def.13.11]{HWZbook}, to a sc$^+$-multisection
\begin{align*}
\tilde{\Lambda} \,:\; \mathbb{R}^{\tilde{l}} \times \cW \;\to\; \mathbb{Q}^+, \qquad
\bigl(t=(t_1,\ldots,t_k) \,,\, w \bigr) \;\mapsto\; \Lambda_t(w) \,:=\; \big ( \Lambda^{x_1}_{t_1} \oplus \cdots \oplus \Lambda^{x_k}_{t_k} \big )(w)
\end{align*}
for $\tilde{l} := l_{x_1} + \cdots + l_{x_k}$. 
Here each $\Lambda_t : \cW \rightarrow \mathbb{Q}^+$ for $t\in\R^{\ti l}$ is a structurable sc$^+$-multisection by \cite[Prop.13.3]{HWZbook}. 
We view the multisection $\tilde{\Lambda}$ as global perturbation of a sc-Fredholm section functor $\ti S$ of a bundle $\ti P$, 
\begin{align*}
\tilde{S}\,:\; \mathbb{R}^{\ti l} \times X &\rightarrow \mathbb{R}^{\ti l} \times \cW =: \ti\cW
\qquad & \tilde{P} \,: \; \mathbb{R}^{\ti l} \times \cW &\rightarrow \mathbb{R}^{\ti l} \times \cX \\
(t,y) &\mapsto (t,S(y))
& (t,w) &\mapsto (t,P(w)) , 
\end{align*}
and claim that $e:\cX\to M$ induces a submersion on its perturbed solution set in the following sense. 

\medskip
\noindent
\ul{Global stabilization properties:} 
{\it There exists $\epsilon_0 > 0$ such that for every $0<\epsilon < \epsilon_0$  
\begin{equation} \label{eq:stabilizedorbifold}
\ti\Theta \,:\; \{ t\in\R^{\ti l} \, | \, |t| < \epsilon \} \times \cX \;\to\; \Q^+, 
\qquad (t,y) \;\mapsto\; \Lambda_t \bigl( S(y) \bigr) = \ti\L(\ti S(t,y))
\end{equation}
is a tame ep$^+$-subgroupoid with surjective reduced linearizations $\rT^R_{(\tilde{S},\tilde{\Lambda})}(t,y)$ for all $(t,y)\in\supp\ti \Theta$. In particular, the realization $|\supp \ti\Theta|$ is a weighted branched orbifold. Moreover, there is a neighbourhood $\cV'\subset\cX$ of $S^{-1}(0)\cap e^{-1}(C)$ such that  
\begin{equation} \label{eq:bigevsubmersive}
\ti e |_{\supp\ti\Theta} \,:\; \supp\ti\Theta  \;\to\; M , \qquad (t,y) \mapsto e(y)
\end{equation}
satisfies $(\ti e |_{\supp\ti\Theta})^{-1}(C)\subset \R^{\ti l}\times\cV'$, and its restriction to 
$\supp\ti\Theta \cap (\R^{\ti l}\times\cV')$ is classically smooth and submersive as in Definition~\ref{def:submersion}.}

\medskip
Note that the auxiliary norm $N$ on $\cW$ pulls back to an auxiliary norm $\ti N$ on $\ti\cW$, and compactness of $\ti S$ is controlled in the sense that for any compact subset $K\subset \R^{\ti l}$ we have compactness of
\begin{equation} \label{eq:controlcompfamily}
\bigl| \{(t,x)\in K \times \cU \,|\, \ti N(\ti S(t,x) ) \leq 1\} \bigr|
\;=\; K \times \bigl| \{(x \in \cU \,|\, N(S(x) ) \leq 1\} \bigr| \;\subset\; \R^{\ti l}\times |\cX|. 
\end{equation}
Next, the restriction of $\ti\Lambda$ to each $\R^{l_{x_i}}\times X \hookrightarrow \R^{\ti l}\times X$ is the local perturbation $\ti\Lambda^{x_i}$ of $\ti S^{x_i}$, since we identify $\R^{l_{x_i}}\cong\{(t_1,\ldots,t_k)\in\R^{\ti l} \,|\, t_j=0 \;\forall j\neq i\}$ and each $\Lambda^{x_j}_0$ is trivial. 
In particular, $\Lambda_0$ is the trivial multisection, with $N(\Lambda_0)=0$. 
Moreover, we have an estimate $N(\Lambda_t)\leq c|t|$ that results from the linear estimates on each $\Lambda^{x_i}_t$. Now for $\epsilon_0\leq \frac 1c$ we can deduce compactness of the stabilized solution set as closed subset of \eqref{eq:controlcompfamily}, 
\begin{equation} \label{eq:compactclosure}
 \ti Z  \,:= \; \bigl| \bigl\{ (t,x) \in \R^{\ti l} \times X \,\big| \, |t| \leq \epsilon_0,  \Lambda_t(S(x)) > 0 \bigr\} \bigr| .
\end{equation}
The next step is to argue that \eqref{eq:compactclosure} is smooth in a neighbourhood of $\ti Z \cap (\{0\} \times |\cX|) = \{0\} \times |S^{-1}(0)|$. Recall here that $\ti Q=\pi^{-1}(\bigcup_{i=1}^k |Q(x_i)|)\subset X$ is an open neighbourhood of $S^{-1}(0)$. 
So for any $x\in \tilde{Q}$ we can use the local properties of some $\ti\L^{x_i}$ with $|x|\in|Q(x_i)|$ to deduce surjectivity of $\rT^R_{(\tilde{S},\tilde{\Lambda})}(0,x)$. 
Then the local implicit function theorems \cite[Thms~15.2,15.3, Rmk.15.2]{HWZbook} yield an open neighbourhood $U(0,x) = \{ |t| < \epsilon'_x \} \times U(x)\subset\R^{\ti l} \times X$ of $(0,x)$ for some $0<\epsilon'_x<\epsilon_0$, and hence a saturated neighbourhood $\ti U(0,x) :=  \{ |t| < \epsilon'_x \} \times \pi^{-1}(|U(x)|) \subset\R^{\ti l} \times\cX$ such that $\ti\Theta|_{\ti U(0,x)} =\tilde{\Lambda} \circ \tilde{S}|_{\ti U(0,x)}$ is a tame branched ep$^+$-subgroupoid of $\tilde{U}(0,x)$. As a consequence, the orbit space of the support $\bigl|\supp \ti\Theta|_{\tilde{U}(0,x)}\bigr|$ is a weighted branched orbifold with boundary and corners.

For $x \in S^{-1}(0) \less e^{-1}(C)$ we can moreover choose $U(x) \cap e^{-1}(C) = \emptyset$, since $|e^{-1}(C)|\subset|\cX|$ is closed. 
For $x \in S^{-1}(0) \cap e^{-1}(C)\subset\cV$ the covering \eqref{eq:cover} guarantees $|x|\in|Q(x_i)|$ for some $1\leq i \leq r$ with $Q(x_i)\subset\cV$ and we choose $U(x)\subset Q(x_i)$. 
This guarantees that the restriction of $\tilde{e} : \mathbb{R}^{\tilde{l}} \times X \to M, (t,y) \mapsto e(y)$ to $\ti U(0,x)$ is sc$^\infty$, and surjectivity of $\rD_{(0,x)}\ti e^{x_i} |_{\ker \rT^R_{(\tilde{S}^{x_i},\tilde{\Lambda}^{x_i})}(0,x)}$ implies surjectivity of $\rD_{(0,x)}\tilde{e}|_{N_{0,x}} : N_{0,x} \rightarrow \rT_{e(x)}M$ on $N_{0,x} := \ker \rT^R_{(\tilde{S},\tilde{\Lambda})}(0,x)$. 
Here $N_{0,x}$ represents the reduced tangent space at $|(0,x)|$ to the weighted branched orbifold $|\supp \ti\Theta|_{\tilde{U}(0,x)} |$. 
Now $\ti e|_{\supp\ti\Theta \cap \tilde{U}(0,x)} : \supp\ti\Theta|_{\tilde{U}(0,x)} \rightarrow M$ is classically smooth since it is a restriction of an sc$^\infty$ map to finite dimensions, and we have shown it to be submersive at $(0,x)$. Hence, by openness of submersivity along each corner stratum, and local compactness of $\supp\ti\Theta|_{\tilde{U}(0,x)}\subset\ti Z$ it follows that $\ti U(0,x)\subset\R^{\ti l}\times\cV$ can be chosen sufficiently small to ensure that $\ti e|_{\supp\ti\Theta \cap \tilde{U}(0,x)}$ is submersive as in Definition~\ref{def:submersion}.
 
Now compactness of $|S^{-1}(0) \cap e^{-1}(C)|$ and $|S^{-1}(0)|$ again allows us to find finite covers
\begin{align*}
&\textstyle
|S^{-1}(0)| \;\subset\; \bigcup_{i=1}^{k'} |U(x'_i)|, 
 \qquad\qquad\qquad\qquad
|S^{-1}(0) \cap e^{-1}(C)| \;\subset\; \bigcup_{i=1}^{r'} |U(x'_i)|
\end{align*}
with $x'_i\in S^{-1}(0) \cap e^{-1}(C)$ for $i=1,\ldots,r'$ and $U(x'_i) \cap e^{-1}(C) = \emptyset$ for $r'<i\leq k'$. 
Then we have $\epsilon:=\min\{\epsilon'_{x'_1}, \ldots \epsilon'_{x'_{k'}}\}  > 0$, an open cover $S^{-1}(0) \subset \cA:= \pi^{-1}\bigl(\bigcup_{i=1}^{k'} |U(x'_i)|\bigr)$, and the functor $\{t \in \R^{\ti l} \, | \, |t| < \epsilon\} \times \cA \rightarrow \mathbb{Q}^+, (t,y) \mapsto \Lambda_t(S(y))$ is a tame branched ep$^+$-subgroupoid, since it is the restriction of $\ti\Theta=\tilde{\Lambda} \circ \tilde{S}$ to an open subset of $\bigcup_{i=1}^{k'}\tilde{U}(0,x'_i)$. 
Moreover, we claim that for a possibly smaller $0<\epsilon <\epsilon_0$ we have
\begin{equation}\label{eq:claimA}
(t,y) \in \{|t| < \epsilon\} \times X , \quad \ti\Theta(t,y) > 0 \qquad\Longrightarrow\qquad  y \in \cA . 
\end{equation}
By contradiction, consider a sequence $\R^{\ti l} \ni t_n\to 0$, $y_n\in X$ with $\ti\Theta(t_n,y_n)>0$ but $y_n\in X\less \cA$. Then compactness of \eqref{eq:compactclosure} guarantees a convergent subsequence $|(t_n,y_n)|\to |(0,y_\infty)|\in\ti Z$, and since $\ti Z\cap \{0\}\times |\cX| = \{0\}\times |\supp \Lambda_0\circ S | = \{0\}\times |S^{-1}(0)|$ this contradicts the fact that $|y_n|\in|\cX|\less|\cA|$, where $|\cA|=\bigcup_{i=1}^{k'} |U(x'_i)|\subset|\cX|$ is an open neighbourhood of $|S^{-1}(0)|$. 
Thus we have shown \eqref{eq:claimA} and can deduce that 
$\tilde{\Theta} =\ti\L\circ \ti S : \{t \in \R^{\ti l} \, | \, |t| < \epsilon \} \times \cX \to \Q^+$
is a tame branched ep$^+$-subgroupoid with $\supp \ti \Theta \subset \R^{\ti l} \times \cA$, and thus
$\bigl|\supp \ti\Theta\bigr|\subset \R^{\ti l} \times\bigcup_{i=1}^{k'} |U(x'_i)|$ is a weighted branched orbifold with boundary and corners, as claimed. 
 
Moreover, from the properties of $\ti e|_{\supp\ti\Theta \cap \tilde{U}(0,x'_i)}$ for $i=1,\ldots,r'$ we know that the restriction of $\tilde{e}$ to $\supp \tilde{\Theta} \cap (\R^{\ti l}\times\cV')$ for $\cV':= \pi^{-1}\bigl(\bigcup_{i=1}^{r'} U(x'_i) \bigr) \subset \cV$ is classically smooth and submersive. 
Here we have $e^{-1}(C) \cap \cA\subset \cV'$ since $U(x_i)$ for $i>r'$ was chosen disjoint from $e^{-1}(C)$, and hence we have
$$
\bigl(\tilde{e}|_{\supp \tilde{\Theta}}\bigr)^{-1}(C) \;=\; \supp \tilde{\Theta} \cap \bigl( \R^{\ti l} \times e^{-1}(C)\bigr) \;\subset\; \R^{\ti l} \times \bigl( e^{-1}(C) \cap \cA \bigr) 
\;\subset\; \R^{\ti l} \times \cV',
$$
and thus $\tilde{e}|_{ \supp \tilde{\Theta}} : \supp \tilde{\Theta} \rightarrow M$ is classically smooth and submersive (in the sense of Definition~\ref{def:submersion}) in the open neighborhood $\supp \tilde{\Theta} \cap (\R^{\ti l} \times \cV')$ of $\bigl(\tilde{e}|_{\supp \tilde{\Theta}}\bigr)^{-1}(C_i) \subset \supp \tilde{\Theta}$ for all $i \in I$.

\medskip
\noindent
{\bf Global transversality from regular values:} As we continue to follow the proof of \cite[Thm.15.4]{HWZbook}, we replace each application of the Sard theorem by countably many Sard arguments to obtain general position to the countably many submanifolds $C_i \subset M$ for $i\in I$.
For that purpose we will consider various restrictions of the projection
$$
\supp \tilde{\Theta} = \bigl\{ (t,y) \in \R^{\ti l} \times \cX \, \big | \, |t| < \epsilon,\, \Lambda_t(S(y)) > 0 \bigr\} \;\to\; \R^{\ti l}, \qquad (t,y) \mapsto t . 
$$
The global properties of $\ti\Theta$ imply that every $(t_0,y_0) \in \supp \tilde{\Theta}$ has a saturated open neighborhood 
$\ti U(t_0,y_0) = \{ t\in\R^{\ti l} \, | \, |t-t_0| < \delta\} \times \pi^{-1}(|U(y_0)|) \subset \R^{\ti l}\times \cX$
satisfying the following:
\begin{itemlist}
\item 
$U(y_0)\subset X$ admits the natural action of the isotropy group $G_{y_0}$; see \cite[Thm.7.1]{HWZbook}, satisfies the properness property \cite[Def.7.17]{HWZbook}, 
and has $d_{\cX}(y_0)$ local faces $\cF^{y_0}_1,\ldots,\cF^{y_0}_{d_{\cX}(y_0)}$ which contain $y_0$; see \cite[Def.2.21, Prop.2.14]{HWZbook}.
\item 
The branched ep$^+$-subgroupoid $\supp\ti\Theta\cap \tilde{U}(t_0,y_0)$ has a local branching structure
$$
\ti\Theta(t,y) \;=\; \Lambda_t(S(y)) \;=\; \tfrac{1}{|J|} \cdot \bigl| \bigl\{j \in J \, | \, (t,y) \in M^{t_0,y_0}_j \bigr\} \bigr|,
$$
given by finitely many properly embedded submanifolds with boundary and corners $M^{t_0,y_0}_j\subset \tilde{U}(t_0,y_0)$, which intersect any intersection of local faces in a manifold with boundary and corners. 
\item 
On each branch $M^{t_0,y_0}_j$, the reduced linearizations $\rT^R_{(\tilde{S},\tilde{\Lambda})}(t,y)$ are surjective for all $(t,y) \in M^{t_0,y_0}_j$, and the restriction of $\tilde{e}|_{\supp\tilde{\Theta}}$ is a submersion 
$M^{t_0,y_0}_j \cap (\mathbb{R}^{\tilde{l}} \times \cV') \rightarrow M$ in general position to the boundary in the sense of Definition~\ref{def:submersion}. 
That is, $\rD_{(t,y)}\tilde{e}|_{N_{t,y}} : N_{t,y} \rightarrow \rT_{e(y)}M$ is surjective on $N_{t,y} := \ker \rT^R_{(\tilde{S},\tilde{\Lambda})}(t,y)$ for all $(t,y) \in M^{t_0,y_0}_j \cap (\R^{\ti l} \times \cV')$.
\end{itemlist}

\noindent
There is a countable cover $\supp\tilde{\Theta}\subset \bigcup_{\beta\in\Z} \tilde{U}(t_{\beta},y_{\beta})$ indexed by $(t_{\beta},y_{\beta})_{\beta \in\Z} \subset \supp \tilde{\Theta}$, since $\R^{\ti l} \times X$ -- and hence its subspace $\supp\ti\Theta$ -- is second-countable, and every open cover of a second-countable space has a countable subcover. 
Moreover, for any given $\beta\in\Z$ there are finitely many choices $\ti\cF_K := \{ |t-t_0| < \delta\}  \times\bigcap_{k\in K}\cF^{y_{\beta}}_k \subset \tilde{U}(t_{\beta},y_{\beta})$ of intersections of finitely many local faces $K\subset\{1,\ldots,d_\cX(y_\beta)\}$, with $\ti\cF_\emptyset:= \tilde{U}(t_{\beta},y_{\beta})$. 
Finally, for each $\beta\in\Z$ and intersection of faces $\ti\cF_K$, there are finitely many smooth manifolds $\ti\cF_K \cap M^{t_\beta,y_\beta}_j$ indexed by $j \in J_\beta$. 
For each of these countably many choices, Sard's theorem asserts that
$\ti\cF_K \cap M^{t_\beta,y_\beta}_j \to \R^{\ti l}$, $(t,y) \mapsto t$
has an open and dense subset $R^\beta_{K,j} \subset\R^{\ti l}$ of regular values. 
Then, since $\R^{\ti l}$ is a Baire space, the set of common regular values $\cR_0:= \bigcap_{\beta\in\Z}\bigcap_{K,j} R^\beta_{K,j} \subset \mathbb{R}^{\tilde{l}}$ is still dense. 
For any $t_0 \in \cR_0$, the sc$^+$-multisection $\Lambda_{t_0} : \cW \rightarrow \mathbb{Q}^+$ is in general position by the usual linear algebra for each restriction of the linearized operators to intersections of faces: 
Consider $(t_0,x_0)\in\ti\cF_K \cap M^{t_\beta,y_\beta}_j \subset \supp\ti\Theta$ and a local section $S+p^j$ in the representation of $\ti\Theta=\ti\Lambda\circ\ti S$ with $M^{t_\beta,y_\beta}_j \subset (S+p^j)^{-1}(0)$. The surjective differential along this intersection of faces can be written as $\rD_{t_0,x_0}(S+p^j)|_{\ti\cF_K} = D \oplus L$, where $L$ is a bounded operator (arising from differentiating $p^j$ in the direction of $\R^{\ti l}$) and $D$ is the reduced linearization -- on the intersection of faces $\cF_K:= \bigcap_{k\in K}\cF^{y_{\beta}}_k \subset U(y_{\beta})\subset X$ -- of the section $S+p^j(t_0,\cdot)$ that is a part of the representation of $\Lambda_{t_0}\circ S$. Then regularity of $t_0$ implies surjectivity of the projection $\Pi:\ker(D\oplus L)\to \R^{\ti l}$, which in turn is equivalent to surjectivity of $D$; see e.g.\ \cite[Lemma~A.3.6]{MS}. 

Moreover, each $\Lambda_t$ for $|t| < \epsilon$ is $(N,\cU)$-admissible, thus any sufficiently small regular $t_0\in\cR_0$ yields an admissible  sc$^+$-multisection $\lambda := \Lambda_{t_0}$ in general position as in \cite[Thm.15.4]{HWZbook}.
To prove our theorem, we have to moreover choose $t_0 \in \cR_0$ so that the restriction $e|_{Z^{\lambda}} : Z^{\lambda} \rightarrow M$ to the solution set $Z^{\lambda} = |\supp \lambda \circ S |$ is in general position to $C_i \subset M$ for all $i \in I$.
For that purpose we consider the countably many projections 
\begin{equation}\label{eq:eCprojections}
 \tilde{e}^{-1}(C_i) \cap \ti\cF_K \cap M^{t_\beta,y_\beta}_j \;\to\; \mathbb{R}^{\tilde{l}}, 
 \qquad (t,x) \mapsto t 
\end{equation}
for any $i\in I$, index $\beta\in\Z$ of the countable cover, intersection of local faces $\ti\cF_K$, and smooth branch $M^{t_\beta,y_\beta}_j\subset \supp \tilde{\Theta}\cap \ti U(t_\beta,y_\beta)$. 
Here we have $\tilde{e}^{-1}(C_i) \cap M^{t_\beta,y_\beta}_j \subset (\tilde{e}|_{\supp\ti\Theta})^{-1}(C)$, so that the restriction 
$\tilde{e}|_{\ti\cF_K \cap M^{t_\beta,y_\beta}_j} : \ti\cF_K\cap M^{t_\beta,y_\beta}_j \to M$ 
is smooth and submersive in a neighborhood of $\tilde{e}^{-1}(C_i)$. 
In particular, it is transverse to $C_i$ so that there is a natural smooth structure on $\tilde{e}^{-1}(C_i) \cap \ti\cF_K \cap M^{t_\beta,y_\beta}_j$. 
Thus we can apply the Sard theorem to each \eqref{eq:eCprojections} to find open and dense subsets $T^{i,\beta}_{K,j} \subset\R^{\ti l}$ of regular values, and a dense set of common regular values $\cT_0:= \bigcap_{\beta\in\Z}\bigcap_{K,j} R^\beta_{K,j}\cap \bigcap_i T^{i,\beta}_{K,j} \subset \mathbb{R}^{\tilde{l}}$. 
Note that $\cT_0\subset \cR_0$, so sufficiently small $t_0\in\cT_0$ yield admissible  sc$^+$-multisections $\lambda:=\Lambda_{t_0}$ in general position. 
Moreover, general position of $e|_{Z^{\lambda}} : Z^{\lambda} \rightarrow M$ to $C_i$ at 
$x\in Z^\lambda \cap e^{-1}(C_i)$ means that the linearizations of $e|_{\cF_K\cap Z^\lambda}$ map onto $\qu{\rT_{e(x)} M}{\rT_{e(x)} C_i}$ for each intersection of local faces $\cF_K \subset U(y_{\beta})\subset X$ that contains $x$. Here the tangent spaces of $\cF_K\cap Z^{\lambda}$ at $x$ are given by those of $\ti\cF_K \cap M^{t_\beta,y_\beta}_j  \cap (\{t_0\}\times X)$ for each branch with $(t_0,x) \in M^{t_\beta,y_\beta}_j \subset\supp\ti\Theta$, so we need to ensure surjectivity of $\rD_{(t_0,x)}\ti e : \ker\Pi \to \qu{\rT_{e(x)} M}{\rT_{e(x)} C_i}$ on the kernel of the projection 
$\Pi: \rT_{(t_0,x)} \bigl( \ti\cF_K \cap M^{t_\beta,y_\beta}_j \bigr) \to \R^{\ti l}$. 
Here $\rD_{(t_0,x)}\ti e : \rT_{(t_0,x)} \bigl( \ti\cF_K \cap M^{t_\beta,y_\beta}_j \bigr) \to \rT_{e(x)} M$ is surjective (since $\ti e|_{\supp\ti\Theta}$ is submersive), 
and regularity $t_0\in T^{i,\beta}_{K,j}$ means that we have $\Pi \, (\rD_{(t_0,x)}\ti e)^{-1}(\rT_{e(x)} C_i) = \R^{\ti l}$, so for any $Y\in \rT_{e(x)} M$ we find $(T,X)\in \rT_{(t_0,x)} \bigl( \ti\cF_K \cap M^{t_\beta,y_\beta}_j \bigr)$ with $\rD_{(t_0,x)}\ti e (T,X) = Y$ and $(T,X')\in (\rD_{(t_0,x)}\ti e)^{-1}(\rT_{e(x)} C_i)$, so that $(0,X-X')\in\ker\Pi$ proves the required surjectivity $\rD_{(t_0,x)}\ti e (0,X-X') = Y - \rD_{(t_0,x)}\ti e (T,X') = [Y]\in \qu{\rT_{e(x)} M}{\rT_{e(x)} C_i}$.
Thus this choice of sufficiently small $t_0\in\cT_0$ also guarantees general position of $e|_{Z^{\lambda}}$ to each of the countably many submanifolds $C_i$, which finishes the proof of the theorem when no boundary values are prescribed.

\medskip
\noindent
{\bf Regular extension:}
To prove the last paragraph of the theorem we consider a given $(\frac 1\alpha N,\cU)$-admissible structurable sc$^+$-multisection $\lambda^{\partial} : P^{-1}(\partial \cX) \rightarrow \mathbb{Q}^+$ that is in general position over the boundary, and with $e|_{Z^\partial}: Z^\partial = \supp \lambda^{\partial} \circ S|_{\partial \cX} \to M$ in general position to finitely many submanifolds $C_i$.
Then we will adjust the above construction of $\lambda : \cW \rightarrow \mathbb{Q}^+$ to also satisfy $\lambda|_{P^{-1}(\partial \cX)} = \lambda^{\partial}$, by following the proof of the transversal extension theorem over ep-groupoids \cite[Thm.15.5]{HWZbook}.

Since $\lambda^\partial$ is supported in $\cU\cap\partial\cX$ with $N(\lambda^\partial)(x)<\alpha$ for all $x\in\partial\cX$ we can find a continuous functor $h : \cX \rightarrow [0,1)$ supported in $\cU$ with $N(\lambda^{\partial})(x) < h(x) < \frac 12 N(\lambda^{\partial})(x) + \frac 12$ for all $x \in \partial \cX$. Then \cite[Thm.14.2]{HWZbook} 
yields a sc$^+$-multisection $\Lambda' : \cW \rightarrow \mathbb{Q}^+$ with $\Lambda'|_{P^{-1}(\partial \cX)} = \lambda^{\partial}$, domain support in $\cU$, and $N(\Lambda')(x) \leq h(x) \leq \frac{\alpha+1}2$ for all $x \in \cX$.
This guarantees compactness of $|\supp\Lambda'\circ S|\subset|\cX|$ and regularity of $|\supp\Lambda'\circ S| \cap |\partial\cX| = |\supp\lambda^\partial\circ S|_{\partial\cX}|$. 
To obtain regularity in the interior we construct $\lambda = \Lambda'\oplus\Lambda_t$ by the above arguments with $S^{-1}(0)$ replaced by $\cS':=\supp\Lambda'\circ S \subset \cX$, noting that $|\cS'|\subset|\cX|$ is also compact. 
To achieve general position to the $C_i$ we need further adjustments. 

\medskip
\noindent
 {\bf Local constructions relative to boundary values:}
For interior points $x\in \cS'\cap\partial_0 \cX$ we construct $\ti\Lambda^x:\R^l \times\cW \to \Q^+$ with domain support in the interior $\R^l \times (\partial_0\cX \cap \cU)$ to cover the cokernels of $\rT^R_{(\tilde{S}^x,\tilde\Lambda')}$ for the stabilized multisection 
$\ti\Lambda':\R^l\times\cW\to\Q^+, (t,w) \mapsto \Lambda'(w)$. 
For $x\in \cS'\cap\partial\cX$ we need no stabilization by a $\R^l$ factor (i.e.\ take $l=0$) due to the general position of $\lambda^\partial$ at $x$. However, we only obtain general position to the $C_i$, rather than submersivity in the following claim. 

\medskip
\noindent
\ul{Local properties relative to boundary:} 
{\it For each $x\in\cS'$ there exists $l_x\in\N_0$ -- with $l_x=0$ for $x\in\cS'\cap\partial\cX$ -- and a locally uniformizing neighborhood $Q(x) \subset X$ of $x$ whose closure is contained in $\cU$, such that for some $\epsilon_x > 0$ we have a tame ep$^+$-subgroupoid 
$\Theta^x : \{ t\in\R^l \, | \, |t| < \epsilon_x \} \times Q(x) \to \Q^+$, $(t,y) \mapsto
\bigl( \Lambda' \oplus \Lambda_t^x\bigr) ( S(y))$
with surjective reduced linearizations, 
and thus a weighted branched orbifold $|\supp \Theta^x|$. 
Moreover, if $x \in \cS'\cap\cV$ then $\ti e^x$ induces a smooth map 
$|\supp\Theta^x| \rightarrow M$, which is in general position to $C_i$ for each $i\in I$. 
}

\medskip
The structure of $\Theta^x$ is established in \cite[Thm.15.5.]{HWZbook}, and the general position to each $C_i$ for $x\in\partial_0\cX$ follows from submersivity. 
To achieve general position to the $C_i$ for $x\in\partial\cX$, recall that $C= \overline{\cup_{i \in I}(C_i)} \subset M$ is closed, so for $x\notin e^{-1}(C)$ we can choose $Q(x)$ disjoint from $e^{-1}(C)$ so that general position to the $C_i\subset C$ is automatic.  For $x \in e^{-1}(C)\subset\cV$ we have $e : \supp\Theta^x \cap \partial\cX = \supp \lambda^\partial \circ S|_{\partial\cX} \to M$ in general position to each $C_i$ by assumption on $\lambda^\partial$. Moreover, we choose $Q(x)\subset\cV$ so that $e : Q(x) \cap \supp\Theta^x \to M$ is smooth, and thus general position to each $C_i$ extends to a neighbourhood $Q_i\subset\cX$ of $x$. Then $Q':=\bigcap_{i\in I} Q_i$ is a neighbourhood of $x$ since $I$ is finite, and we can replace $Q(x)$ by a uniformizing neighbourhood in $Q'$ to achieve general position to all $C_i$. 

\medskip
\noindent
{\bf From local to global relative to boundary:}
This portion of the proof is started by picking a finite cover $|\cS'|\cap|\partial\cX| = |\supp \lambda^\partial\circ S|_{\partial\cX}| = \bigcup_{i=-k_\partial}^0 |Q(x_i)| \subset |\cX|$ by the above neighbourhoods for $x_i\in \cS' \cap \partial\cX$. Next we cover $|\cS'|\setminus \bigcup_{i=-k_\partial}^0 |Q(x_i)| \subset \bigcup_{i=1}^k |Q(x_i)|$ with neighbourhoods of  interior points $x_i\in\cS'\cap \partial_0\cX$ whose associated multisections $\Lambda^{x_i}$ are supported in the interior, $\text{dom-supp}\,\Lambda^{x_i}\subset \R^{l_x}\cap \partial_0\cX$. 
Then we define $\tilde{\Lambda} : \R^{\ti l} \times \cW \to\mathbb{Q}^+$ by $\ti\Lambda (t, w) := \Lambda_t(w) :=  \big (\Lambda' \oplus \Lambda^{x_1}_{t_1} \oplus \cdots \oplus \Lambda^{x_k}_{t_k} \big )(w)$. This multisection is constructed so that $\Lambda_0=\Lambda'$ and $\Lambda_t|_{P^{-1}(\partial\cX)}=\lambda^\partial$ for any $t\in\R^{\ti l}$. Moreover, the estimate $N(\Lambda_t)\leq N(\Lambda') + c|t| \leq \frac{1+\alpha}{2} + c|t|$ allows us to guarantee admissibility $N(\Lambda_t)\leq 1$ by choosing $|t|\leq \frac{1-\alpha}{2c}$. Then compactness of $\ti Z$ in \eqref{eq:compactclosure} follows as above, and its smoothness is established using a covering
$|S^{-1}(0)| \subset \bigcup_{i=-k_\partial}^{k'} |U(x'_i)|$
where $|U(x'_i)|$ for $i\leq 0$ arise from $x'_i\in\cS'\cap\partial\cX$ and cover a neighbourhood of $|\partial\cX|$.  
Moreover, $U(x'_i)\subset \R^{\ti l}\times Q(x'_i)$ can be chosen as in the prior proof of the local properties such that $\ti e|_{\supp\Theta} : U(x'_i) \rightarrow M$ is in general position to $C_i$ for each $i\in I$. 
This establishes the following. 

\medskip
\noindent
\ul{Global stabilization properties with fixed boundary values:} 
{\it There exists $\epsilon_0 > 0$ such that $\ti\Theta:= \ti\Lambda \circ \ti S : \{ |t| < \epsilon \} \times \cX \to \Q^+$ is a tame ep$^+$-subgroupoid with surjective reduced linearizations for every $0<\epsilon < \epsilon_0$. 
In particular, $|\supp\ti\Theta|$ is a weighted branched orbifold. 
Moreover, there is a neighbourhood $\cV'\subset\cX$ of $S^{-1}(0)\cap e^{-1}(C)$ such that  
$\ti e |_{\supp\ti\Theta} : \supp\ti\Theta \to M$ satisfies $(\ti e |_{\supp\ti\Theta})^{-1}(C)\subset \R^{\ti l}\times\cV'$, and its restriction to $\supp\ti\Theta \cap (\R^{\ti l}\times\cV')$ is classically smooth and in general position to each $C_i$.
}

\medskip
\noindent
{\bf Global transversality relative to boundary:}
In this final step we use the fact that $\Lambda_t$ is $(N,\cU)$-admissible for $|t|\leq \frac{1-\alpha}{2c}$ and choose a common regular value of countably many projections as before. 
The only difference to the proof above is that the restriction of $\tilde{e}|_{\supp\tilde{\Theta}}$ to a branch $M^{t_0,y_0}_j \cap (\R^{\ti l}\times \cV') \rightarrow M$ is not necessarily submersive but still in general position to each of the $C_i$, that is 
$\rD_{(t,y)}\tilde{e}|_{N_{t,y}} : N_{t,y} \rightarrow \qu{\rT_{e(y)}M}{\rT_{e(y)}C_i}$ is surjective for each $i\in I$.
When considering the projections \eqref{eq:eCprojections}, this suffices to obtain smooth structures on $\tilde{e}^{-1}(C_i) \cap \ti\cF_K \cap M^{t_\beta,y_\beta}_j$ for each branch and intersection of faces $\ti\cF_K$. 
Then general position of $e|_{Z^{\lambda}} : Z^{\lambda} \rightarrow M$ to $C_i$ at $x\in Z^\lambda \cap e^{-1}(C_i)$ for $\lambda =\Lambda_{t_0}$ with a regular value $t_0\in\R^l$ again requires surjectivity of $\rD_{(t_0,x)}\ti e : \ker\Pi \to \qu{\rT_{e(x)} M}{\rT_{e(x)} C_i}$ on the kernel of the projection 
$\Pi: \rT_{(t_0,x)} \bigl( \ti\cF_K \cap M^{t_\beta,y_\beta}_j \bigr) \to \R^{\ti l}$.
To see that $[Y]\in\qu{\rT_{e(x)} M}{\rT_{e(x)} C_i}$ is in the image we use the above surjectivity of $\rD_{(t_0,x)}\tilde{e}|_{N_{t_0,x}}$ to find $(T,X)\in\rT_{(t_0,x)} \bigl( \ti\cF_K \cap M^{t_\beta,y_\beta}_j \bigr)$ with $\rD_{(t_0,x)}\tilde{e}(T,X) \in [Y]$. Then regularity of $t_0$ 
yields $(T,X')\in (\rD_{(t_0,x)}\ti e)^{-1}(\rT_{e(x)} C_i)$, so that $(0,X-X')\in\ker\Pi$ solves 
$[\rD_{(t_0,x)}\ti e (0,X-X') ] = [Y - \rD_{(t_0,x)}\ti e (T,X')] = [Y] \in \qu{\rT_{e(x)} M}{\rT_{e(x)} C_i}$.
This finishes the proof with prescribed boundary values.
\end{proof}

\bibliographystyle{alpha}

\end{document}